\documentclass[11pt,reqno]{amsart}
\usepackage{amsmath,amsfonts,amssymb,wasysym,mathrsfs}
\usepackage[latin1]{inputenc}
\usepackage{shapepar}
\usepackage{graphicx}
\usepackage{cancel}
\usepackage{float}
\usepackage{enumerate}

\usepackage{relsize}

\usepackage{hyperref}

\usepackage{caption}
\usepackage{xcolor}
\usepackage[left=2cm,right=2cm,top=3cm,bottom=3cm]{geometry}

\usepackage{ifpdf} 
\usepackage{color}
\newcommand{\HH}{{\mathbb M}} 
\newcommand{\R}{{\mathbb R}}

\newcommand{\N}{{\mathbb N}}

\newcommand{\be}[1]{\begin{equation}\label{#1}}
\newcommand{\ee}{\end{equation}}

\newtheorem{theorem}{Theorem}[section]
\newtheorem{lemma}[theorem]{Lemma}
\newtheorem{corollary}[theorem]{Corollary}
\newtheorem{proposition}[theorem]{Proposition}
\newtheorem{remark}[theorem]{Remark}
\newtheorem{definition}[theorem]{Definition}

\newcommand{\dint}{\displaystyle\int}
\numberwithin{equation}{section}

\definecolor{darkblue}{rgb}{0.05, .05, .65}
\definecolor{darkgreen}{rgb}{0.1, .65, .1}
\definecolor{darkred}{rgb}{0.8,0,0}

\newcommand{\hdot}{\dot{H}^1(\HH^n)}
\newcommand{\hnorm}{{\dot{H}^1}}

\begin{document}

\title[Concentration for nonlinear diffusion and P\'olya-Szeg\H{o} inequality]{Concentration comparison for nonlinear diffusion \\ on model manifolds and P\'olya-Szeg\H{o} inequality}

\author[M.~Muratori]{Matteo Muratori}
\author[B.~Volzone]{Bruno Volzone}

\address{M.~Muratori and B.~Volzone: Politecnico di Milano, Dipartimento di Matematica, Piazza Leonardo da Vinci 32, 20133 Milano, Italy}
\email[M.~Muratori]{matteo.muratori@polimi.it}
\email[B.~Volzone]{bruno.volzone@polimi.it}

	\makeatletter
	\@namedef{subjclassname@2020}{
		\textup{2020} Mathematics Subject Classification}
	\makeatother



\begin{abstract}
We investigate the validity of the \emph{mass concentration comparison} for a class of nonlinear diffusion equations, commonly known as filtration equations, posed on Riemannian manifolds $ \HH^n $ that are spherically symmetric, that is, \emph{model manifolds}. The concentration comparison states that the solution of a certain diffusion equation that takes the \emph{radially decreasing} (Schwarz) rearrangement $ u_0^\star $ as its initial datum is more concentrated than the original solution starting from $u_0$. This is known to hold in $\R^n$ as a consequence of the celebrated \emph{P\'olya-Szeg\H{o} inequality}, which asserts that the $ L^2 $ norm of the gradient of a function $f$ (belonging to an appropriate Sobolev space) is always larger than the $ L^2 $ norm of the gradient of its radially decreasing rearrangement $f^\star$. However, if $\HH^n$ is a general model manifold, it is not for granted that the P\'olya-Szeg\H{o} inequality holds; in fact, we will provide a simple condition involving the \emph{scalar curvature} of $\HH^n$ under which such an inequality actually fails. The main result we prove states that, given any continuous, nondecreasing, and nontrivial function $ \phi: [0,+\infty) \to [0,+\infty) $, the filtration equation $ \partial_t u = \Delta \phi(u) $ satisfies the concentration comparison in $ \HH^n \times (0,+\infty) $ if and only if $ \HH^n $ supports the P\'olya-Szeg\H{o} inequality. In particular, the validity of such a comparison for the \emph{heat equation} is sufficient to guarantee that the same holds for all filtration equations. Moreover, we prove that if $ \HH^n $ supports a \emph{centered isoperimetric inequality} then the P\'olya-Szeg\H{o} inequality, and thus the concentration comparison, holds. This allows us to include important examples such as the hyperbolic space and the sphere. Finally, we establish the same equivalence result restricted to radial functions, for which mild sufficient conditions on $\HH^n$ are available.
\end{abstract}

\maketitle

\setcounter{page}{1}

\numberwithin{equation}{section}

\section{Introduction}
Let $\HH^{n}$ be an $n$-dimensional ($ n \ge 2 $) complete Riemannian manifold. The main goal of this paper is to establish  symmetrization results, depending on the geometric properties of $ \HH^n $, in the form of \emph{mass concentration comparison} for nonnegative solutions of nonlinear diffusion Cauchy problems of the form
  \begin{equation}\label{filt-eqintr}
      \begin{cases} 
          \partial_t u = \Delta \phi(u) & \text{in } \HH^n \times (0,+\infty) \, , \\
          u= u_0 & \text{on } \HH^n \times \{ 0 \} \, ,
      \end{cases}
  \end{equation}
  also known as \emph{filtration equations} (see \cite{Vaz} and references therein). Specifically, the manifolds $\HH^{n}$ we consider are the so called \emph{Riemannian models}, which have an intrinsic spherical symmetry (we refer to Subsection \ref{models} for a detailed description), whereas the nonlinearity $\phi$ can be an arbitrary continuous and nondecreasing function on $ [0,+\infty) $. Typically, such kinds of concentration comparison results consist in coupling \eqref{filt-eqintr} with the following problem:
\begin{equation}\label{filt-eq-symmintr}
      \begin{cases} 
          \partial_t \overline{u} = \Delta \phi(\overline{u}) & \text{in } \HH^n \times (0,+\infty) \, , \\
          \overline{u} = u_0^\star & \text{on } \HH^n \times \{ 0 \} \, ,
      \end{cases}
  \end{equation}
 $u_0^\star$ being the \emph{Schwarz rearrangement} of $u_0$ on the manifold $\HH^{n}$, which is by construction a nonincreasing function with respect to the geodesic distance $r:=\mathrm{d}(x,o)$ from the \emph{pole} $o\in \HH^{n}$, sharing with $u$ the same Riemannian volume measure $V$ of upper level sets (see Subsection \ref{Schwarzsecti} for the main definitions and properties). In this respect, a comparison in the form of mass concentrations (or \emph{integral comparison}) can be expressed as
\begin{equation}\label{conccompintr}
 \int_{B_r} u^\star(x,t) \, dV(x) \leq \int_{B_r} \overline{u}(x,t) \, dV(x) \qquad \forall t >0 \, ,
\end{equation}
for all geodesic balls $B_{r}$ of radius $r>0$ centered at the pole. In this case we will say that $u(\cdot,t)$ is \emph{less concentrated} than $\overline{u}(\cdot,t)$, and we will use the symbol $ u^\star(\cdot,t) \prec \overline{u}(\cdot,t)$. Normally, one should write \eqref{conccompintr} with $  \overline{u}^\star $ in the place of $  \overline{u} $, but a key property we will prove is that $ \overline{u}(\cdot,t) $ is always radial and nonincreasing, thus it coincides with its Schwarz rearrangement.

Inequality \eqref{conccompintr} is a very powerful \emph{a priori} estimate, which can be exploited to determine various sharp properties of solutions, such as regularity in suitable $L^{p}$ scales or smoothing effects in direct forms; we refer \emph{e.g.}\ to \cite[Chapter 17]{Vaz} for a detailed account. In general, however, the validity of \eqref{conccompintr} is not for granted. More precisely, we will show that \eqref{conccompintr} is \emph{equivalent} to the fact that the model manifold $ \HH^n $ supports the \emph{P\'olya-Szeg\H{o} inequality}, that is,  
 \begin{equation}\label{P-Zintro1}
      \int_{\HH^n} \left| \nabla v^\star \right|^2 dV  \le \int_{\HH^n} \left| \nabla v \right|^2 dV 
    \end{equation}
for every function $ v $ belonging to an appropriate Sobolev space (see Subsections \ref{Sobolevdensi} and \ref{Schwarzsecti}). This statement is at the core of our main results, and such a kind of equivalence represents an intrinsic geometric characterization of the manifold $\HH^{n}$, which seems to be, to our knowledge, completely new. On the other hand, we notice that the above P\'olya-Szeg\H{o} inequality \eqref{P-Zintro} (or even its more general version with the $L^p$ norms of gradients) follows once a \emph{centered isoperimetric inequality} is shown to hold in $\HH^{n}$:
\begin{equation}\label{iso-1intro-A}
\mathrm{Per} (B_r) \le \mathrm{Per}(\Omega) \, ,
\end{equation}
where $B_r$ is the geodesic ball \emph{centered} at the pole $o$ having the same volume measure $ V $ as $  \Omega $ (the latter being an arbitrary bounded Borel set), and $\mathrm{Per}(\Omega)$ stands for the perimeter of $\Omega$ in $\HH^{n}$ in the sense of De Giorgi. As happens in the Euclidean case, we will show that \eqref{iso-1intro-A} implies \eqref{P-Zintro1}, see Proposition \ref{iso-polya} below. That said, by our main result the sole validity of \eqref{P-Zintro1} ensures in turn the mass concentration comparison \eqref{conccompintr}: in this regard, the assumption on the validity of \eqref{P-Zintro1} in order for \eqref{conccompintr} to hold is \emph{weaker} than \eqref{iso-1intro-A}. 

The class of Riemannian models for which the centered isoperimetric inequality \eqref{iso-1intro-A} is satisfied has not been completely understood so far: remarkable examples, aside the classical Euclidean space, are the \emph{hyperbolic space} $ \mathbb{H}^n $ and the sphere $\mathbb{S}^{n}$, see \cite{BogelDuz} and \cite{Gromov}, respectively. On the contrary, we are able to identify a clear geometric condition on $\HH^n$ that actually \emph{prevents} the validity of \eqref{P-Zintro1} (and therefore of \eqref{conccompintr} and \eqref{iso-1intro-A}), which will be stated in the next subsection.

It is beyond the scope of this paper to provide a full account of the several applications of isoperimetric inequalities in functional analysis  and partial differential equations. Nonetheless, we point out that such inequalities, in specific \emph{weighted} Euclidean frameworks, have been deeply investigated in \cite{B96, B99} through equivalent formulations, where the reference space is $\mathbb{R}^n$ endowed with the \emph{Gaussian measure} or, more in general, with a \emph{log-concave measure}. Moreover, in the recent paper \cite{BoVo}, interesting connections between isoperimetric inequalities and \emph{Poincar\'e inequalities} have been obtained in the case of \emph{Cauchy measures}.

\subsection{Highlights of our main results}
For the reader's convenience, we write in advance two of the main results of the paper, in a fashion that is not completely rigorous yet because we would need some technical concepts that are introduced in the next Section \ref{prem-tool}, such as the specific notion of solution to \eqref{filt-eqintr} and the actual Sobolev space in which \eqref{P-Zintro1} is supposed to hold. For the precise statements, we refer to Section \ref{mainconcenttheo}. 

\begin{theorem}\label{th-main-intro}
Let $\HH^n$ be a complete model manifold. Let $ \phi:[0,+\infty) \to [0,+\infty) $ be an arbitrary continuous, nondecreasing, and nontrivial function. Then the following properties are equivalent: 
\begin{enumerate}[(a)]
    \item \label{TAi} $ \HH^n $ supports the P\'olya-Szeg\H{o} inequality \eqref{P-Zintro1};

    \medskip 

    \item \label{TBi} For every $ u_0 \ge 0 $ it holds
    \begin{equation*}\label{conc-prop-intro}
      {u}^\star(\cdot , t)  \prec \overline{u}(\cdot, t) \qquad \forall t > 0 \, ,
    \end{equation*}
    where $u$ (resp.~$ \overline{u} $) is the solution of the Cauchy problem \eqref{filt-eqintr} (resp.~\eqref{filt-eq-symmintr}), and $  u^\star$ is its (spatial) Schwarz rearrangement.
\end{enumerate}
\end{theorem}

\begin{theorem}\label{polya-failure-intro}
    Let $ \HH^n $ be a model manifold with pole $o$. Suppose that the scalar curvature of $ \HH^n $ does not attain a global maximum at $ o $.
    Then $ \HH^n $ does not support the P\'olya-Szeg\H{o} inequality \eqref{P-Zintro1}.
\end{theorem}

In fact, Theorem \ref{th-main-intro} also has a \emph{purely radial} counterpart, that is, we can show that the P\'olya-Szeg\H{o} inequality restricted to radial functions is equivalent to the concentration comparison for \eqref{filt-eqintr} when $u_0 \ge 0$ is an arbitrary radial datum (see Theorem \ref{th-conc-rad} below). Interestingly, in Proposition \ref{nazarov} we are able to provide some conditions ensuring the validity of such a radial P\'olya-Szeg\H{o} inequality, which are much weaker than the aforementioned centered isoperimetric inequality.

\subsection{A brief history of symmetrization in elliptic and parabolic PDEs}
The idea of using symmetrization with the aim of obtaining sharp \emph{a priori} estimates for solutions to PDEs goes back to the pioneering works of Weinberger \cite{wein} and Maz'ya \cite{maz}. A complete and well-described method to face this kind of problems for elliptic PDEs can be found in the classical paper by Talenti \cite{Talenti1}. Typically, the techniques involved consist in comparing a second-order elliptic PDE of the form $ \mathcal L w=f \ge 0$ set in a certain domain $\Omega$ of $\R^{n}$ and coupled with the homogeneous Dirichlet boundary condition $w=0$ on $\partial \Omega$, with a suitable \emph{symmetrized problem} having the form (up to constants depending on the ellipticity of $\mathcal{L}$) $-\Delta \overline{w}=f^{\star}$ in $\Omega^{\star}$, completed with the boundary condition $\overline{w}=0$ on $\partial \Omega^{\star}$, where $\Omega^{\star}$ is the \emph{ball} centered at the origin having the same volume measure as $\Omega$ and $f^{\star}$ is the \emph{Schwarz rearrangement} of $f$.
Remarkably, the latter problem is purely \emph{radial} and therefore much easier to study, as it reduces to an ODE for which an explicit expression of $\overline{w}$ is available.

At the core of the method is a clever choice of a test function defined in terms of the upper level sets $\left\{w>t>0\right\}$ of $w$, combined with Federer's coarea formula \cite{Fed} and De Giorgi's isoperimetric inequality (in the Euclidean setting) plus some technical rearrangement inequalities, in order to obtain \emph{a priori} estimates for the radial derivatives of $w^{\star}$, which allows one to finalize the \emph{pointwise comparison} 
\begin{equation}\label{Talenti}
w^{\star}(x)\leq \overline{w}(x) \, , \quad x \in \Omega^{\star} \, .
\end{equation}
Such an inequality turns out to be crucial to obtain sharp regularity estimates (\emph{i.e.}\ with optimal constants) for $w$ in terms of $f$, along with key energy bounds. This kind of results was later generalized to the nonlinear setting by Talenti himself \cite{talNON}. Since then, many investigations have been undertaken in several directions with a similar focus, among which we quote \cite{ATlincei, ALTa, aflt}. An alternative and original approach presented in \cite{lions} consists in replacing the powerful tools of the coarea formula and the isoperimetric inequality with the Euclidean \emph{P\'olya-Szeg\H{o}} inequality
\begin{equation*}\label{P-Zintro}
      \int_{\Omega^{\star}} \left| \nabla w^\star \right|^2 dx  \le \int_{\Omega} \left| \nabla w \right|^2 dx \, ,
    \end{equation*}
which, as recalled in the beginning, is actually weaker than the isoperimetric inequality. This leads to a direct calculation of a differential quotient of the Dirichlet energy on the upper level sets of $w^{\star}$ (which are balls), allowing to reach the same pointwise comparison as in \eqref{Talenti}. This effective method has been properly modified into an essential tool for the achievement of \emph{integral comparisons} for solutions of \emph{nonlocal} elliptic PDEs in \cite{ferone2021symmetrization, FVnon, brandolini2022comparison} and in the more recent work \cite{ferone2024symmetrization} (note that in such a nonlocal setting a pointwise estimate of the form \eqref{Talenti} is shown to generally \emph{fail}).
 
In the parabolic context, the first noticeable contribution was due to Bandle \cite{Band}, where the author considers Cauchy-Dirichlet problems for linear parabolic equations of the type
\begin{equation}\label{lin-bandle}
\partial_{t} u + \mathcal{L}u=f \, ,
\end{equation}
where $\mathcal L$ is a second-order uniformly elliptic operator as above, set in a bounded cylinder $Q:=\Omega\times(0,T)$ with homogeneous boundary conditions $ u=0 $ on $ \partial \Omega \times (0,T) $ and prescribed initial datum
\[
u(x,0)=u_{0}(x) \ge 0 \, ,  \quad x\in\Omega \, .
\]
In this case, the original problem is coupled with the one corresponding to the classical heat equation (up to constants) with source term
\begin{equation}\label{linparabsym}
\partial_t\overline{u}-\Delta \overline{u}=f^{\star} \, ,
\end{equation}
posed in the cylinder $Q^\star:=\Omega^{\star}\times(0,T)$, still with homogeneous boundary conditions and initial datum given by
\[
\overline{u}(x,0)=u^{\star}(x) \, , \quad x \in \Omega^\star \, .
\]
Note that here $f$ may depend on time as well, so that $f^{\star}(x,t)$ stands for the (spatial) Schwarz rearrangement of $x \mapsto f(x,t)$ for any fixed $t>0$. 
Working in the setting of classical solutions, and observing that $ \overline{u}(\cdot,t) $ is by construction radially nonincreasing, a mass concentration comparison between $u$ and $\overline{u}$ is derived, written in the form
\begin{equation}\label{massBandle}
\int_{B_{r}}u^{\star}(x,t) \, dx\leq \int_{B_{r}}\overline{u}(x,t) \, dx \qquad \forall t>0 \, ,
\end{equation}
for all Euclidean balls $B_r$ centered at the origin. In fact, the latter turns out to be the only possible rearrangement comparison for parabolic equations, in the sense that, similarly to the nonlocal elliptic case, no pointwise rearrangement comparison can hold. The tools used in the proof of \eqref{massBandle} again involve the isoperimetric inequality and the coarea formula, but an additional issue concerning the presence of the time derivative appears: more precisely, it is necessary to address a technical obstruction regarding time derivation under the integral sign in the $x$ variable on the upper level sets of $u(\cdot,t)$. Moreover, as just mentioned, the radial form of the {symmetrized} problem \eqref{linparabsym} implies by maximum principle arguments that the solution $\overline{u}$ is radially nonincreasing with respect to $x$, therefore the comparison between $u^\star$ and $\overline{u}$ is expressed in terms of a boundary-value problem for an integro-differential inequality satisfied by the \emph{difference of concentrations}
\[
\chi(r,t) := \int_{B_{r}}\left[u^{\star}(x,t)-\overline{u}(x,t)\right] dx \, ;
\]
the thesis $\chi\leq0$ is then reached by means of a further maximum principle argument. Afterwards, such a result was generalized to less regular settings (that is, working with \emph{weak} solutions) in 
\cite{MossinoRak}, see also \cite{AlvLionTromb} for more nontrivial extensions.

In the realm of \emph{nonlinear} parabolic equations a first valuable result, in the form of a mass concentration comparison of the type \eqref{massBandle}, was obtained by V\'azquez in \cite{vazquez} for the Euclidean \emph{filtration equation}
\begin{equation*}\label{filtr}
\partial_t{u} = \Delta\phi(u) \, ,
\end{equation*}
where the choice $\phi(u)=u^{m}$, $m>1$, corresponds to the classical \emph{porous medium equation}. The approach presented in this work overcomes the delicate issue regarding the time derivative $\partial_{t}u$ mentioned above, by looking for comparison results in the cascade of nonlinear elliptic problems in the iterative form coming from the classical \emph{Euler implicit time discretization scheme}. Specifically, upon fixing a certain time interval $[0,T]$, one constructs an approximate piecewise-constant-in-time solution $u$ by recursively solving the semilinear elliptic equations
\[
-\Delta\phi(u(\cdot,t_{k})) + \tfrac{1}{h} \, u(\cdot,t_{k}) = \tfrac{1}{h} \, u(\cdot,t_{k-1}) \, ,
\]
where $h=T/N$, $N\in\mathbb{N} \setminus \{0 \} $, $t_{k}=kh$ for $k=0,1,\ldots,N$, and $u(\cdot,0)=u_{0}$ is the initial datum. According to this viewpoint, each of the above equations is coupled with 
\[
-\Delta\phi(\overline u(\cdot,t_{k})) + \tfrac{1}{h} \, \overline u(\cdot,t_{k}) = \tfrac{1}{h} \, \overline u(\cdot,t_{k-1}) \, ,
\]
the corresponding initial datum being $\overline{u}(\cdot,0)=u^{\star}_{0}$. Once the elliptic mass concentration comparison
\begin{equation}\label{ell-k-conc}
\int_{B_{r}}u^{\star}(x,t_{k}) \, dx\leq
\int_{B_{r}}\overline{u}(x,t_{k}) \, dx 
\end{equation}
is established for all $k \in \N$ and all $ r>0 $, it is possible to invoke the celebrated \emph{Crandall-Liggett} theorem from nonlinear semigroup theory (see \emph{e.g.}\ \cite[Chapter 10]{Vaz}), hence by passing to the limit in the above scheme as $h\rightarrow0^+$ one obtains the so called \emph{mild solutions} $u$ and $\overline{u}$. In particular, convergence holds globally in $ L^1 $, which suffices to take limits in the discrete concentration inequalities \eqref{ell-k-conc} to recover \eqref{massBandle}. The additional presence of a source term $f$ as in \eqref{lin-bandle} would not be a problem in order for this approach to work. We point out that the advantages of this strategy, along with key consequences, are well described in a greater generality in \cite{VANS05}.
The implicit time-discretization scheme focused to obtaining mass concentration comparisons for nonlinear diffusion equations was also successfully exploited in the more recent works \cite{VazVol1, VazVol2, VazVolSire}, where the dynamics is driven by {nonlocal} \emph{fractional Laplacian} diffusion operators.

For what concerns the application of symmetrization methods in PDEs on \emph{Riemannian manifolds}, we must comment that, to date, there are only a few papers  available in the literature. As explained before, one of the key questions is the validity of a centered isoperimetric inequality on the reference manifold $\HH^{n}$. For instance, a pointwise comparison result in the form of Talenti \eqref{Talenti} has been obtained in \cite{Ngo} for elliptic equations on the \emph{hyperbolic space} $\mathbb{H}^{n}$. In the parabolic setting, a mass concentration comparison result was derived in \cite{chenwei} for filtration equations on a certain class of manifolds $\HH^{n}$ complying with the following constant bounds from below for the Ricci curvature:
\begin{equation}\label{db-ricci}
\mathrm{Ric} \geq (n-1) \, \kappa \, ,
\end{equation}
for $ \kappa \in\left\{0,1\right\}$. Nevertheless, we need to stress that the isoperimetric inequality used in this paper is in the following form illustrated by Brendle \cite{Brendle} and Gromov \cite{Gromov}:
\begin{equation}\label{iso-1intro}
\alpha_{\kappa} \, \mathrm{Per}_{\mathbb{R}^{n}(\kappa)} (B_r)\le \mathrm{Per}_{\HH^{n}}(\Omega) \, ,
\end{equation}
where we use the concise notation $\mathbb{R}^{n}(\kappa)$ to denote $\mathbb{R}^{n}$ when $\kappa=0$ and the unit sphere $\mathbb{S}^{n}$ when $\kappa=1$, whereas $ \alpha_\kappa $ is a positive constant accounting for the \emph{asymptotic volume ratio} of $\HH^n$ when $ \kappa=0 $ or the total volume ratio between $ \HH^n $ and $ \mathbb{S}^n $ when $ \kappa=1 $. More precisely, in \eqref{iso-1intro} we let $B_{r}$ denote the geodesic ball in $\R^{n}(\kappa)$ whose volume is $\alpha_{\kappa}^{-1} \, V_{\HH^{n}}(\Omega)$, where $V_{\HH^{n}}$ stands for the Riemannian volume measure in the reference manifold $\HH^n$. Note that inequalities of the type \eqref{iso-1intro}, and their rigidity properties, have been recently investigated in \cite{BK} even in the more general setting of \emph{metric-measure spaces}. This approach, however, forces one to consider either $\R^{n}$ or $\mathbb{S}^{n}$ as a \emph{target} manifold in the concentration comparison result: as a consequence, the ``radial'' parabolic symmetrized problem is necessarily set in $\R^{n}(\kappa)$ rather than on the manifold $\HH^{n}$ itself, which is instead at the core of our approach. Furthermore, due to \eqref{db-ricci}, the results of \cite{chenwei} are applicable only for manifolds with \emph{nonnegative Ricci curvature}, leaving out \emph{e.g.}\ the hyperbolic space $\mathbb{H}^{n}$, which is actually included in our main result.

\subsection{Outline of contents and paper organization} The next Section \ref{prem-tool} serves as an introduction to the fundamental geometric and functional tools that are broadly used throughout the rest of the paper. In particular, it is divided into four subsections, which are focused on model manifolds, Schwarz rearrangements, Sobolev spaces, and the well-posedness of \eqref{filt-eqintr}, respectively. 
In Section \ref{mainconcenttheo} we state and describe in detail our main results, namely Theorem \ref{th-conc} (\emph{i.e.}~Theorem \ref{th-main-intro}), Theorem \ref{th-conc-rad}, and Theorem \ref{polya-failure} (\emph{i.e.}~Theorem \ref{polya-failure-intro}), along with the companion Propositions \ref{iso-polya} and \ref{nazarov}.
In Section \ref{loc-elliptic} we focus on the analysis of local elliptic problems, and in Section \ref{loc-parabolic} we pass from elliptic to parabolic problems in the spirit of the time discretization introduced above. After having laid all of the groundwork, in Section \ref{proof-main} we finalize the proofs of the main results. In addition, we devote Appendix \ref{aux} to the proofs of technical auxiliary results, most of which are stated in Section \ref{prem-tool}. 

\section{Preliminary geometric and functional-analytic tools}\label{prem-tool}
In this section we provide an overview of the geometric setting in which we set our main problems, along with related functional spaces and inequalities that will be used regularly in the sequel. Some of the proofs, which can be considered standard or merely technical, are deferred until Appendix \ref{aux}.  

\subsection{Model manifolds and geometric fundamentals}\label{models}
We let $ \HH^n $ denote an $n$-dimen\-sional ($n  \ge 2$) \emph{model manifold}, that is, a Riemannian manifold whose metric $  \mathsf{g} $ can be written as
\begin{equation}\label{metric-model}
 \mathsf{g} \equiv dr \otimes dr + \psi(r)^2 \,  \mathsf{g}_{\mathbb{S}^{n-1}} \, , 
\end{equation}
where $r$ is the geodesic distance to a fixed point $o \in \HH^n$ called the \emph{pole} of the manifold, $dr$ is the tangent vector in the direction of $r$ (\emph{i.e.}~the \emph{radial} vector), $  \mathsf{g}_{\mathbb{S}^{n-1}} $ is the canonical metric on the $(n-1)$-dimensional unit sphere and $\psi : [0,+\infty) \to [0,+\infty) $ is a suitable smooth real function that determines all of the properties of $ \HH^n $. In this setting, every point $ x \in \HH^n \setminus \{  o \} $ can be uniquely identified by a pair $ (r,\theta) \in (0,+\infty) \times \mathbb{S}^{n-1} $ (the so called \emph{polar coordinates}), representing its distance to the pole and, respectively, the angle of the geodesic connecting it to $o$. For a short introduction to model manifolds we refer to \cite[Section 3.10]{Grig-book}. Note that model manifolds are a special but significant subcase of \emph{warped products}, see \emph{e.g.}~\cite[Section 1.8]{AMR}. 

The function $ \psi $ generating the metric of $ \HH^n $ can be any real function on $ [0,+\infty) $ satisfying the following properties (see \cite[Definition 1.1]{AMR}):
\begin{equation}\label{prop-psi}
\psi \in C^\infty([0,+\infty)) \, , \qquad \psi>0 \quad \text{in } (0,+\infty) \, , \qquad \psi'(0)=1 \, , \qquad \psi^{(2k)}(0)=0 \quad \forall k \in \N \, , 
\end{equation}
which are necessary and sufficient for $ \HH^n $ to be a smooth Riemannian manifold. Often, most of the analysis can be carried out by only requiring the conditions $ \psi(0)=0 $ and $ \psi'(0)=1 $, although, rigorously, the underlying model manifold may fail to be smooth at the pole $o$. We point out that by defining $ \psi $ in the whole half line $ [0,+\infty) $ we are implicitly assuming that $ \HH^n $ is \emph{complete} and \emph{noncompact}, which is the case we are mainly interested in. Compact model manifolds (without boundary) can be achieved by defining $ \psi $ in a bounded interval $[0,b]$ and assuming that $ \psi(b)=0 $.

Some classical examples of model manifolds are the Euclidean space $ \R^n $, the hyperbolic space $ \mathbb{H}^n $, and the unit sphere $ \mathbb{S}^{n} $, which correspond to $ \psi(r)=r $, $ \psi(r)=\sinh(r) $ and, respectively, $ \psi(r)=\sin r $ (defined on $ [0,\pi] $). If, in addition, $ \psi $ is a \emph{convex} function, then $ \HH^n $ turns out to be a \emph{Cartan-Hadamard} manifold, that is, a complete and simply connected Riemannian manifold with nonpositive sectional curvatures (see the classical reference \cite{GW}). The prototypical example is again $ \mathbb{H}^n $, whose sectional curvatures are identically equal to $-1$. As outlined in the Introduction, Cartan-Hadamard manifolds have received a lot of attention in the last decade in relation to nonlinear diffusion; we stress, however, that here we \emph{do not} need $ \HH^n $ to be Cartan-Hadamard.

Let us denote by $ \mathrm{d}(x,y) $ the geodesic distance on $ \HH^n $ induced by the metric $ \mathsf{g}$ according to \eqref{metric-model}. Clearly, if $y=o$ we have $ \mathrm{d}(x,o)=r $, whereas if $ \mathrm{d}(x,o) = \mathrm{d}(y,o) = r$ then $ \mathrm{d}(x,y) = \psi(r) \, \mathrm{d}_{\mathbb{S}^{n-1}}(\theta_x , \theta_y) $, where $\mathrm{d}_{\mathbb{S}^{n-1}}(\theta_x , \theta_y)$ stands for the geodesic distance on the $(n-1)$-dimensional unit sphere between the angles $ \theta_x $ and $ \theta_y $ associated with $ x $ and $y$, respectively. We can then define the geodesic ball of radius $r>0$ centered at some $x_0 \in  \HH^n $ by
$$
B_r(x_0):= \left\{x \in \HH^n : \ \mathrm{d}(x,x_0) < r \right\} ;
$$
in the special case $ x_0 = o $, to lighten the notation we simply write $ B_r $, since we will often deal with such balls. The Riemannian \emph{volume measure} determined by $ \mathsf{g}$ in the coordinate frame $ x \equiv (r,\theta) $ is given by the product measure
$$ dV(x) = \psi(r)^{n-1} \, dr \, dV_{{\mathbb{S}}^{n-1}}(\theta) \, ,$$
where $ dV_{{\mathbb{S}}^{n-1}} $ is the Riemannian volume (or area) measure on the $(n-1)$-dimensional unit sphere. In particular, the volume of balls centered at $ o $ reads 
\begin{equation}\label{eq-G-vol}
V(B_r) = \omega_n \, \int_{0}^{r} \psi(s)^{n-1} \, ds \, ,
\end{equation}
where the constant $ \omega_n $ represents the total area of the $ (n-1) $-dimensional unit sphere, that is, $ \omega_n := V_{{\mathbb{S}}^{n-1}}({\mathbb{S}}^{n-1}) $. For future convenience, we set
\begin{equation}\label{eq-G}
G(r) := \int_{0}^{r} \psi(s)^{n-1} \, ds \, ,
\end{equation}
which is clearly a bijection of $ [0,+\infty) $ onto itself. Also, the $(n-1)$-dimensional Hausdorff measure induced by $ dV $ and $ \mathrm{d} $, which we denote by $ \mathcal{V}_{n-1} $, on geodesic spheres centered at $ o $ reads  
\begin{equation}\label{E1}
\mathcal{V}_{n-1}(\partial B_r) = \omega_n \, \psi(r)^{n-1} \, ,
\end{equation}
which coincides with the \emph{perimeter} of $ B_r $. 

In the special framework of model manifolds, the expression of relevant geometric quantities such as sectional, Ricci, and scalar curvatures is quite simple (we refer to \cite[Section 1.8]{AMR}, \cite[Section 15.2]{Grig-recor} or \cite[Section 2]{GMV-17} for a brief account). More precisely, the {radial} sectional curvature $ {K}_{\mathrm{rad}} $ at $ x \equiv (r,\theta) $, that is, the Gauss curvature associated with planes in the tangent space $ T_x \HH^n$ that contain the radial vector $ dr $, equals
\begin{equation}\label{sec-rad}
{K}_{\mathrm{rad}} =  -\frac{\psi''(r)}{\psi(r)} \, ,
\end{equation}
whereas the sectional curvature $ {K}_{\perp} $ associated with planes that are orthogonal to $dr$ equals
\begin{equation*}\label{sec-orth}
{K}_{\perp} = \frac{1-\left[\psi'(r)\right]^2}{\psi(r)^2}  \, .
\end{equation*}
Because the Ricci curvature $ \mathrm{Ric}(e,e) $ in an arbitrary unit direction $ e \in  T_x \HH^n $ is the sum of all of the $ (n-1) $ sectional curvatures associated with planes containing $e$, thanks to \eqref{sec-rad} we have that the radial Ricci curvature $ \mathrm{Ric}_{\mathrm{rad}} := \mathrm{Ric}(dr,dr) $ reads 
\begin{equation}\label{ric-rad}
\mathrm{Ric}_{\mathrm{rad}} = -(n-1) \,  \frac{\psi''(r)}{\psi(r)} \, ,
\end{equation}
whereas the Ricci curvature $ \mathrm{Ric}_{\perp} $ in any of the $(n-1)$ vectors orthogonal to $ dr $ reads  
\begin{equation}\label{ric-orth}
\mathrm{Ric}_{\mathrm{rad}} = - \frac{\psi''(r)}{\psi(r)} + (n-2) \, \frac{1-\left[\psi'(r)\right]^2}{\psi(r)^2} 
  \, .
\end{equation}
The \emph{scalar curvature} $S(x)$ at $x$ is the trace of the Ricci curvature tensor or, equivalently, the sum of $ \mathrm{Ric}(e_i,e_i) $ over an orthonormal basis $ \{ e_i \}_{i=1\ldots n} $ of $ T_x(\HH^n) $. In particular, for model manifolds, thanks to \eqref{ric-rad} and \eqref{ric-orth} we obtain
\begin{equation}\label{scal-curv}
S(x) \equiv S(r) =-(n-1) \left[ 2 \, \frac{\psi''(r)}{\psi(r)} + (n-2) \, \frac{\left[\psi'(r)\right]^2-1}{\psi(r)^2}  \right] .
\end{equation}
There is a well-known relation between the scalar curvature and the asymptotic volume and perimeter of ``small'' geodesic balls. Indeed, for every $ x \in \HH^n $, we have the following expansions (see \emph{e.g.}~\cite[Chapter XII.8]{Cha}):
\begin{equation}\label{chavel-vol}
  V\!\left( B_r(x) \right)  = \frac{\omega_n}{n} \, r^{n} \left( 1 - \tfrac{S(x)}{6(n+2)} \, r^2 + \mathcal{O}\!\left( r^3 \right) \right) 
\end{equation}
and 
\begin{equation}\label{chavel-area}
  \mathcal{V}_{n-1}\!\left( \partial B_r(x) \right)  = \omega_n \, r^{n-1} \left( 1 - \tfrac{S(x)}{6 n} \, r^2 + \mathcal{O}\!\left( r^3 \right) \right) ,
\end{equation}
which in fact hold on any Riemannian manifold, not only on models.

Finally, we recall that the Riemannian gradient of a regular enough function $ f \equiv f(r,\theta) $, in the above polar-coordinate system, reads 
$$
\nabla f = \left( \partial_r f \, ,  \tfrac{1}{\psi(r)^2} \, \nabla_\theta f \right) ,
$$
where $ \nabla_\theta $ denotes the gradient of $ f  $ restricted to $ \mathbb{S}^{n-1} $; moreover, the Laplace-Beltrami operator, or simply \emph{Laplacian}, is equal to (see \emph{e.g.}~\cite[Section 3.2]{Grig-recor} or \cite[Section 3.10]{Grig-book})
\begin{equation*}\label{laplacian}
\Delta f = \partial_{rr}^2 f + (n-1) \, \frac{\psi'}{\psi} \, \partial_r f + \tfrac{1}{\psi(r)^2} \, \Delta_\theta f \, ,
\end{equation*}
the symbol $ \Delta_\theta $ denoting the Laplace-Beltrami operator on $  \mathbb{S}^{n-1} $. It is plain that, for smooth enough and compactly supported functions, the following integration-by-parts formula holds:
$$
\int_{\HH^n} \left\langle \nabla f , \nabla g \right\rangle dV = - \int_{\HH^n} f  \, \Delta g \,  dV \, ,
$$
where $ \left\langle \cdot , \cdot \right\rangle $ stands for the scalar product in $ T_x \HH^n $ induced by $ \mathsf{g} $. In particular, in the relevant case where $f$ is a purely \emph{radial} function, that is, $ f \equiv f(r) $ (thus we will write $ ' $ in the place of $ \partial_r $), we obtain the identities
\begin{equation}\label{laplacian-rad}
\Delta f = f'' + (n-1) \, \frac{\psi'}{\psi} \, f'
\end{equation}
and
\begin{equation*}\label{grad-norm-rad}
\int_{\HH^n} \left| \nabla f  \right|^2 dV =  \int_{0}^{+\infty} \left| f'(r) \right|^2  \psi(r)^{n-1} \, dr \, .
\end{equation*}

\begin{remark}[On the concept of radial function]\label{rem-rad}\rm 
In order to avoid confusion, we stress that whenever we refer to a radial function, as above, we implicitly mean that it only depends on the geodesic distance $r$ {from} $o$, that is, it is radially symmetric \emph{with respect to the pole} $o$ and not just any point. Moreover, in that case, we will often see $f$ both as a one-variable function on $ [0,+\infty) $ and as a function on $ \HH^n $, without changing notation when no ambiguity occurs. 
\end{remark}

\subsection{Sobolev spaces and density properties} \label{Sobolevdensi} 
As is customary, we let $ \hdot$ denote the space of functions $ v \in L^2_{\mathrm{loc}}(\HH^n) $ with $ \nabla v \in L^2(\HH^n) $ that can be approximated by test functions, in the sense that
\begin{equation}\label{seq-def-h1}
\exists  \{ \varphi_k \} \subset C^\infty_c(\HH^n): \qquad \lim_{k \to \infty} \left\| \nabla \varphi_k - \nabla  v \right\|_{L^2(\HH^n)} = 0 \quad \text{and} \quad  \varphi_k \underset{k\to\infty}{\longrightarrow} v \quad \text{in } L^2_{\mathrm{loc}}(\HH^n) \, .
\end{equation}
Formally, we should write $ \nabla v \in L^2(\HH^n ; T_x {\HH^n})  $, but we will avoid this notation for the sake of readability, taking for granted that when we write $\nabla v \in L^2(\HH^n) $ (or similar expressions) we mean that the weak gradient of $v$ is a measurable function such that $ |\nabla v| \in  L^2(\HH^n) $.    

Following \emph{e.g.}~the approach of \cite[Chapter 5]{Grig-recor}, it is always possible to endow $ \hdot $ with a Hilbert norm that induces the above convergence, for instance
\begin{equation}\label{hil-norm}
\left\| v \right\|_\hnorm^2  := \left\| v \right\|_{L^2(B_1)}^2 + \left\| \nabla v \right\|_{L^2(\HH^n)}^2 .
\end{equation}
This is not difficult to check, but for completeness we prove it (Proposition \ref{lemma-norm}). Before, let us recall that for every bounded smooth domain $ \Omega \subset \HH^n $ the symbol $ H^1(\Omega) $ denotes the space of all functions $ v \in L^2(\Omega) $ with $ \nabla v \in L^2(\Omega) $ (endowed with the squared norm $ \| v \|_{L^2(\Omega)}^2 + \| \nabla v \|_{L^2(\Omega)}^2 $), and $ H^1_0(\Omega) $ denotes the subspace formed by functions that in addition vanish on $ \partial \Omega $ (endowed with the norm $ \| \nabla v \|_{L^2(\Omega)} $). Finally, the symbol $ H^{-1}(\Omega) $ stands for the \emph{dual space} of $ H^1_0(\Omega) $.

Next, we recall some standard local Poincar\'e-type inequalities and Sobolev embeddings, for which we refer \emph{e.g.}~to \cite[Chapter 2]{Heb}. 

\begin{proposition}[local Poincar\'e inequality]\label{local poin}
  Let $ R>0 $. There exists a constant $C_R>0$, depending on $ n,\psi,R $, such that
  \begin{equation}\label{lp-ineq}
   \left\|  v - \overline{v}_R \right\|_{L^2(B_R)}   \le C_R \left\| \nabla v \right\|_{L^2(B_R)} \qquad \forall v \in H^1(B_R) \, ,
  \end{equation}
  where
  $$
\overline{v}_R := \frac{ \int_{B_R} v \, dV }{V(B_R)} \, .
  $$
  Moreover, the embedding $ H^1(B_R) \hookrightarrow L^2(B_R) $ is compact.
\end{proposition}

\begin{proposition}[local Sobolev inequality]\label{local sob}
Let $ R>0 $. Let $ p \in [2,\infty) $ if $ n=2 $ and $ p \in [2,2n/(n-2)]  $ if $ n \ge 3 $. There exists a constant $C_{p,R}>0$, depending on $ n,\psi,R,p $, such that
\begin{equation}\label{loc-sob}
 \left\| v \right\|_{L^p(B_R)}   \le C_{p,R} \left\| \nabla v \right\|_{L^2(B_R)} \qquad \forall v \in H^1_0(B_R) \, .
\end{equation}
\end{proposition}

As a consequence of Proposition \ref{local poin}, it is readily seen that \emph{any} local $ L^2 $ norm can be controlled with the $ \hnorm $ norm.

\begin{lemma}\label{equiv}
  Let $ R>0 $. There exists a constant $ K_R>0 $, depending on $n,\psi,R$, such that
\begin{equation}\label{loc-norm}
\left\| v \right\|_{L^2(B_R)} \le K_R \left\| v \right\|_\hnorm \qquad \forall v \in \hdot \, .
\end{equation}
\end{lemma}
\begin{proof}
See Appendix \ref{aux}, Subsection \ref{aux-1}.
\end{proof}

As mentioned above, the norm $ \| \cdot \|_\hnorm $ makes $ \hdot $ a Hilbert space.

\begin{proposition}[The reference Sobolev space]\label{lemma-norm}
The normed space $ \big( \hdot \, , \| \cdot \|_\hnorm  \big) $ is a separable Hilbert space in which $ C_c^\infty(\HH^n) $ is dense.
\end{proposition}
\begin{proof}
See Appendix \ref{aux}, Subsection \ref{aux-1}.
\end{proof}

A subtle question is whether $ \| v \|_\hnorm $ is an \emph{equivalent} norm to $ \| \nabla v  \|_{L^2(\HH^n)} $ in the space $  \hdot$,  \emph{i.e.}, if the additional term $ \| v \|_{L^2(B_1)} $ in \eqref{hil-norm} can actually be dropped. This problem is strictly related to the \emph{nonparabolicity} of $ \HH^n $, that is, to the fact that $ \HH^n $ admits a minimal positive \emph{Green's function}, which, for model manifolds, amounts to the \emph{finiteness} of the integral (see \cite[Proposition 3.1]{Grig-recor})
$$
\int_1^{+\infty} \frac{1}{\psi(r)^{n-1}} \, dr \, . 
$$
Indeed, it turns out that such norms are equivalent if and only if $\HH^n $ is nonparabolic, and this is in turn equivalent to the fact that \emph{constants do not belong} to $ \hdot $. For the details, we refer to \cite[Theorem 5.7]{Grig-recor}. Since, to our purposes, there is no reason to rule out parabolic manifolds, the introduction of the $ L^2(B_1) $ norm in \eqref{hil-norm} is in general unavoidable. 

Next, we let $ H^1_c(\HH^n) $ be the subspace of functions in $ \hdot $ that in addition have \emph{compact support}. It is a standard fact that every function in $ H^1_c(\HH^n) $ can be approximated by a sequence of functions in $ C^\infty_c(\HH^n) $ with respect to $ \| \cdot \|_ \hnorm$, hence we will take it for granted from here on. The following result shows, by means of a simple truncation argument, that it is always possible to approximate functions in $ \hdot $ with functions in $ H^1_c(\HH^n) $ without exceeding the $ L^\infty(\HH^n) $ norm.

\begin{lemma}\label{approx-above}
 Let $ v \in \hdot $. There exists a sequence $ \{ v_k \} \subset H^1_c(\HH^n) $ such that
 \begin{equation*}\label{approx-conv}
  - v^- \le v_k \le v^+ \quad \forall k \in \N \qquad \text{and} \qquad 
     \lim_{k \to \infty} \left\| v_k - v \right\|_{\hnorm} = 0 \, .
 \end{equation*}
 Moreover, the sequence $ \{ v_k \} $ can be chosen in such a way that
 \begin{equation*}\label{ev-ug}
    v_k = v \qquad \text{in } B_{R_k} \, , 
 \end{equation*}
 for some increasing sequence $ R_k \to +\infty $.
\end{lemma}
\begin{proof}
See Appendix \ref{aux}, Subsection \ref{aux-1}.
\end{proof}

We now switch to suitable space-time Sobolev spaces, which are the proper functional framework where to set the weak formulation of \eqref{filt-eqintr}. 

\begin{definition}[The reference Bochner space]\label{def-sob-parabolico}
    Given $ T >0 $, we denote by
    \begin{equation}\label{not-H}
L^2\big( (0,T) ; \hdot \big)
    \end{equation}
    the space of all functions 
    $ u \in L^2_{\mathrm{loc}}\!\left( \HH^n \times [0,T]  \right) $ such that $ \nabla u \in L^2\!\left( \HH^n \times (0,T) \right) $ and 
    $$
  u(\cdot,t)  \in \hdot \qquad \text{for a.e. } t \in (0,T) \, .
    $$
 \end{definition}

We recall that, given a Banach space $ \mathcal{X} $ and $ p \in [1,\infty) $, the \emph{Bochner space} $ L^p\!\left( (0,T) ; \mathcal{X} \right) $ is usually defined as the space of all (classes of equivalence of) functions $ u : (0,T) \to \mathcal{X} $ such that there exists a sequence of simple functions $ s_k : (0,T) \to \mathcal{X} $ satisfying
\begin{equation}\label{v-norm}
\lim_{k \to \infty} \int_0^T \left\| u(t) - s_k(t) \right\|_{\mathcal{X}}^p \, dt = 0 \, .
\end{equation}
If $ \mathcal{X} $ is separable, then \cite[Appendix E]{Cohn} it turns out that $ u \in L^p\!\left( (0,T) ; \mathcal{X} \right)  $ if and only if it is \emph{weakly measurable} and satisfies 
\begin{equation}\label{v-norm-weak}
\int_0^T \left\| u(t) \right\|_{\mathcal{X}}^p \, dt < + \infty \, ,
\end{equation}
where weak measurability means that the real function $ t \mapsto \Lambda u (t) $ is measurable for every $ \Lambda \in \mathcal{X}' $. Such a characterization is very useful since \eqref{v-norm} is, \emph{a priori}, more difficult to verify than weak measurability and \eqref{v-norm-weak}. In particular, for such functions, the \emph{Bochner integral}
\begin{equation*}\label{v-norm-bohn}
\int_0^t u(s) \, ds  \in \mathcal{X} \qquad \forall t \in (0,T) 
\end{equation*}
is always well defined. The space $ L^p\!\left( (0,T) ; \mathcal{X} \right)  $ is also Banach when it is endowed with the norm \eqref{v-norm-weak} (raised to the power $p$). Since $ \hdot $ is separable (recall Proposition \ref{lemma-norm}), the notation \eqref{not-H} is consistent with that of Bochner spaces with $ p=2 $ and $ \mathcal{X}=\hdot $ (weak measurability will be proved along the proof of the next proposition), so that $ L^2\big( (0,T) ; \hdot \big) $ indeed becomes a Hilbert space endowed with the (squared) norm
    $$
\left\| u \right\|_{L^2\left( (0,T) ; \hnorm \right)}^2 := \int_0^T \left\| u(\cdot,t) \right\|_{\hnorm}^2 dt \qquad \forall  u \in L^2\big( (0,T) ; \hdot \big) \, .
    $$
As an important consequence of the definition of $ \hdot $ and the basic properties of Bochner spaces, we have that space-time compactly supported smooth functions are dense in $ L^2\big( (0,T) ; \hdot \big) $.

\begin{proposition}
\label{density-bochner}
Let $ u \in L^2\big( (0,T) ; \hdot \big) $. There exists a sequence $  \left\{ \xi_k \right\} \subset C^\infty_c(\HH^n \times (0,T))  $ such that
\begin{equation}\label{density-bochner-test}
\lim_{k \to \infty}\left\| u-\xi_k \right\|_{L^2\left( (0,T) ; \hnorm \right)} = 0 \, .
\end{equation}
If, in addition, $ \|u\|_{L^\infty( \HH^n \times(0,T) )} \le M $ for some constant $M>0$, then $ \{ \xi_k \} $ can be chosen in such a way that 
\begin{equation*}\label{density-bochner-test-bdd}
 \left\| \xi_k \right\|_{L^\infty( \HH^n \times(0,T) )} \le M \qquad \forall k \in \N \, .
\end{equation*}
\end{proposition}
\begin{proof}
See Appendix \ref{aux}, Subsection \ref{aux-1}.
\end{proof}

\subsection{Schwarz rearrangements and P\'olya-Szeg\H{o} inequality}\label{Schwarzsecti}
We denote by $\mathcal{L}_{0}(\HH^{n})$ the collection of all measurable functions $f:\HH^{n}\rightarrow\R$ whose \emph{upper positive level sets} (also called \emph{superlevel sets}) are finite, \emph{i.e.}~such that
\[
V\!\left(\left\{x\in \HH^{n}: \ |f(x)|>t\right\}\right)<+\infty
\]
for all $t>0$. For any $f\in\mathcal{L}_{0}(\HH^{n})$, the \emph{distribution function} $\mu_{f}$ of $f$ is defined by
\begin{equation}\label{distr-fun}
\mu_{f}(t):= V\!\left( \left\{x\in \HH^{n}: \ |f(x)|>t\right\} \right) \qquad \forall t>0 \, .
\end{equation}
Clearly $ t \mapsto \mu_f(t) $ is nonincreasing, and it is readily seen that it is right continuous. 

The one dimensional \emph{Hardy-Littlewood} rearrangement of $f$ is then defined as the generalized inverse of $\mu_{f}$, through the formula
\[
f^{\ast}(s):=\sup\left\{t\geq 0: \ \mu_{f}(t)>s\right\} ,
\]
which can also be equivalently written as
\begin{equation}\label{f-sharp-int}
f^{\ast}(s)= \int_0^{+\infty} \chi_{\left\{  \mu_f >s \right\}}(t) \, dt \, .
\end{equation}
It is plain that $ [0,+\infty) \ni s \mapsto  f^\ast(s) $ is always nonincreasing and lower semicontinuous, possibly taking the value $+\infty$ at $ s=0 $ only. In particular, it is right continuous but may have (countably many) left jumps. 

We recall that $f$ and $f^{\ast}$ are \emph{equimeasurable} functions, in the sense that $ \mu_f = \mu_{f^*} $ (see for instance \cite{BS}), and by the Cavalieri principle (or layer-cake representation) it holds
\begin{equation}\label{pn}
\left\| f \right\|_{L^{p}(\HH^n)}=\|f^{\ast}\|_{L^{p}((0,+\infty))} \qquad \forall p \in [1,\infty] \, .
\end{equation}
In agreement with the notations of Subsection \ref{models}, upon setting $r \equiv \mathrm{d}(x,o)$, we can define the \emph{Schwarz rearrangement} of $f$ as the radial function $f^{\star}$ defined through the identity
\begin{equation}\label{from-ast-to-star}
f^{\star}(x) := f^{\ast}\!\left(V\!\left(B_{\mathrm{d}(x,o)}\right)\right)=f^{\ast}\!\left(V\!\left(B_r\right)\right) \qquad \forall x \in \HH^n \, .
\end{equation}
By construction, the $t$-upper level set of $f^{\star}$ is exactly the geodesic ball $B_{\rho(t)}$ centered at the pole $o$, where the radius $\rho(t)>0$ is the unique positive number complying with $V\!\left(B_{\rho(t)}\right)=\mu_{f}(t)$, that is,
\[
\int_{0}^{\rho(t)} \psi(s)^{n-1} \, ds=\frac{\mu_{f}(t)}{\omega_{n}} \, ,
\]
or equivalently (recall \eqref{eq-G})
\begin{equation}\label{exprrho}
\rho(t)=G^{-1}\!\left(\frac{\mu_{f}(t)}{\omega_{n}}\right) .
\end{equation}
We will say that a function $f$ is \emph{rearranged} if it is essentially radially nonincreasing: in this case, it is apparent that $f=f^{\star}$ almost everywhere, with identity everywhere if $ f $ is in addition right continuous. 

Still from the Cavalieri principle, we observe that identity \eqref{pn} is much mure general: for any Borel-measurable function $F:[0,+\infty)\rightarrow[0,+\infty)$ it holds
\begin{equation}\label{identity-int}
\int_{\HH^{n}}F(|f|) \, dV=\int_{\HH^{n}}F(f^{\star}) \, dV =\int_{0}^{+\infty}F(f^{\ast}(s)) \, ds \, ,
\end{equation}
see for instance \cite[Theorem 1.1.1]{Kes}. In particular, by using \eqref{identity-int} with the composition of arbitrary nondecreasing functions and characteristic functions $ \chi_{|f|>t} $, one can easily deduce that if $F:[0,+\infty)\rightarrow[0,+\infty)$ is any nondecreasing function with $F(0)=0$, then $ \left(F(|f|)\right)^\star = F(f^\star) $ almost everywhere, a fact that we will exploit at several stages (see also \cite[Proposition 1.1.4]{Kes}). 

Next, we focus on some useful properties of Schwarz rearrangements in relation to integral inequalities, approximations, and Sobolev spaces. First of all, for any pair of functions $f,g \in \mathcal{L}_{0}(\HH^{n})$, we have the following crucial \emph{Hardy-Littlewood inequality} (see \emph{e.g.}\ \cite[Theorem 2.2]{BS}): 
\begin{equation}\label{h-litt}
\int_{\HH^{n}}f\,g\,dV \leq \int_{\HH^{n}}f^{\star} \, g^{\star}\,dV =\int_{0}^{+\infty}f^{\ast}(s) \, g^{\ast}(s) \, ds \, .
\end{equation}
Moreover, we recall the $L^1$-nonexpansivity property of the rearrangement map $f \mapsto f^{\star}$, for which we refer to \cite[Proposition 1.2.1]{Kes}: for every $ f,g \in \mathcal{L}_{0}(\HH^{n}) $, it holds
\begin{equation}\label{nonexp-schw}
\left\|f^{\star}-g^{\star}\right\|_{L^1(\HH^{n})}
\leq \left\| f-g \right\|_{L^1(\HH^{n})} .
\end{equation}

\begin{lemma}\label{approx-rearrange-below}
Let $ f \in \mathcal{L}_{0}(\HH^{n}) $ and $ \{ f_k \} \subset  \mathcal{L}_{0}(\HH^{n}) $ be a sequence such that:
 \begin{itemize}
     \item $ - f^- \le f_k \le f^+ $ for all $ k \in \N $;
     \smallskip 
     \item $ \lim_{k \to \infty} f_k(x) = f(x) $ for a.e.~$ x \in \HH^n $.
 \end{itemize}
 Then $ \lim_{k \to \infty} f_k^\star(x) = f^\star(x) $ for every $ x \in \HH^n $.
\end{lemma}
\begin{proof}
See Appendix \ref{aux}, Subsection \ref{aux-1}.
\end{proof}

Since an important part of the paper is devoted to establishing comparison results between local integrals of functions, the following definition will play a fundamental role.

\begin{definition}[Concentration ordering]
 Let $f,g \in L^1(\HH^n)$. We say that $f$ is {less concentrated} than $g$, and we write $f \prec g$, if
 $$
 \int_{B_r} f^\star \, dV \leq \int_{B_r} g^\star \, dV \qquad \forall r>0 \, .
 $$
 \end{definition}
Clearly, this definition can be adapted to $L^1$ functions defined in a measurable subset $\Omega$ of $\HH^n$, by extending them to zero outside $\Omega$. The partial order relation $\prec$ is called \emph{comparison of mass concentrations} (or \emph{concentration comparison} for short), and it turns out to enjoy several remarkable properties that we summarize below (we refer to \cite[Proposition 2.1]{ALTa}).

\begin{proposition}
\label{Propconves}
Let $f,g \in L^1(\HH^{n})$. The following statements are equivalent:
\begin{enumerate}[(a)]
\item $f \prec g$;
\smallskip
\item \label{HL} For every nonnegative $ h \in L^{\infty}(\HH^{n})$ it holds
$$
\int_{\HH^{n}} f \, h \, dV \le \int_{\HH^{n}} g^\star h^\star \, dV \, ;
$$
\smallskip
\item \label{HLC} For every nonnegative $ h \in L^{\infty}(\HH^{n})$ it holds
$$
 \int_{\HH^{n}} f^\star h^\star \, dV \le  \int_{\HH^{n}} g^\star h^\star \, dV \, ;
$$
\smallskip
\item \label{LP} For every convex, nonnegative, locally Lipschitz function $H$, with $H(0)=0$, it holds 
$$
\int_{\HH^{n}}H(|f|) \, dV \leq \int_{\HH^{n}}H(|g|) \, dV \, . 
$$
\end{enumerate}
\end{proposition}
In particular, from Proposition \ref{Propconves}\eqref{LP} we immediately deduce that, if $f \prec g$, then
$$
\left\| f \right\|_{L^p(\HH^{n})} \le \left\| g \right\|_{L^p(\HH^{n})} \qquad \forall p \in[1,\infty] \, .
$$

Now, we switch to Sobolev properties of Schwarz rearrangements.
\begin{lemma} \label{lem-cpt-h1}
Let $ v \in H^1_c(\HH^n) $. Then $ v^\star \in H^1_c(\HH^n) $. If, in addition, $ v $ is Lipschitz, then $v^\star$ is also Lipschitz. 
\end{lemma}
\begin{proof}
See Appendix \ref{aux}, Subsection \ref{aux-1}.
\end{proof}

\begin{definition}[P\'olya-Szeg\H{o} inequality]\label{def-polya}
    We say that a noncompact model manifold $ \HH^n $ supports the P\'olya-Szeg\H{o} inequality if, for every $ v \in \hdot \cap \mathcal{L}_{0}(\HH^{n}) $, we have $ v^\star \in \hdot \cap \mathcal{L}_{0}(\HH^{n})$ and 
    \begin{equation}\label{P-Z}
      \int_{\HH^n} \left| \nabla v^\star \right|^2 dV  \le \int_{\HH^n} \left| \nabla v \right|^2 dV \, . 
    \end{equation}
\end{definition}

A weaker version of the above definition is the following \emph{radial} P\'olya-Szeg\H{o} inequality.

\begin{definition}\label{def-polya-radial}
    We say that a noncompact model manifold $ \HH^n $ supports the \emph{radial} P\'olya-Szeg\H{o} inequality if, for every {radial function} $ v \in \hdot \cap \mathcal{L}_{0}(\HH^{n}) $, we have $ v^\star \in \hdot \cap \mathcal{L}_{0}(\HH^{n})$ and \eqref{P-Z} holds.
\end{definition}

Clearly, an arbitrary radial function need not be nonincreasing, therefore \eqref{P-Z} is not for granted even in such a restricted setting. However, in this case we are able to provide a mild assumption on $\psi$ under which $ \HH^n $ supports the radial version of \eqref{P-Z} (see Proposition \ref{nazarov} below). The situation is completely different when dealing with the general inequality of Definition \ref{def-polya}, where there is currently no explicit condition available on $\psi$ ensuring its validity (see anyway Proposition \ref{iso-polya} and related comments below). 

In the following, we will see that \eqref{P-Z} can \emph{a priori} be required in a much smaller class of functions and, at the same time, \emph{a posteriori} holds in a larger class than $ v \in \hdot \cap \mathcal{L}_{0}(\HH^{n}) $. To this end, first of all we note that, from the very definition of Schwarz rearrangement, if $ v \in \mathcal{L}_{0}(\HH^{n}) $ then also $ v^\star \in \mathcal{L}_{0}(\HH^{n}) $, whereas in general it is not for granted that $ v \in \hdot $ entails $ v^\star \in \hdot $. However, the latter property is always true if $ v \in H^1_c(\HH^n) $ (regardless of the validity of the P\'olya-Szeg\H{o} inequality), because in that case Lemma \ref{lem-cpt-h1} ensures that $ v^\star $ also belongs to $ H^1_c(\HH^n) $ (hence \emph{a fortiori} to $ \hdot $).

\begin{proposition}\label{polya-compact-to-h1}
    Assume that the P\'olya-Szeg\H{o} inequality \eqref{P-Z} holds for every $ v \in H^1_c(\HH^n) \cap L^\infty(\HH^n) $ with $ v \ge 0 $. Then $ \HH^n $ supports the  P\'olya-Szeg\H{o} inequality in the whole $ \hdot \cap \mathcal{L}_{0}(\HH^{n}) $. 
\end{proposition}
\begin{proof}
See Appendix \ref{aux}, Subsection \ref{aux-1}.
\end{proof}

It is natural to ask whether the space $ \hdot \cap \mathcal{L}_{0}(\HH^{n}) $, in which we stated \eqref{P-Z}, coincides with the largest possible one for the inequality to make sense, that is, the space of functions with finite upper level sets, locally integrable, and with square integrable gradient. The answer is positive.

\begin{proposition}\label{all-grad}
Let $ v \in L^1_{\mathrm{loc}}(\HH^n) \cap  \mathcal{L}_{0}(\HH^{n}) $ with $ \nabla v \in L^2(\HH^n) $. Then $v \in \hdot \cap \mathcal{L}_{0}(\HH^{n}) $.  
\end{proposition}
\begin{proof}
See Appendix \ref{aux}, Subsection \ref{aux-1}.
\end{proof}

\subsection{Weak solutions of filtration equations}
Let $ \phi : [0,+\infty) \to [0,+\infty) $ be an arbitrary real function such that 
  \begin{equation}\label{cond-phi}
  \phi \text{ is continuous, nondecreasing, and nonconstant, with $\phi(0)=0 \, .$}
  \end{equation}
For future convenience, we set
\begin{equation}\label{primi}
\Phi(u) := \int_0^u \phi(w) \, dw \qquad \forall u \in [0,+\infty) 
\end{equation}
and
\begin{equation}\label{ell-lim}
    \ell := \lim_{u \to +\infty} \phi(u) \in (0,+\infty] \, .
\end{equation}
As in \cite[Subsection 2.2]{GIMP}, we notice that $\phi$ always admits two significant pseudo-inverse functions, defined by
\begin{equation}\label{pseudo}
\phi_l^{-1}(\rho) := \min \! \left\{ u \in [0,+\infty) : \  \phi(u) = \rho \right\} \qquad \forall \rho \in [0,\ell) \, ,
\end{equation}
and
\begin{equation}\label{pseudo-bis}
\phi_r^{-1}(\rho) := \max \! \left\{ u \in [0,+\infty) : \  \phi(u) = \rho \right\} \qquad \forall \rho \in [0,\ell) \, ,
\end{equation}
both of which coincide with $ \phi^{-1} $ when $\phi$ is strictly increasing. It is easy to check that $ \phi^{-1}_l $ is lower semicontinuous, strictly increasing, satisfies $ \phi^{-1}_l(0)=0 $, and acts as a right inverse, in the sense that $ \phi\!\left(\phi^{-1}_l(\rho)\right) = \rho  $ for all $ \rho \in [0,\ell) $. 
Clearly, all of these properties also hold up to $ \rho=\ell $ when $\phi$ is eventually constant. Similar considerations are true for $ \phi^{-1}_r $, except that it is upper semicontinuous and should be set equal to $ +\infty $ at $ \rho=\ell $ when $\phi$ is eventually constant.

For any nonnegative initial datum $ u_0 \in L^1(\HH^n) $ with $ \Phi(u_0) \in L^1(\HH^n) $, we study the Cauchy problem \eqref{filt-eqintr} and provide some fundamental well-posedness results, in agreement with the Euclidean theory  dealt with in \cite[Chapters 5 and 9]{Vaz}. The correct functional counterpart, for  a general model manifold $  \HH^n$, turns out to be the Bochner space introduced in Definition \ref{def-sob-parabolico}.  

\begin{definition}[Weak energy solutions]\label{weak-energy-sol}
 Let $ u_0 \ge 0 $ with $ u_0 , \Phi(u_0) \in L^1(\HH^n) $. We say that $ u \ge 0 $ is a {weak energy solution} of the Cauchy problem \eqref{filt-eqintr} if 
 $$
u \in L^1\!\left(\HH^n \times (0,T) \right) , \quad \phi(u) \in L^2\big( (0,T) ; \hdot \big) \qquad \forall T>0 \, ,
 $$
 and the identity
\begin{equation}\label{weak-form}
\int_0^{+\infty} \int_{\HH^n} u \, \partial_t \xi \, dV dt = \int_0^{+\infty} \int_{\HH^n} \left\langle \nabla \phi(u), \nabla \xi  \right\rangle dV dt - \int_{\HH^n} u_0(x) \, \xi(x,0) \, dV(x)
\end{equation}
holds for every $ \xi \in C^1_c\!\left( \HH^n \times [0,+\infty) \right) $.
\end{definition}

\begin{proposition}[Existence and uniqueness of weak energy solutions]\label{exuni-weak}
 Let $ u_0 \ge 0 $ with $ u_0 , \Phi(u_0) \in L^1(\HH^n) $. There exists a unique weak energy solution of the Cauchy problem \eqref{filt-eqintr}, in the sense of Definition \ref{weak-energy-sol}, which enjoys the following  additional properties:
 \begin{enumerate}[(i)]
     \item \label{cont-L1} 
   $u \in C\!\left([0,+\infty) ; L^1(\HH^n) \right)
     $ and if $ v $ is the weak energy solution associated to another $v_0 \ge 0$ with $ v_0 , \Phi(v_0) \in L^1(\HH^n)$, then
     \begin{equation}\label{contr-L1-p}
         \left\| \left( u(\cdot,t) - v(\cdot,t) \right)^+ \right\|_{L^1(\HH^n)} \le \left\| \left( u_0 - v_0 \right)^+ \right\|_{L^1(\HH^n)} \qquad \forall t >0 \, ;
     \end{equation}
     \smallskip
     \item \label{energy-est} The energy inequality 
     $$ \int_0^T \int_{\HH^n} \left| \nabla \phi(u) \right|^2 dV dt + \int_{\HH^n} \Phi(u(x,T)) \, dV(x) \le \int_{\HH^n} \Phi(u_0) \, dV \qquad \forall T>0 $$ 
     holds;
     \smallskip
      \item \label{contr-Lp} The time function 
     $$
     [0,+\infty) \ni t \mapsto \left\| u(\cdot,t) \right\|_{L^p(\HH^n)} \in [0,+\infty]
     $$
     is nonincreasing and right continuous for all $ p \in [1,\infty] $;
     \smallskip
     \item  \label{energy-decreas} $ \phi(u(\cdot,t)) \in \hdot $ for every $ t>0 $ and the time function
     $$
     [0,+\infty) \ni t \mapsto  \int_{\HH^n} \left| \nabla \phi(u(x,t)) \right|^2 dV(x) \in [0,+\infty]
     $$
is nonincreasing and right continuous on $ (0,+\infty) $, including $ t=0 $ if $ \phi(u_0) \in \hdot $;
\smallskip
     \item \label{rad-inc} If $  u_0 $ is a radially nonincreasing function, then $ u(\cdot,t) $ is also radially nonincreasing for all $t>0 $.
 \end{enumerate}
\end{proposition}

In order to prove Proposition \ref{exuni-weak} on the level of generality considered here, we need some technical auxiliary results about the construction of weak energy solutions, hence its proof is deferred until Section \ref{loc-parabolic}. We stress that the key property \eqref{rad-inc} holds on \emph{arbitrary model manifolds}, that is, no specific assumptions on $\psi$ and its derivatives are required. This is one of the main reasons why a substantial preliminary part is needed.

\begin{remark}[Limit solutions]\rm \label{limit-sol}
Thanks to the Cauchy estimate \eqref{contr-L1-p}, we can properly define the so-called \emph{limit solutions} for any nonnegative initial datum that merely lies in $ L^1(\HH^n) $, exactly as in \cite[Chapter 6]{Vaz}. More precisely, a limit solution associated with $ u_0 \in L^1(\HH^n) $ is any $ C\!\left([0,+\infty) ; L^1(\HH^n) \right) $ curve that is obtained as a limit of a sequence of weak energy solutions $\{ u_k \}$ taking initial data $ \{ u_{0k} \} $, complying with the requirements in Definition \ref{weak-energy-sol}, such that $ u_{0k} \to u_0 $ as $ k \to \infty $ in $ L^1(\HH^n) $; from \eqref{contr-L1-p}, it is readily seen that every nonnegative $ u_0 \in L^1(\HH^n) $ gives rise to a limit solution, and the latter is independent of the chosen sequence $\{ u_{0k} \}$. We stress that, in many important situations, limit solutions turn out to actually be \emph{weak solutions} (of finite energy for positive times), since an $ L^1 $--$ L^\infty $ \emph{smoothing effect} holds: see \emph{e.g.}~\cite{BGV, DMO, FM, GM-na}.
\end{remark}

\section{The concentration theorem and its consequences}\label{mainconcenttheo}

By virtue of Proposition \ref{exuni-weak}\eqref{rad-inc} we know, in particular, that the solution $ \overline{u}(\cdot,t) $ of the \emph{rearranged} Cauchy problem \eqref{filt-eq-symmintr} is always \emph{radially nonincreasing}, regardless of the validity of the P\'olya-Szeg\H{o} inequality. Our main result, which we state next, asserts that it is \emph{more concentrated} than  the original solution (via its Schwarz rearrangement) if and only if the P\'olya-Szeg\H{o} inequality holds.

\begin{theorem}[The concentration comparison]\label{th-conc}
Let $\HH^n$ be any $n$-dimensional complete and noncompact model manifold. Let $ \phi:[0,+\infty) \to [0,+\infty) $ be an arbitrary real function that fulfills \eqref{cond-phi}, and let $ \Phi $ be its primitive according to \eqref{primi}. Then the following properties are equivalent: 
\begin{enumerate}[(a)]
    \item \label{TA} $ \HH^n $ supports the P\'olya-Szeg\H{o} inequality, in the sense of Definition \ref{def-polya}; 

    \medskip 

    \item \label{TB} For every $ u_0 \ge 0 $ with $ u_0 , \Phi(u_0) \in L^1(\HH^n) $, it holds
    \begin{equation}\label{conc-prop}
      {u}^\star(\cdot , t)  \prec \overline{u}(\cdot, t) \qquad \forall t > 0 \, ,
    \end{equation}
    where $u$ (resp.~$ \overline{u} $) is the weak energy solution of the Cauchy problem \eqref{filt-eqintr} (resp.~\eqref{filt-eq-symmintr}), in the sense of Definition \ref{weak-energy-sol}, and $  u^\star$ is its (spatial) Schwarz rearrangement according to \eqref{from-ast-to-star}.
\end{enumerate}
\end{theorem}

\begin{remark}[Equivalence between linear and nonlinear concentration]\rm \label{lin-nonlin}
Because the fun\-ction $ \phi(u)=u $ is clearly admissible, which gives rise to the \emph{heat equation} in $ \HH^n $, we immediately deduce that if the concentration inequality \eqref{conc-prop} is satisfied by the heat flow, then it is also satisfied by \emph{all} the nonlinear filtration flows \eqref{filt-eqintr}, for every $\phi$ complying with \eqref{cond-phi}. Upon introducing the \emph{heat kernel} $ K(x,y,t) $ of $ \HH^n $, this means that 
\begin{equation}\label{heat}
 \left( \int_{\HH^n} K(\cdot,y,t) \, u_0(y) \, dV(y) \right)^\star \prec \int_{\HH^n} K(\cdot,y,t) \, u_0^\star(y) \, dV(y) \qquad \forall t > 0 \, ,
\end{equation}
for every nonnegative $ u_0 \in L^1(\HH^n) $. Note that, in $\R^n$, the equivalence between \eqref{heat} and the P\'olya-Szeg\H{o} inequality had already been observed in \cite[Lemma 7.17]{LL}. 
\end{remark}

\begin{remark}[Limit solutions and concentration comparison]\rm \label{limit-sol-conc}
The implication \eqref{TA}$ \Rightarrow $\eqref{TB} is actually more general. Indeed, it also applies to \emph{limit solutions} for every nonnegative $ u_0 \in L^1(\HH^n) $, since the latter are obtained by definition as limits in $ L^1(\HH^n) $ of weak energy solutions (recall Remark \ref{limit-sol}), and the concentration inequality \eqref{conc-prop} is stable under $ L^1(\HH^n) $ convergence thanks to \eqref{nonexp-schw}. 
\end{remark}

\begin{remark}[Compact manifolds]\rm \label{cpt}
In order to keep our approach as much as possible unified, we state and prove our main results, along with auxiliary intermediate results, under the assumption that $ \HH^n $ is \emph{complete} and \emph{noncompact}, which amounts to requiring that the function $ \psi $ appearing in \eqref{metric-model} is defined in the whole $ [0,+\infty) $. In case $\psi $ is only defined on $ [0,R_0] $ for some $ R_0 \in (0,+\infty) $, then either $ \HH^n $ is a compact manifold with boundary (if $ \psi(R_0)>0 $) or it is a compact manifold without boundary (if $ \psi(R_0)=0 $ and the analogues \eqref{prop-psi} of hold). In both frameworks, minor modifications to our strategy ensure that Theorem \ref{th-conc} is still true. We just point out that, if $ \partial \HH^n  \neq \emptyset $, the Cauchy problems \eqref{filt-eqintr}--\eqref{filt-eq-symmintr} should be completed with \emph{Dirichlet} boundary conditions, and $ \hdot $ should be replaced by $ H^1_0(\HH^n) $, \emph{i.e.}~the space of $ H^1(\HH^n) $ functions vanishing on $ \partial \HH^n $. 
\end{remark}

It is natural to look for suitable conditions ensuring that the P\'olya-Szeg\H{o} inequality \eqref{P-Z}, and therefore the thesis of Theorem \ref{th-conc}\eqref{TB}, holds. This is always the case if $ \HH^n $ supports a \emph{centered isoperimetric inequality}, that is,
\begin{equation}\label{iso-1}
\mathrm{Per} (B_r)\le \mathrm{Per}(\Omega) \qquad \forall \Omega \in \mathscr{B}_b(\HH^n) \, ,
\end{equation}
where we recall that $ B_r \equiv B_r(o) $ is the geodesic ball \emph{centered} at the pole $o$, of radius $ r>0 $, having the same volume measure $ V $ as $  \Omega $, and $ \mathscr{B}_b(\HH^n) $ stands for the class of all bounded Borel sets in $ \HH^n $. Here we adopt the notation $ \mathrm{Per}(\cdot) $ for the perimeter functional in the usual dense of De Giorgi, see \emph{e.g.}~\cite[Subsection 1.2]{MS}. In view of \eqref{eq-G-vol}, \eqref{eq-G}, and \eqref{E1}, we observe that \eqref{iso-1} is equivalent to  
\begin{equation*}\label{iso-2}
\omega_n \left[ \psi\!\left(  G^{-1}\!\left( \tfrac{V(\Omega)}{\omega_n} \right) \right) \right]^{n-1} \le \mathrm{Per}(\Omega) \qquad \forall \Omega \in \mathscr{B}_b(\HH^n) \, .
\end{equation*}

\begin{proposition}[A sufficient condition for the P\'olya-Szeg\H{o} inequality]\label{iso-polya}
Let $\HH^n$ be any $n$-dimensional complete and noncompact model manifold. Suppose that the centered isoperimetric inequality \eqref{iso-1} holds. Then $ \HH^n $ supports the P\'olya-Szeg\H{o} inequality, in the sense of Definition \ref{def-polya}.
\end{proposition}

Currently, such a condition is known to be satisfied by the Euclidean space $ \mathbb{R}^n $, the hyperbolic space $ \mathbb{H}^n $ and, in the compact case, the sphere $ \mathbb{S}^n $; for an excellent account on such results, we refer the reader to \cite{Fu}, see in particular Section 5.5 there. Since the proof of Proposition \ref{iso-polya} is rather classical, we defer it to Appendix \ref{aux-3-bis}.

A general result on model manifolds is still missing, but it is worth mentioning that a key, necessary condition, seems to be the \emph{decreasing monotonicity} of the \emph{scalar curvature} as a function of the distance to $ o $ (recall identity \eqref{scal-curv}). Indeed, a remarkable result by Brendle \cite[Theorem 1.4]{Bre13} establishes that, under such an assumption plus some further technical hypotheses, the geodesic spheres $ \partial B_r(o) $ are the unique hypersurfaces with \emph{constant mean curvature} (it is the manifold counterpart of a Euclidean result known in the literature as \emph{Alexandrov theorem}). In particular, since isoperimetric sets always have constant mean curvature, this appears to be a strong indication that such manifolds should support a centered isoperimetric inequality. Note that, in a very recent paper, Maggi and Santilli have extended Brendle's theorem (actually in the ``Schwarzschild'' version \cite[Theorem 1.1]{Bre13}) to rougher sets satisfying the constant mean curvature condition in a weak sense only \cite[Theorem 1.1]{MS}. 

Since the nonlinear flow \eqref{filt-eqintr} preserves radiality, regardless of the monotonicity of $u_0$, it is not difficult to obtain an analogue of Theorem \ref{th-conc} restricted to the sole radial setting. 
\begin{theorem}[The \emph{radial} concentration comparison]\label{th-conc-rad}
Let the hypotheses of Theorem \ref{th-conc} hold. Then the following properties are equivalent: 
\begin{enumerate}[(a)]
    \item \label{TA-rad} $ \HH^n $ supports the radial P\'olya-Szeg\H{o} inequality, in the sense of Definition \ref{def-polya-radial}; 

    \medskip 

    \item \label{TB-rad} For every radial $ u_0 \ge 0 $ with $ u_0 , \Phi(u_0) \in L^1(\HH^n) $, it holds
    \begin{equation*}\label{conc-prop-rad}
      {u}^\star(\cdot , t)  \prec \overline{u}(\cdot, t) \qquad \forall t > 0 \, ,
    \end{equation*}
    where $u$ (resp.~$ \overline{u} $) is the weak energy solution of the Cauchy problem \eqref{filt-eqintr} (resp.~\eqref{filt-eq-symmintr}), in the sense of Definition \ref{weak-energy-sol}, and $  u^\star$ is its (spatial) Schwarz rearrangement according to \eqref{from-ast-to-star}.
\end{enumerate}
\end{theorem}

The advantage of restricting the concentration comparison theorem to radial data (and solutions) is that, as the next proposition shows, in this case some semi-explicit conditions are available guaranteeing that a given model function $\psi$ supports the radial version of the P\'olya-Szeg\H{o} inequality. 

\begin{proposition}[A sufficient condition for the \emph{radial} P\'olya-Szeg\H{o} inequality]\label{nazarov}
Let $\HH^n$ be any $n$-dimensional complete and noncompact model manifold. Suppose that the corresponding model function $\psi$ has the following property:
\begin{equation}\label{prop-nazarov}
\begin{gathered}
\psi(R)^{n-1} \le \psi(S)^{n-1} + \psi(T)^{n-1}   \\[0.1cm]
\text{for all $ R,S,T>0 $ such that} \\[0.1cm]
\int_0^{R} \psi(r)^{n-1} \, dr = \int_0^{S} \psi(r)^{n-1} \, dr - \int_0^{T} \psi(r)^{n-1} \, dr \, .
\end{gathered}
\end{equation}
Then $\HH^n$ supports the radial P\'olya-Szeg\H{o} inequality, in the sense of Definition \ref{def-polya-radial}.
\end{proposition}

In terms of the radial volume function $G$ defined in \eqref{eq-G}, we observe that \eqref{prop-nazarov} is equivalent to asking that the function $ \left[\psi \circ G^{-1}\right]^{n-1} $ satisfies the following inequality:
\begin{equation}\label{prop-nazarov-G}
\begin{gathered}
\left[\psi\!\left( G^{-1}\!\left(\mu\right)  \right) \right]^{n-1} \le  \left[\psi\!\left( G^{-1}\!\left( \mu + \nu \right)  \right) \right]^{n-1} + \left[\psi\!\left( G^{-1}\!\left(\nu\right) \right) \right]^{n-1} \qquad 
\forall \mu,\nu>0  : \ \mu+\nu < \tfrac{V(\HH^n)}{\omega_n} \, . 
\end{gathered}
\end{equation}

\begin{remark}[Nondecreasing $\psi$ -- I] \label{incr-psi}\rm 
   It is straightforward to check that \eqref{prop-nazarov} is always satisfied by \emph{nondecreasing} functions. Indeed, the bottommost identity in \eqref{prop-nazarov} readily implies that $ S>R $, therefore the uppermost inequality becomes trivial in this case. Note that such a property is in fact a purely \emph{radial} version of the centered isoperimetric inequality, that is, it corresponds precisely to \eqref{iso-1} when $ \Omega $ is an  arbitrary centered \emph{annulus} having the same volume as $B_r$.
In particular, it is significantly weaker than \eqref{iso-1}; for instance, one can easily construct increasing model functions $\psi$ that fall within the class of the next Theorem \ref{polya-failure} where, as we will see, the general P\'olya-Szeg\H{o} inequality fails (see comments below). 
\end{remark}

Finally, we are able to provide a wide class of model manifolds $ \HH^n $ in which the P\'olya-Szeg\H{o} inequality, and therefore the concentration comparison \eqref{conc-prop}, \emph{fails}. Remarkably, this is always the case if the \emph{scalar curvature} $ S(x) $ is \emph{strictly increasing} as a function of $ r \equiv d(x,o) $ (see Figure \ref{profiles} above).

\begin{theorem}[Failure of the P\'olya-Szeg\H{o} inequality]\label{polya-failure}
    Let $ \HH^n $ be any $n$-dimensional mo\-del manifold. Suppose that there exists some $ \hat{o} \in \HH^n $ such that $ S(\hat{o}) > S(o) $. Then $ \HH^n $ does not support the P\'olya-Szeg\H{o} inequality, in the sense of Definition \ref{def-polya}.
\end{theorem}

\begin{remark}[Nondecreasing $\psi$ -- II]\rm 
Note that any model function $\psi$ such that (for instance) $ \psi''>0 $ everywhere on $ [0,+\infty) $, $\psi''(r)=\mathsmaller{\mathcal{O}} (1)$, $\psi'(r)=\mathcal{O}(1)$, $ \psi(r) = \mathcal{O}(r) $ as $r \to +\infty$ satisfies $ S(o)<0 $, $ \lim_{r \to +\infty} S(r) = 0 $ (recall \eqref{scal-curv}), and is strictly increasing, therefore it falls within the class of Proposition \ref{nazarov} but at the same time fulfills the hypotheses of Theorem \ref{polya-failure}. These simple examples show that the radial P\'olya-Szeg\H{o} inequality can actually hold when the general P\'olya-Szeg\H{o} inequality fails.
\end{remark}

\begin{figure}
	\centering 
	\includegraphics[scale=0.85]{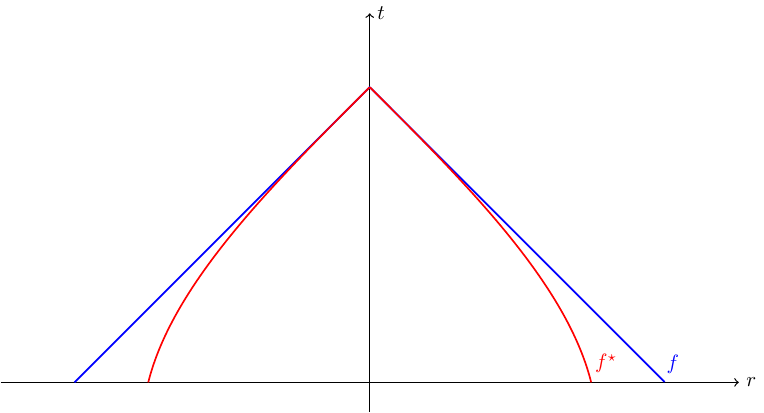}
	\caption{A visual comparison between the function $ f(x) = \left( 1-\mathrm{d}(x,\hat{o}) \right)^+ $ and a third-order approximation of its Schwarz rearrangement $f^\star$ about the pole $o$, in a setting where $ S(\hat{o})>S(o) $. The two profiles are plotted as radial functions of the distance from $\hat{o}$ and $o$, respectively. Note that if instead $ S(\hat{o})=S(o) $ the two profiles should coincide, whereas in this case the graph of $f^\star$ is steeper, which accounts for the fact that the $L^2$ norm of its gradient is \emph{larger} than the one of the original function $ f$, contradicting the P\'olya-Szeg\H{o} inequality. For the analytical details, see the proof of Theorem \ref{polya-failure}, and in particular formula \eqref{graph-rt}.}
    \label{profiles}
    \end{figure}

\section{Local semilinear elliptic Dirichlet problems}\label{loc-elliptic}

In order to prove Theorem \ref{th-conc}, we will proceed by several steps. The key idea is to construct weak energy solutions via a by now common \emph{time discretization}, which implies dealing with \emph{semilinear elliptic} problems first. We refer \emph{e.g.}~to \cite[Chapter 10]{Vaz} for a detailed description of this approach. The upside, as explained in the Introduction, is that exploiting the P\'olya-Szeg\H{o} inequality to establish concentration properties is much easier for elliptic problems than parabolic ones.

Hence, let us now focus on the following semilinear elliptic Dirichlet problem:
\begin{equation}\label{ellip-1}
\begin{cases}
-\Delta v + \beta(v) = f & \text{in } B_R \, , \\
v=0 & \text{on } \partial B_R \, ,
\end{cases}
\end{equation}
where $R>0$, $ f  $ is a nonnegative function in $B_R$,  and $ \beta  $ is a nonnegative, nondecreasing, and (at least) continuous real function on $[0,+\infty)$.

Here and in the sequel, we take for granted that $ \HH^n $ is a general complete and noncompact model manifold whose corresponding model function is $\psi$.  

\begin{definition}\label{weak-sol}
	Let  $ \beta: [0,+\infty) \to [0,+\infty) $ be a continuous and nondecreasing function with $ \beta(0)=0 $. Let  $ f \in L^\infty\!\left( B_R \right) $ with $ f \ge 0 $. We say that  $ v \ge 0 $ is a weak solution to \eqref{ellip-1} if $ v \in H^1_0\!\left( B_R \right) \cap L^\infty\!\left( B_R \right) $ and the identity
	\begin{equation}\label{weak-sol-ell}
	\int_{B_R}  \left\langle \nabla v , \nabla w \right\rangle  dV +\int_{B_R}  \beta(v) \, w \,  dV = \int_{B_R}  f \, w \,  dV  
	\end{equation}
	holds for every $ w \in H^1_0\!\left( B_R \right)$.
\end{definition}
For practical purposes, we set
\begin{equation} \label{defB}
\mathsf{B}(v) := \int_0^v \beta(\upsilon) \, d\upsilon  \qquad \forall v \in \mathbb{R} \, ,
\end{equation}
where we implicitly define $ \beta(\upsilon) = - \beta(-\upsilon) $ for $ \upsilon < 0 $. Note that in this way $\mathsf B$ becomes a nonnegative, convex, and $ C^1(\mathbb{R}) $ function. Our first goal is to provide a variational characterization of \eqref{ellip-1} and study qualitative and quantitative properties of weak solutions. The result we present is rather classical, but for the reader's convenience we provide a complete proof in the Appendix. 

\begin{proposition}\label{weak-sol-min}
Let  $ \beta: [0,+\infty) \to [0,+\infty) $ be a continuous and nondecreasing function with $ \beta(0)=0 $. Let $ f \in L^\infty\!\left( B_R \right) $ with $ f \ge 0 $. Then there exists a unique weak solution $v_0$ of the Dirichlet problem \eqref{ellip-1}, in the sense of Definition \ref{weak-sol}, which can be characterized as 
	\begin{equation}\label{argmin}
v_0 = \underset{  v \in H^1_0(B_R)}{ \operatorname{argmin}} \left\{ \frac 1 2  \int_{B_R}  \left| \nabla v \right|^2  dV +\int_{B_R}  \mathsf{B}(v) \,  dV - \int_{B_R}  f \, v \,  dV  \right\} . 
\end{equation}
Specifically, we have 
$$ 
\left\| v_0 \right\|_{L^\infty(B_R)} \le C 
$$ 
for a suitable constant $C>0$ depending only on $ V(B_R), \| f \|_{\infty},p,C_{p,R}  $, where $ p>2 $ is any fixed exponent as in Proposition \ref{local sob} and $ C_{p,R} $ is the corresponding Sobolev constant appearing in \eqref{loc-sob}. 
\end{proposition}
\begin{proof}
See Appendix \ref{aux}, Subsection \ref{aux-2}.
\end{proof}

The next lemma requires that $\beta$ is also smooth with $ \beta' $ bounded away from zero, a technical assumption that it necessary to rigorously prove some crucial bounds on the weak solutions provided by Proposition \ref{weak-sol-min} and the key preservation of \emph{radial nonincreasing monotonicity}.
\begin{lemma}\label{ulteriori-prop-approx}
Let  $ \beta: [0,+\infty) \to [0,+\infty) $ be a strictly increasing smooth function with $ \beta(0)=0 $ and 
$$
\inf_{v \ge 0} \beta'(v) =: \mathcal{I} >0 \, .
$$
Let $ f \in C^\infty\!\left( \overline{B}_R \right) $ with $ f \ge 0 $. Then the weak solution $v_0$ to \eqref{ellip-1} enjoys the following additional properties:
\begin{enumerate}[(i)]

    \item \label{pA} For all $ p \in [1,\infty] $ it holds
    \begin{equation}\label{Lp-nonexp}
      \left\| \beta(v_0) \right\|_{L^p(B_R)}  \le  \left\| f \right\|_{L^p(B_R)} ;
    \end{equation}
\smallskip
    \item \label{pB} If $ v_1 $ is the weak energy solution to \eqref{ellip-1} with $f$ replaced by another function $ g \in C^\infty\!\left(\overline{B}_R\right) $ with $g \ge 0$, then 
        \begin{equation}\label{L1-contr-A}
      \left\| \left( \beta(v_0) - \beta(v_1) \right)^+ \right\|_{L^1(B_R)}  \le  \left\| \left( f - g \right)^+ \right\|_{L^1(B_R)} ;
    \end{equation}
    \smallskip
    \item \label{noninc} There exists a positive constant $ c_R $, depending only on $ n, \psi, R $, such that if  
    \begin{equation}\label{const-below}
      \mathcal{I}  \ge c_R
    \end{equation}
    then $ v_0 $ is radially nonincreasing whenever $f$ is.
    
\end{enumerate}
\end{lemma}
\begin{proof}
Let us proceed with the same order in which the properties are stated.

\noindent \eqref{pA} For $ p \in (1,\infty) $, it is enough to plug the test function $ w=\beta(v_0)^{p-1} $ into \eqref{weak-sol-ell}, which yields 
$$
	(p-1) \, \int_{B_R}  \left| \nabla v_0 \right|^2 \beta'(v_0) \, \beta(v_0)^{p-2} \, dV +\int_{B_R}  \beta(v_0)^p \,  dV = \int_{B_R}  f \, \beta(v_0)^{p-1}  \,  dV  \, ,
$$
so that \eqref{Lp-nonexp} follows upon observing that the first term on the left-hand side is nonnegative and applying H\"older's inequality with exponents $ p $ and $ p' $ on the right-hand side. Rigorously, the computation is justified for $ p \ge 2 $, however, it can also be settled for $ p \in (1,2) $ by means of a routine approximation of $ \beta^{p-1} $ with increasing smooth functions, for instance $ (\beta+\varepsilon)^{p-1} - \varepsilon^{p-1} $ as $ \varepsilon \to 0^+ $. Finally, the borderline cases $ p=1,\infty $ can be reached just by taking the corresponding limits in \eqref{Lp-nonexp}.

\noindent \eqref{pB} First of all, we write the weak formulation satisfied by $ (v_0-v_1) $: 
	\begin{equation}\label{weak-1}
\int_{B_R}  \left\langle \nabla \!\left( v_0 - v_1 \right) , \nabla w \right\rangle  dV + \int_{B_R} \left(  \beta(v_0) - \beta(v_1) \right) w \,  dV = \int_{B_R}  \left(f - g\right)  w \,  dV \, ,
\end{equation}
for every $ w \in H^1_0(B_R) $. Then we pick a sequence of smooth real functions $ J_k : [0,+\infty) \to [0,+\infty) $ approximating the derivative of the positive part (see also \cite[Section 3.2]{Vaz}), namely such that: 
$$
J_k(0) = 0 \, , \qquad J_k' \ge 0 \, , \qquad 0 \le J_k \le 1 \, , \qquad \lim_{k \to \infty} J_k (v) = \chi_{(0,+\infty)}(v) \quad  \forall v \in [0,+\infty) \, ,
$$
and plug the test function $ w = J_k\!\left(\left( v_0 - v_1 \right)^+\right) $ into \eqref{weak-1}, obtaining: 
	\begin{equation*}\label{weak-2}
	\begin{aligned}
& \, \int_{B_R} J'_k\!\left(\left( v_0 - v_1 \right)^+\right) \left| \nabla\! \left( v_0 - v_1 \right)^+ \right|^2  dV +\int_{B_R} \left(  \beta(v_0) - \beta(v_1) \right) J_k\!\left(\left( v_0 - v_1 \right)^+\right)  dV \\
 = & \, \int_{B_R}  \left( f - g \right)  J_k\!\left(\left( v_0 - v_1 \right)^+\right)   dV \, ,
\end{aligned}
\end{equation*}
which, upon letting $ k \to \infty $ and exploiting the monotonicity of $ \beta $, yields  
	\begin{equation*}\label{weak-3}
\int_{B_R} \left(  \beta(v_0) - \beta(v_1) \right)^+  dV	\le \int_{B_R}  \left(f-g\right)^+  dV \, ,
	\end{equation*}
	that is \eqref{L1-contr-A}. 

\noindent \eqref{noninc} Since $f$ is radial and $ \HH^n $ is spherically symmetric, it is a standard fact that $ v_0 $ is also radial. Indeed, one can repeat the proof of Proposition \ref{weak-sol-min} with the additional constraint that the $ \operatorname{argmin} $ in  \eqref{argmin} is taken over all \emph{radial} functions in $ H^1_0(B_R) $. The differential equation satisfied by $ v_0 \equiv v_0(r) $ then reads
\begin{equation}\label{eq-rad-dec}
-v_0'' - (n-1) \, \frac{\psi'}{\psi} \,  v_0' + \beta(v_0)= f \, .
\end{equation}
Alternatively, one can solve \eqref{eq-rad-dec} completed with the initial conditions $  v_0'(0)=0 $ and $ v_0(0)=\alpha>0 $, showing via ODE shooting methods that $ v_0(R)=0 $ for a unique $ \alpha=\alpha_R $. Because $ v_0 $ is a smooth function on $ \HH^n $, we have that $ v_0(r) $ is also smooth on $ [0,+\infty) $ and complies with $ v_0^{(2k+1)}(0)=0 $ for all $ k \in \mathbb{N} $. In particular, by taking the  derivative of \eqref{eq-rad-dec} in $r$, we obtain the identity
\begin{equation}\label{eq-rad-dec-prime}
-\left(v_0'\right)'' - (n-1) \, \frac{\psi'}{\psi} \left( v_0' \right)' + \left[ -(n-1) \, \frac{\psi'' \psi - \left( \psi' \right)^2}{\psi^2} + \beta'(v_0) \right]v_0' = f' \le 0 \, ,
\end{equation}
at least on $ (0,+\infty) $. In order to prove that $  v_0' \le 0 $, let us set $ w_0 :=  v_0' $, multiply \eqref{eq-rad-dec-prime} by $ w_0^+ \psi^{n-1} $ and integrate from an arbitrary $ 0 < \epsilon< R $ to  $ R $; this leads to 
\begin{equation}\label{deriv-ineq}
\begin{aligned}
- w_0'(R) \, w_0^+(R) \, \psi(R)^{n-1} + w_0'(\epsilon) \, w_0^+(\epsilon) \, \psi(\epsilon)^{n-1} & \\
+ \int_{\epsilon}^R \left[ \left(w_0^+\right)' \right]^2 \psi^{n-1} \, dr 
 + \int_{\epsilon}^R \left[ -(n-1) \, \frac{\psi'' \psi - \left( \psi' \right)^2}{\psi^2} + \beta'(v_0) \right] \left( w_0^+\right)^2 \psi^{n-1} \, dr & \le 0 \, .
\end{aligned}
\end{equation}
If we define the constant $c_R$ in \eqref{const-below} as
$$
c_R = (n-1) \, \sup_{r \in (0,R)} \left( \frac{\psi'' \psi - \left( \psi' \right)^2} {\psi^2}\right)^+ + 1 \, ,
$$
which is certainly finite due to \eqref{prop-psi}, then from \eqref{deriv-ineq} and the fact that $ w_0^+(R)=0 $ (simple consequence of the nonnegativity of $v_0$ and its vanishing on $\partial B_R$), we end up with
$$
 w_0'(\epsilon) \, w_0^+(\epsilon) \, \psi(\epsilon)^{n-1} + \int_{\epsilon}^R \left( w_0^+\right)^2 \psi^{n-1} \, dr \le 0 \, .
$$
 The thesis then follows upon letting $ \epsilon \to 0^+ $, which yields $ w_0^+ = \left( v_0'\right)^+ = 0 $ identically since the leftmost term vanishes.   
\end{proof}

By approximating a general function $\beta$ as in Definition \ref{weak-sol} with smooth functions falling within the hypotheses of Lemma \ref{ulteriori-prop-approx}, we can show that the $L^p$ stability properties are preserved.

\begin{proposition}\label{weak-sol-comparison}
Let  $ \beta: [0,+\infty) \to [0,+\infty) $ be a continuous and nondecreasing function with $ \beta(0)=0 $. Let $ f,g \in L^\infty\!\left( B_{R} \right) $ with $ f,g \ge 0 $. Then properties \eqref{Lp-nonexp} and \eqref{L1-contr-A} continue to hold. Moreover, if $ \hat{g} \in L^\infty\!\left( B_{S} \right) $ with $\hat g \ge 0$ and $ \hat g \ge f $  in $B_R \subset B_S $, and we denote by $ \hat{v}_1 $ the weak solution to \eqref{ellip-1} with $ B_{S} $ in the place of $ B_R $ and $ \hat g $ in the place of $ f $, it holds 
\begin{equation}\label{ee-2}
  v_0  \le \hat{v}_1  \qquad \text{in } B_R \, .
\end{equation}
\end{proposition}
\begin{proof}
Let us begin with \eqref{ee-2}. The weak formulation satisfied by $ v_0 - \hat v_{1} $, in $ B_R $, entails
	\begin{equation*}\label{weak-bis}
\int_{B_R} \left\langle \nabla \!\left( v_0 - \hat v_{1} \right) , \nabla w \right\rangle  dV + \int_{B_R} \left(  \beta(v_0) - \beta(\hat v_{1}) \right) w \,  dV \le  0
\end{equation*}
for all nonnegative functions $ w \in H^1_0(B_R) $, so that by taking $ w = \left( v_0 - \hat v_{1} \right)^+ $ we end up with 
\begin{equation*}\label{weak-ter}
\int_{B_R} \left| \nabla \!\left( v_0 - \hat v_{1} \right)^+ \right|^2  dV + \int_{B_R} \left(  \beta(v_0) - \beta(\hat v_{1}) \right) \left( v_0 - \hat v_{1} \right)^+  dV = 0 \, ,
\end{equation*}
which yields $ \left( v_0 - \hat v_{1} \right)^+ = 0 $, that is, $ v_0 \le \hat v_{1} $ in $ B_R $.

In order to prove that \eqref{Lp-nonexp} and \eqref{L1-contr-A} are still satisfied, let us take a sequence of nonlinearities $ \{ \beta_k \} $ satisfying all of the assumptions of Lemma \ref{ulteriori-prop-approx}, with $ \mathcal{I} \equiv \mathcal{I}_k $ possibly vanishing as $ k \to \infty $, such that
$$
\beta_k \underset{k \to \infty}{\longrightarrow} \beta \qquad \text{locally uniformly in } [0,+\infty) \, ,
$$
and a uniformly bounded sequence $ \{ f_k \} $ of nonnegative and smooth functions in $ \overline{B}_R $ that converges pointwise almost everywhere to $ f $. If we let $ \{ v_k \} $ denote the corresponding sequence of solutions to \eqref{ellip-1} (with $\beta \equiv \beta_k $ and $ f \equiv f_k $), owing to Proposition \ref{weak-sol-min} it is readily seen that $ \{ v_k \} $ is bounded in $ H^1_0(B_R) \cap L^\infty(B_R) $, thus it converges strongly to $v_0$ at least in $L^p(B_R) $ for all $ p \in [1,\infty) $, and this is enough to pass to the limit in the $ k $-versions of \eqref{Lp-nonexp} and \eqref{L1-contr-A} (upon exploiting similar approximations for $v_1$ with the same $\{ \beta_k \}$ and another smooth sequence $ \{ g_k \} $ converging to $g$).
\end{proof}

The next result establishes that, if $\HH^n$ supports the P\'olya-Szeg\H{o} inequality, then property \eqref{noninc} in Lemma \ref{ulteriori-prop-approx} actually holds in full generality.
\begin{corollary} \label{decr-min}
	Let the hypotheses of Proposition \ref{weak-sol-min} hold. Suppose in addition that $ \HH^n $ supports the P\'olya-Szeg\H{o} inequality, in the sense of Definition \ref{def-polya}, and that $f$ is radially nonincreasing. Then the weak solution to \eqref{ellip-1}, in the sense of Definition \ref{weak-sol}, is a $ C^1\!\left( \overline{B}_R \right) $ radially nonincreasing function. 
\end{corollary}
\begin{proof}
See Appendix \ref{aux}, Subsection \ref{aux-2}.
\end{proof}

\section{Local and global nonlinear parabolic problems}\label{loc-parabolic}

In order to construct solutions to \eqref{filt-eqintr} in the whole of $ \HH^n $, we adopt a common approximation strategy and consider, at first, \emph{local Cauchy-Dirichlet problems} on geodesic balls of radii $R>0$ centered at $o$:
  \begin{equation}\label{filt-eq-local}
      \begin{cases} 
          \partial_t u_R = \Delta \phi(u_R) & \text{in } B_R \times (0,+\infty) \, , \\
          \phi(u_R) = 0 & \text{on } \partial B_R \times (0,+\infty)  \, , \\
          u_R = u_0 & \text{on } B_R \times \{ 0 \} \, .
      \end{cases}
  \end{equation}
  Since it is enough for our purposes, we limit ourselves to dealing with \emph{bounded} data and solutions.

\begin{definition}\label{weak-energy-sol-loc}
 Let $ u_0 \ge 0 $ with $ u_0  \in L^\infty(B_R) $. Let $ \phi $ comply with \eqref{cond-phi}. We say that $ u_R \ge 0 $ is a weak solution to \eqref{filt-eq-local} if, for all $T>0$, we have 
 $$
u_R \in L^\infty\!\left(B_R \times (0,T) \right) , \quad \phi(u_R) \in L^2\left( (0,T) ; H^1_0(B_R) \right) ,
 $$
 and the identity
\begin{equation}\label{weak-form-loc}
\int_0^{T} \int_{B_R} u_R \, \partial_t \xi \, dV dt = \int_0^{T} \int_{B_R} \left\langle \nabla \phi(u_R), \nabla \xi  \right\rangle dV dt - \int_{B_R} u_0(x) \, \xi(x,0) \, dV(x)
\end{equation}
holds for every $ \xi \in H^1\!\left( (0,T) ; L^2(B_R) \right) \cap  L^2\!\left( (0,T) ; H^1_0(B_R) \right) $ such that $ \xi(\cdot,T) = 0 $.
\end{definition}

There are by now many approaches one can exploit to construct solutions of problem \eqref{filt-eq-local}, all of which involve some approximation procedure (see \emph{e.g.}~\cite[Chapters 5, 6, and 10]{Vaz}). In the next two propositions we collect the main properties that such solutions enjoy, as a consequence of the fact that different constructions converge to the \emph{same} weak solution, the latter being uniquely identified. In order to avoid further technicalities and prove concentration properties of solutions in a more convenient way, we will first consider \emph{bijective} nonlinearities only (Proposition \ref{exuni-weak-loc}), recovering the whole class of $\phi$ satisfying \eqref{cond-phi} with an extra approximation step (Proposition \ref{approx-phireg}).

\begin{proposition}\label{exuni-weak-loc}
 Let $ u_0 \ge 0 $ with $ u_0\in L^\infty(B_R) $. Let $ \phi: [0,+\infty) \to [0,+\infty) $ be a continuous and increasing bijection. Then there exists a unique weak solution $u_R$ of the Cauchy-Dirichlet problem \eqref{filt-eq-local}, in the sense of Definition \ref{weak-energy-sol-loc}. Moreover, such a solution enjoys the following  additional properties: 
 \begin{enumerate}[(i)]
     \item \label{cont-L1-loc} 
   $u_R \in C\!\left([0,+\infty) ; L^1(B_R) \right)$, and if $ v_R $ is the weak solution to \eqref{filt-eq-local} associated with another $v_0 $ in the place of $ u_0 $, then
     \begin{equation}\label{contr-L1-p-loc}
         \left\| \left( u_R(\cdot,t) - v_R(\cdot,t) \right)^+ \right\|_{L^1(B_R)} \le \left\| \left( u_0 - v_0 \right)^+ \right\|_{L^1(B_R)} \qquad \forall t >0 \, ;
     \end{equation} 
\smallskip
\item  \label{comp-R} If $ v_{R+1} $ is the weak solution to \eqref{filt-eq-local} with $B_{R+1}$ in the place of $ B_R $ and any $ v_0 \ge u_0 $ in the place of $ u_0 $, it holds
\begin{equation}\label{comp-R-eq}
  u_{R} \le v_{R+1} \qquad \text{in } B_R \times (0,+\infty) \, ;
\end{equation}
     \smallskip
     \item \label{u-disc} $ u_R $ is the uniform limit in $ L^\infty_{\mathrm{loc}}\!\left([0,+\infty) ; L^1(B_R) \right) $, as $ h \to 0^+ $, of the piecewise-constant curves
\begin{equation}\label{pcint}
u^{(h)}(\cdot,t) := w_i \qquad \text{if } t \in \left[i h , (i+1)h \right)  ,
\end{equation}
where the sequence $ \{ w_i \} $ is recursively defined by setting $ w_0 := u_0 $ and solving the Dirichlet problems
\begin{equation}\label{eq-ui-rec}
\begin{cases} 
- \Delta \phi(w_{i+1}) + \frac 1 h \, w_{i+1} = \frac 1 h \, w_{i} & \text{in } B_R \, , \\
\phi(w_{i+1}) = 0 & \text{on } \partial B_R \, ;
\end{cases}
\end{equation}
\smallskip
     \item \label{energy-est-loc} $u_R$ satisfies the energy inequality
     \begin{equation}\label{EIR}
     \int_0^T \int_{B_R} \left| \nabla \phi(u_R) \right|^2 dV dt + \int_{B_R} \Phi(u_R(x,T)) \, dV(x) \le \int_{B_R} \Phi(u_0) \, dV \qquad \forall T>0 \, ; 
     \end{equation}
     \smallskip
      \item \label{contr-Lp-loc} The time function 
     \begin{equation}\label{nn-mon}
     [0,+\infty) \ni t \mapsto \left\| u_R(\cdot,t) \right\|_{L^p(B_R)} 
     \end{equation}
     is nonincreasing, continuous for all $ p \in [1,\infty) $, and right continuous for $p=\infty$;
     \smallskip
     \item  \label{energy-decreas-loc} $ \phi(u_R(\cdot,t)) \in H^1_0(B_R) $ for all $ t>0 $ and the time function
         \begin{equation}\label{EIR-mon}
    [0,+\infty) \ni  t \mapsto  \int_{B_R} \left| \nabla \phi(u_R(x,t)) \right|^2 dV(x) \in [0,+\infty]
     \end{equation}
     is nonincreasing and right continuous for all $ t>0$, including $ t=0 $ if $ \phi(u_0) \in H^1_0(B_R) $;
     \smallskip
     \item  \label{brezis-loc} $u_R \in AC_{\mathrm{loc}}\!\left([0,+\infty) ; H^{-1}(B_R) \right)$ and the estimate
     \begin{equation}\label{est-bre}
         \left\| \partial_t u_R(\cdot,t) \right\|_{H^{-1}(B_R)} \le \sqrt{\frac{\int_{B_R} \Phi(u_0) \, dV}{t}} \qquad \text{for a.e. } t > 0 
     \end{equation}
     holds.
 \end{enumerate}
\end{proposition}
\begin{proof}
The proof is deferred until Appendix \ref{aux}, Subsection \ref{aux-3},  since it can be considered by now rather standard, as in this context the fact of working on Riemannian manifolds does not cause any additional obstructions.  
\end{proof}

Now we show that the weak solutions to \eqref{filt-eq-local} can also be obtained as limits of more regular solutions, in particular, those associated with a \emph{smooth} and \emph{nondegenerate} nonlinearity $\phi$. This is a typical trick in nonlinear parabolic equations, which has been exploited both on the level of specific degenerate or singular PDEs and on the level of abstract contraction semigroups in Hilbert spaces, see \cite[Section 5.5]{Vaz} and \cite[Proof of Theorem 3.1]{BreOp}, respectively. In the present setting, we need such an approximation to make sure that if $ u_0 $ is radially nonincreasing, so is the solution at all times. Moreover, this method will also allow us to remove the extra assumption in Proposition \ref{exuni-weak-loc} that $\phi$ is bijective and obtain a well-posedness result under the mere hypothesis \eqref{cond-phi}.

\begin{proposition}\label{approx-phireg}
 Let $ u_0 \ge 0 $ with $ u_0  \in L^\infty(B_R) $. Let $ \phi $ comply with \eqref{cond-phi} and $ \left\{ \phi_k \right\} \subset C^\infty([0,+\infty)) $ be a sequence of nonnegative functions such that for all $k \in \N$ 
 \begin{equation}\label{phi-k}
 \phi_k(0)=0 \, , \qquad \tfrac{1}{k+1} \le \phi_k' \le k+1  \, , \qquad \lim_{k \to \infty} \phi_k = \phi \quad \text{locally uniformly on $ [0,+\infty) \, . $}
 \end{equation}
 Then there exists a unique weak solution $ u_R $ of the Cauchy-Dirichlet problem \eqref{filt-eq-local}, in the sense of Definition \ref{weak-energy-sol-loc}, which satisfies \eqref{cont-L1-loc}, \eqref{comp-R}, \eqref{energy-est-loc}, \eqref{contr-Lp-loc}, \eqref{energy-decreas-loc}, \eqref{brezis-loc} of Proposition \ref{exuni-weak-loc}. Moreover, upon letting $ u_{R,k} $ denote the weak solution of the same Cauchy-Dirichlet problem with $ \phi_k $ in the place of $ \phi $, the following properties hold:
 \begin{enumerate}[(i)]

\item \label{R0} If $ u_0 $ is radially nonincreasing, then each $ u_{R,k}(\cdot,t) $ is also radially nonincreasing for all $ t>0 $;
\smallskip
\item \label{U1} For all $t>0$, we have
\begin{equation}\label{unif-L1}
 u_{R,k}(\cdot,t)  \underset{k \to \infty}{\longrightarrow} u_R(\cdot,t) \qquad \text{weakly$^*$ in } L^\infty(B_R) \,  ;  
\end{equation}
\smallskip
\item \label{RA} If $ u_0 $ is radially nonincreasing, then $ u_R(\cdot,t) $ is also radially nonincreasing for all $ t>0 $.
 \end{enumerate}
\end{proposition}
\begin{proof}
We prove each item separately.

\noindent \eqref{R0} Let us temporarily make the additional assumption that $ u_0 \in C^\infty\!\left( \overline{B}_R \right) $. From Proposition \ref{exuni-weak-loc}\eqref{u-disc}, we know that each solution $ u_{R,k} $ can be obtained as the uniform limit in  $ L^\infty_{\mathrm{loc}}\!\left([0,+\infty) ; L^1(B_R) \right) $, as $ h \to 0^+ $, of the piecewise-constant curves
\begin{equation}\label{disc-h}
u^{(h)}(\cdot,t) = w_i \qquad \text{if } t \in \left[i h , (i+1)h \right) ,
\end{equation}
where $ w_0 = u_0 $ and $w_{i+1}$ solves 
\begin{equation*}\label{eq-ui-rec-k}
\begin{cases} 
- \Delta \phi_k(w_{i+1}) + \frac 1 h \, w_{i+1} = \frac 1 h \, w_{i} & \text{in } B_R \, , \\
\phi_k(w_{i+1}) = 0 & \text{on } \partial B_R \, .
\end{cases}
\end{equation*}
Note that each $ v_{i+1} :=  \phi_k(w_{i+1}) $ is the weak solution of the semilinear elliptic Dirichlet problem \eqref{ellip-1} with 
$$
\beta = \tfrac{1}{h} \, \phi_k^{-1} \qquad \text{and} \qquad  f = \tfrac{1}{h} \, \phi_k^{-1}(v_i) \, .
$$
Hence, because $ \phi'_k \le k+1  $, thanks to Lemma \ref{ulteriori-prop-approx}\eqref{noninc} we have that under the condition (which is valid for small $h>0$)
\begin{equation}\label{bdd-k-below}
\frac{1}{h(k+1)} \ge c_R
\end{equation}
the nonincreasing radiality of $ v_i $, and thus of $ f $, entails the nonincreasing radiality of $ v_{i+1} $. Since $ u_0 $ is by assumption radially nonincreasing, by induction we have that every $ v_i $, and therefore every $ w_i $, is radially nonincreasing. Passing to the limit as $ h \to 0^+ $ in \eqref{disc-h} we can finally assert that $ u_{R,k}(\cdot,t) $ is also radially nonincreasing for all $ t>0 $ (note that \eqref{bdd-k-below} is eventually fulfilled), as radial monotonicity is clearly stable under $ L^1(B_R) $ convergence. For a similar reason, we can safely remove the extra hypothesis $ u_0 \in C^\infty\!\left( \overline{B}_R \right) $ by picking a sequence of smooth and radially nonincreasing data that approximate $ u_0 $ in $ L^1(B_R) $, applying the just established radial monotonicity property to the corresponding solutions, and taking limits in $ L^1(B_R) $ using \eqref{contr-L1-p-loc}.

\noindent \eqref{U1} 
From \eqref{est-bre} applied to the solutions $  u_R \equiv u_{R,k}  $ with $ \Phi \equiv \Phi_k $, taking advantage of the fact that $ \left\| u_{R,k}(\cdot,t) \right\|_{L^\infty(B_R)} \le \| u_0 \|_{L^\infty(B_R)}  $ for all $ t>0 $, we can deduce that $ \left\{ u_{R,k} \right\}_k $ is a sequence of equicontinuous curves in $ C\!\left([0,T];H^{-1}(B_R)\right) $ for all $ T>0 $; moreover, it takes values in a precompact set of $ H^{-1}(B_R) $, since each $ u_{R,k} $ lies in the $L^2(B_R)$ ball of radius $ \| u_0 \|_{L^2(B_R)} $ and $ L^2(B_R) $ is compactly embedded into $ H^{-1}(B_R) $. Hence, from the Ascoli-Arzel\`a theorem in infinite dimensions, we can infer that there exists some (nonnegative) $\tilde{u} \in  C\!\left([0,+\infty);H^{-1}(B_R)\right) $ such that 
\begin{equation}\label{id-lim}
\lim_{k \to \infty} \left\| u_{R,k} - \tilde u \right\|_{C\left([0,T];H^{-1}(B_R)\right) } =0 \qquad \forall T > 0 \, ,
\end{equation}
at least up to a subsequence, independent of $T$, which will not be relabeled. Still from the uniform boundedness of $ \left\{ u_{R,k} \right\}_k $, the energy inequality \eqref{EIR}, and the monotonicity of \eqref{EIR-mon} (all of them applied to $u_R \equiv u_{R,k}$, $ \phi \equiv \phi_k $, and $ \Phi \equiv \Phi_k $), we can deduce that
$$
\left\{ \phi_k\!\left(u_{R,k}(\cdot,t)\right) \right\}_k \text{ is bounded in $ H^1_0(B_R) $ for all $t>0 \, .$}
$$
Therefore, for every $ t>0 $ there exists a further subsequence (\emph{a priori} depending on $ t $) such that $ \left\{ \phi_{k}\!\left(u_{R,k}(\cdot,t)\right) \right\}_k $ is strongly  (and pointwise almost everywhere) convergent to some $\rho \in  L^2(B_R) $. Also, thanks to \eqref{id-lim} and the above noticed boundedness of $ \left\{ u_{R,k}(\cdot,t) \right\}_k $ in $ L^\infty(B_R) $, we have
\begin{equation}\label{wstar}
u_{R,k}(\cdot,t) \underset{k \to \infty}{\longrightarrow} \tilde u(\cdot,t) \qquad \text{weakly$^*$ in } L^\infty(B_R) \, .
\end{equation}
Let us show that $ \rho = \phi(\tilde{u}(\cdot,t)) $. To this end, from the uniform convergence in \eqref{phi-k} and the definitions \eqref{pseudo} and \eqref{pseudo-bis}, it is straightforward to check that
$$
\begin{gathered}
\phi^{-1}_l (\rho(x)) \le \liminf_{k \to \infty} u_{R,k}(x,t) \le \limsup_{k \to \infty} u_{R,k}(x,t) \le \phi^{-1}_r (\rho(x)) \wedge \| u_0 \|_{L^\infty(B_R)} \qquad \text{for a.e. } x \in B_R \, .
\end{gathered}
$$
As a result, by Fatou's lemma we can infer that for every nonnegative $g \in L^1(B_R) $ 
$$
\begin{aligned}
\int_{B_R} \phi^{-1}_l(\rho) \, g \, dV \le \int_{B_R} \left[\liminf_{k \to \infty} u_{R,k}(x,t) \right] g(x) \, dV(x) \le & \, \lim_{k \to \infty} \int_{B_R} u_{R,k}(x,t) \, g(x) \, dV(x) \\
= & \, \int_{B_R} \tilde{u}(x,t) \, g(x) \, dV(x)
\end{aligned}
$$
and
$$
\begin{aligned}
\int_{B_R} \left[\phi^{-1}_r (\rho) \wedge \| u_0 \|_{L^\infty(B_R)} \right] g \, dV \ge \int_{B_R} \left[ \limsup_{k \to \infty} u_{R,k}(x,t) \right] g(x) \, dV(x) \ge & \, \lim_{k \to \infty} \int_{B_R} u_{R,k}(x,t) \, g(x) \, dV(x) \\
= & \, \int_{B_R} \tilde{u}(x,t) \, g(x) \, dV(x) \, .
\end{aligned}
$$
Thanks to the arbitrariness of $ g $, it follows that
$$
\phi^{-1}_l(\rho(x)) \le \tilde{u}(x,t) \le \left[\phi^{-1}_r (\rho(x)) \wedge \| u_0 \|_{L^\infty(B_R)} \right] \qquad \text{for a.e. } x \in B_R \, ,
$$
which ensures the identification  $ \rho = \phi(\tilde{u}(\cdot,t)) $. Hence, we can assert that along the same sequence under which \eqref{id-lim} holds, we have
\begin{equation*}\label{weak-phi}
\phi_k(u_{R,k}(\cdot,t))  \underset{k \to \infty}{\longrightarrow} \phi(\tilde u(\cdot,t)) \qquad \text{weakly in } H^1_0(B_R) \text{ for all } t>0 \, .
\end{equation*}
On the other hand, still from the energy inequality, we can assert that $ \left\{ \phi_k\!\left(u_{R,k}\right) \right\}_k $ is bounded in $ L^2\!\left( (0,T) ; H^1_0(B_R) \right) $, which, due to the previous identifications, ensures that 
$$
\phi\!\left(u_{R,k}\right) \underset{k\to\infty}{\longrightarrow} \phi(\tilde{u}) \qquad \text{weakly in } L^2\!\left( (0,T) ; H^1_0(B_R) \right) .
$$
Therefore, passing to the limit in the weak formulations
$$
\int_0^{T} \int_{B_R} u_{R,k} \, \partial_t \xi \, dV dt = \int_0^{T} \int_{B_R} \left\langle \nabla \phi_k\!\left(u_{R,k}\right) , \nabla \xi  \right\rangle dV dt - \int_{B_R} u_0(x) \, \xi(x,0) \, dV(x) \, ,
$$
for every $ \xi \in H^1\!\left( (0,T) ; L^2(B_R) \right) \cap  L^2\!\left( (0,T) ; H^1_0(B_R) \right) $ such that $ \xi(\cdot,T) = 0 $, it is plain that $ \tilde u $ is also a weak solution to \eqref{filt-eq-local}, in the sense of Definition \ref{weak-energy-sol-loc}. It follows that $ \tilde u $ is indeed a weak solution to \eqref{filt-eq-local}, so we can write $ \tilde u = u_R $, such a solution being uniquely identified (see the end of the proof of Proposition \ref{exuni-weak-loc}) and observe that \eqref{unif-L1} is just \eqref{wstar}. 

Next, we notice that \eqref{contr-L1-p-loc}, \eqref{comp-R-eq}, \eqref{EIR}, and \eqref{est-bre} are easily preserved at the limit as a direct consequence of the above established weak-convergence properties. Moreover, the inequalities
$$
\|  u_{R,k}(\cdot,t) \|_{L^p(B_R)}  \le \| u_0 \|_{L^p(B_R)} \qquad \forall t > 0
$$
are stable under the limit as $ k \to \infty $ for all $ p \in [1,\infty] $, ensuring that
\begin{equation}\label{R-p-weak}
\|  u_{R}(\cdot,t) \|_{L^p(B_R)}  \le \| u_0 \|_{L^p(B_R)} \qquad \forall t>0 \, .
\end{equation}
From the weak formulation satisfied by $u_R$, it is not difficult to check that
$$
u_R(\cdot,t) \underset{t \to 0^+}{\longrightarrow} u_0 \qquad \text{weakly in } L^p(B_R) \, ,
$$
which, in combination with \eqref{R-p-weak}, yields 
\begin{equation}\label{cont-norms}
\lim_{t \to 0^+}\left\| u_R(\cdot,t) - u_0 \right\|_{L^p(B_R)} = 0 \qquad \forall p \in [1,\infty) \, .
\end{equation}
In order to establish that $u_R \in C\!\left([0,+\infty) ; L^1(B_R) \right)$, we can exploit \eqref{cont-norms} along with the following \emph{semigroup property}:
\begin{equation}\label{SP}
\begin{gathered}
\text{for every $ \tau>0 $, we have that the translated solution $ t \mapsto u_R(\cdot,t+\tau) $} \\
\text{is the weak solution to \eqref{filt-eqintr} with initial datum $ u_0 \equiv  u_R(\cdot,\tau) \, , $}
\end{gathered}
\end{equation}
which easily follows from Definition \ref{weak-energy-sol-loc}, the above established $ H^{-1}(B_R) $ continuity property, and the uniqueness of weak solutions. As a consequence, from \eqref{contr-L1-p-loc} we can infer
$$
\left\| u_R(\cdot,t) - u_R(\cdot,s) \right\|_{L^1(B_R)} \le \left\| u_R(\cdot,|t-s|) - u_0 \right\|_{L^1(B_R)} \qquad \forall t,s>0 \, ,
$$
whence the claimed continuity in $ L^1(B_R) $ readily follows from \eqref{cont-norms}. Moreover, still by combining the semigroup property with \eqref{R-p-weak} and \eqref{cont-norms}, it is plain that the monotonicity and continuity of \eqref{nn-mon} hold as well.

As concerns the monotonicity and right continuity of  \eqref{EIR-mon}, there are several ways to prove it, the shortest of which is repeating the same $ H^{-1}(B_R) $ gradient-flow argument as in the case when $ \phi $ is bijective, since the abstract theory from \cite{BreMon,BreOp} also works, for $ L^\infty(B_R) $ data, under the mere assumption \eqref{cond-phi} (see the proof of Proposition \ref{exuni-weak-loc} in the Appendix).
 
\noindent \eqref{RA} It is a straightforward consequence of \eqref{R0} and \eqref{U1}.
\end{proof}

Before continuing we stress that, actually, it would also be possible to reconstruct $u_R$ as a limit of discretized curves as in \eqref{eq-ui-rec} even when $\phi$ is not bijective, but we prefer to avoid it here because in that case the equivalence between \eqref{eq-ui-rec} and the semilinear elliptic problems addressed in Section \ref{loc-elliptic} works up to \emph{differential inclusions} (see \emph{e.g.}\ \cite[Chapter 10]{Vaz}). 

We are now ready to prove Proposition \ref{exuni-weak}. We will first establish uniqueness and then existence, since the latter takes advantage of the former at a specific point regarding the ``semigroup property'' of constructed weak energy solutions.

\begin{proof}[Proof of Proposition \ref{exuni-weak} (uniqueness)]
We carefully adapt a technique originally due to Ole\u{\i}nik (see \emph{e.g.}~\cite[Theorem 5.3]{Vaz}). First of all, for each $ M>0 $, we introduce the truncation operator
$$
T_M(v):=
\begin{cases}
   v & \text{if } -M \le v \le M \, , \\
   M & \text{if } v>M \, , \\
   -M & \text{if } v<-M \, .
\end{cases}
$$
Given any two weak energy solutions $ u_1 $ and $ u_2 $ of the same Cauchy problem \eqref{filt-eqintr} and an arbitrary $T>0$, let us consider the following function:
$$
F (x,t) :=
\begin{cases}
\int_t^T  T_M \! \left( \phi(u_2(x,s)) - \phi(u_1(x,s)) \right) ds  & \text{if } t \in [0,T] \, , \\
0 & \text{if } t>T \, .
\end{cases}
$$
From the assumptions, we can easily deduce that
$$ 
f(x,t) := T_M \! \left( \phi(u_2(x,t)) - \phi(u_1(x,t)) \right) \in L^2\big( (0,T) ; \hdot \big) \, ;
$$
therefore, by virtue of Lemma \ref{equiv} and Proposition \ref{density-bochner}, we know that there exists a sequence $ \{ f_k \} \subset C_c^\infty\!\left(\HH^n \times (0,T)\right)  $ such that
\begin{equation}\label{f-k-approx}
\begin{gathered}
  \lim_{k \to \infty} \int_0^T \int_{B_r} \left| f - f_k \right|^2  dV dt +\int_0^T \int_{\HH^n} \left| \nabla  f - \nabla  f_k \right|^2  dV dt =  0 \, , \\
   \left\| f_k \right\|_{L^\infty(\HH^n \times (0,T))} \le M \qquad \forall k \in \N \, ,
   \end{gathered}
\end{equation}
for all $r>0$. Next, we construct the following natural approximations of $ F $:
$$
F_k(x,t) := \int_t^T f_k (x,s) \, ds \, ,
$$
all of which belong to $ C^\infty_c\!\left( \HH^n \times [0,T) \right) $, whence it is feasible to plug them into the weak formulation satisfied by $ u_2-u_1 $, obtaining the identity 
\begin{equation}\label{oleinik-1-pre}
 - \int_0^{T} \int_{\HH^n} \left( u_2-u_1 \right) f_k \, dV dt =  \int_0^{T} \int_{\HH^n} \left\langle \nabla\!\left( \phi(u_2) - \phi(u_1) \right), \nabla F_k  \right\rangle dV dt \, .
\end{equation}
From \eqref{f-k-approx}, it is plan that we can safely pass to the limit as $ k \to \infty $ in \eqref{oleinik-1-pre} to obtain
\begin{equation}\label{oleinik-1}
\begin{aligned}
& \, - \int_0^{T} \int_{\HH^n} \left( u_2-u_1 \right) T_M \! \left( \phi(u_2) - \phi(u_1) \right) dV dt \\
= & \, \int_0^{T} \int_{\HH^n} \left\langle \nabla\!\left( \phi(u_2) - \phi(u_1) \right), \int_t^T \nabla T_M\!\left( \phi(u_2) - \phi(u_1) \right) ds  \right\rangle dV dt \\
= & \, \int_0^{T} \int_{\HH^n} \left\langle \nabla\!\left( \phi(u_2) - \phi(u_1) \right), \int_t^T \chi_{\left\{|\phi(u_2)-\phi(u_1)|<M\right\}} \nabla\! \left( \phi(u_2) - \phi(u_1) \right) ds  \right\rangle dV dt \, .
\end{aligned}
\end{equation}
Since the integral in the first line of \eqref{oleinik-1} is nonnegative regardless of $M$ (recall that $\phi$ is increasing), and it is monotone increasing with respect to such a truncation parameter, by taking the limit as $M \to +\infty$ and using Fubini's theorem we end up with
$$
\begin{aligned}
  & \, - \int_0^{T} \int_{\HH^n} \left( u_2-u_1 \right) \left( \phi(u_2) - \phi(u_1) \right) dV dt \\
  = & \, \int_{\HH^n} \int_0^{T} \left\langle \nabla\!\left( \phi(u_2) - \phi(u_1) \right), \int_t^T \nabla\! \left( \phi(u_2) - \phi(u_1) \right) ds \right\rangle dt dV \, ,
 \end{aligned}
$$
that is, 
$$
\begin{aligned}
\int_0^{T} \int_{\HH^n} \left( u_2-u_1 \right) \left( \phi(u_2) - \phi(u_1) \right) dV dt 
- \frac 1 2  \int_{\HH^n} \int_0^{T} \frac{d}{dt} \left| \int_t^T \nabla\! \left( \phi(u_2) - \phi(u_1) \right) ds  \right|^2 dt dV = 0 \, ,
 \end{aligned}
$$
which, upon integrating in $dt$, entails 
\begin{equation}\label{oleinik-2}
\int_0^{T} \int_{\HH^n} \left( u_2-u_1 \right) \left( \phi(u_2) - \phi(u_1) \right) dV dt + \frac 1 2  \int_{\HH^n} \left| \int_0^T \nabla\! \left( \phi(u_2) - \phi(u_1) \right) ds \right|^2 dV = 0 \, .
\end{equation}
From \eqref{oleinik-2}, the arbitrariness of $T$, and the monotonicity of $\phi$, we can infer that $ \phi(u_1) = \phi(u_2) $ almost everywhere in $ \HH^n  \times (0,+\infty)$. On the other hand, if such functions coincide, then from the weak formulations satisfied by $u_1$ and $u_2$ it is immediate to check that also $u_1 = u_2$ almost everywhere in $ \HH^n  \times (0,+\infty)$.
\end{proof}

\begin{proof}[Proof of Proposition \ref{exuni-weak} (existence and additional properties)]
First of all, let us prove that the sequence $ \{ u_k \} $ of solutions of the Cauchy-Dirichlet problems \eqref{filt-eq-local} with $ R\equiv k $ and initial data $ u_0 $ replaced by
\begin{equation*}\label{trunc}
u_{0,k} := k \wedge u_0 \, ,
\end{equation*}
tacitly extended to zero in $ B_k^c \times [0,+\infty) $ throughout the proof, converges pointwise almost everywhere in $ \HH^n \times (0,+\infty) $ to a weak energy solution of the Cauchy problem \eqref{filt-eqintr}; then, we will focus on the stated additional properties.   

Thanks to Proposition \ref{exuni-weak-loc}\eqref{comp-R}, we have that $ \{ u_k \} $ is a monotone nondecreasing sequence, which therefore converges pointwise almost everywhere in $ \HH^n \times (0,+\infty) $ to a certain nonnegative function $u$. On the other hand, from property \eqref{contr-Lp-loc} of the same proposition we easily deduce that $ u \in L^\infty\!\left((0,+\infty);L^1(\HH^n)\right) $, therefore, by monotone convergence,
\begin{equation}\label{bas-conv-1}
 \lim_{k \to \infty} \left\| u_k - u \right\|_{L^1(\HH^n \times (0,T) )} = 0 \qquad \forall T>0 \, .
\end{equation}
We now follow a strategy similar to that of \cite[Proposition 3.4]{FM}. Taking advantage of the energy inequality \eqref{EIR} we deduce, in particular, that for all $ M>0 $ and all $ t>0 $ it holds
\begin{equation*}\label{equiint-a}
\begin{aligned}
 \Phi(M) \cdot  V\!\left(\left\{x \in B_1 : \  u_k(x,t) > M \right\} \right) \le \int_{B_k} \Phi(u_k(x,t)) \, dV(x) 
 \le  \int_{\HH^n} \Phi(k \wedge u_0) \, dV  \le  \int_{\HH^n} \Phi(u_0) \, dV \, .
\end{aligned}
\end{equation*}
Hence, because $ \lim_{M \to +\infty} \Phi(M) = +\infty $ (simple consequence of $\phi$ being nonconstant), we can pick $ M $ large enough, independently of $k$ and $t$, such that 
\begin{equation}\label{equiint-bis}
V\!\left(\left\{x \in B_1 : \  u_k(x,t) > M \right\} \right) \le \tfrac{1}{2}  \, V(B_1) \, .
\end{equation}
By virtue of \eqref{equiint-bis} and the local Poincar\'e inequality \eqref{lp-ineq}, we can infer that 
$$
\begin{aligned}
\sqrt{\tfrac{V(B_1)}{2}} \ \overline{\phi(u_k(\cdot,t))}_1  \le & \left\| \overline{\phi(u_k(\cdot,t))}_1 \right\|_{L^2(\{x \in B_1: \ u_k(\cdot,t) \le  M \})} \\
\le & \left\| \phi(u_k(\cdot,t))\right\|_{L^2(\{x \in B_1 : \ u_k(\cdot,t) \le  M \})} + \left\| \phi(u_k(\cdot,t)) - \overline{\phi(u_k(\cdot,t))}_1 \right\|_{L^2(B_1)} \\
\le & \sqrt{V(B_1)} \, \phi(M) + C_1 \left\| \nabla \phi(u_k(\cdot,t)) \right\|_{L^2(B_1)} . 
 \end{aligned}
$$
Upon exploiting again the local Poincar\'e inequality, we thus obtain
\begin{equation}\label{ee-gg}
 \left\| \phi(u_k(\cdot,t))\right\|_{L^2(B_1)}   \le \sqrt{2V(B_1)} \, \phi(M) + \left(1+\sqrt 2\right) C_1 \left\| \nabla \phi(u_k(\cdot,t)) \right\|_{L^2(B_1)} ,
\end{equation}
so that by squaring \eqref{ee-gg} and integrating over $ (0,T) $ we end up with 
\begin{equation}\label{ee-gg-bis}
\int_0^T \left\| \phi(u_k(\cdot,t))\right\|_{L^2(B_1)}^2  dt \le 4 T \,  V(B_1) \, \phi(M)^2 + \left( 6 +4\sqrt 2 \right) C_1^2 \int_0^T \left\| \nabla \phi(u_k(\cdot,t))\right\|_{L^2(B_1)}^2  dt \, .
\end{equation}
Thanks to \eqref{ee-gg-bis}, the energy inequality, and the fact that each $ \phi(u_k(\cdot,t)) $, extended to zero outside $ B_k $, belongs to $ H^1_c(\HH^n) $, it follows that
\begin{equation}\label{weak-hdot-conv}
\left\{ \phi(u_k) \right\} \text{ is bounded in } L^2\big((0,T) ; \hdot \big) \, .
\end{equation}
Because the pointwise limit of $ \{ u_k \} $, and hence of $ \{ \phi(u_k) \} $, has already been identified, in view of \eqref{weak-hdot-conv} we can then assert that
$$
 \phi(u_k) \underset{k \to \infty}{\longrightarrow} \phi(u) \qquad \text{weakly in } L^2\big((0,T) ; \hdot \big) \, .
$$
This convergence property, along with \eqref{bas-conv-1}, allows us to pass to the limit as $k \to \infty$ in the local weak formulation \eqref{weak-form-loc} satisfied by each $ u_k $, showing that $u$ is indeed a weak energy solution to \eqref{filt-eqintr}.

Next, let us prove the additional properties. From the same argument that lead to \eqref{bas-conv-1}, we also have
\begin{equation}\label{L1-ae}
 \lim_{k \to \infty} \left\| u_k(\cdot, t) - u(\cdot,t) \right\|_{L^1(\HH^n)}   = 0 \qquad \text{for a.e. } t>0 \, .
\end{equation}
The continuity of $ t \mapsto u(\cdot,t) $ with values in $ L^1(\HH^n) $ will then follow from the Ascoli-Arzel\`a theorem in infinite dimension, provided we can show that the sequence $ \{ u_k \} $ is equicontinuous with values in such a space. To this aim, let us temporarily make the assumption $ u_0 \in L^1(\HH^n) \cap L^2(\HH^n) $ (we will explain at the end of the argument how to remove it). Our goal is to prove that for every $ \varepsilon>0 $ there exist $ \delta_\varepsilon>0 $ and $ k_\varepsilon \in \N $ such that 
\begin{equation}\label{L1-smallness-glob}
 \left\| u_{k}(\cdot,s) - u_0 \right\|_{L^1(\HH^n)}  < \varepsilon  \qquad \forall s \in (0,\delta_\varepsilon) \, , \ \forall k>k_\varepsilon \, , 
\end{equation}
which amounts to disproving the existence of some $ \varepsilon_0>0 $ and a sequence $ s_k \to 0^+ $ such that 
\begin{equation}\label{L1-smallness-contr-glob}
\limsup_{k \to \infty}  \left\| u_{k}(\cdot,s_k) - u_0 \right\|_{L^1(\HH^n)}  \ge \varepsilon_0 \, .
\end{equation}
By using the weak formulation satisfied by $ u_k $ with special separable test functions, and again the energy inequality \eqref{EIR}, it is not difficult to deduce that
\begin{equation}\label{L1-smallness-contr-iden}
\begin{aligned}
\lim_{k \to \infty} \int_{\HH^n} u_k(x,s_k) \, \eta(x) \, dV(x) = &  \lim_{k \to \infty} \left( - \int_0^{s_k} \int_{\HH^n} \left\langle \nabla \phi(u_k), \nabla \eta  \right\rangle dV dt + \int_{\HH^n} u_0 \, \eta \, dV \right) = \int_{\HH^n} u_0 \, \eta \, dV \, ,
\end{aligned}
\end{equation}
for every $ \eta \in C^\infty_c(\HH^n) $. On the other hand, upon combining identity \eqref{L1-smallness-contr-iden} with the inequality  $$ \limsup_{k \to \infty} \| u_k(\cdot,s_k) \|_{L^2(\HH^n)} \le \| u_0 \|_{L^2(\HH^n)} \, , $$ 
consequence of \eqref{nn-mon}, it is plain that actually $ \{ u_k(\cdot,s_k) \} $ converges strongly in $ L^2(\HH^n) $ to $ u_0 $ \eqref{L1-smallness-contr-iden}. Moreover, with similar arguments, we obtain 
\begin{equation}\label{L1-L1}
\lim_{k \to \infty} \left\| u_k(\cdot,s_k) \right\|_{L^1(\HH^n)} = \left\| u_0 \right\|_{L^1(\HH^n)} .
\end{equation}
Since strong $ L^2(\HH^n) $ convergence implies convergence in $ L^1_{\mathrm{loc}}(\HH^n) $, the latter combined with \eqref{L1-L1} ensures that in fact $  \{ u_k(\cdot,s_k) \} $ strongly converges in $ L^1(\HH^n) $ to $ u_0 $, contradicting \eqref{L1-smallness-contr-glob}. Property \eqref{L1-smallness-glob} is thus established. As a result, the sequence $ \left\{ u_k \right\} $ is uniformly equicontinuous in $ L^1(\HH^n) $: indeed, thanks to the $ L^1 $-contraction property \eqref{contr-L1-p-loc} applied to $ u_R \equiv u_{k}(\cdot,\cdot+t-s) $ and $ v_R \equiv u_{k} $, we find
\begin{equation*}\label{L1-smallness-contr-quater}
\begin{gathered}
\left\| u_{k}(\cdot , t) - u_{k}(\cdot , s) \right\|_{L^1(\HH^n)} \le \left\| u_{k}(\cdot, t-s) - u_0 \right\|_{L^1(\HH^n)} < \varepsilon \qquad 
\forall t> s \ge 0 : \ t-s < \delta_\varepsilon \, , \ \forall k > k_\varepsilon \, .
\end{gathered}
\end{equation*}
Therefore, upon appealing to the Ascoli-Arzel\`a theorem in infinite dimensions (we have already established \eqref{L1-ae}), we can finally assert that \eqref{unif-L1} holds. In order to remove the assumption $ u_0 \in L^2(\HH^n) $, for every $ \varepsilon>0 $ let $ u_{0}^\varepsilon \in L^1(\HH^n) \cap  L^\infty(\HH^n)  $ be a nonnegative initial datum such that $ \left\| u_{0}^\varepsilon - u_0 \right\|_{L^1(\HH^n)} < \varepsilon $, and let us denote by $ \left\{ u_k^\varepsilon \right\}_k $ the corresponding solutions to \eqref{filt-eq-local} with $ R\equiv k $ and $ u_0 \equiv u_0^\varepsilon $. From the just established equicontinuity property, we know that there exists $ \delta_\varepsilon>0 $ such that for all $ k \in \N $
$$
\left\| u_k^\varepsilon(\cdot,t) - u_k^\varepsilon(\cdot,s) \right\|_{L^1(\HH^n)} < \varepsilon \qquad  \forall t> s \ge 0 : \ t-s < \delta_\varepsilon \, ;
$$
on the other hand, the $ L^1(\HH^n) $ contraction principle \eqref{contr-L1-p-loc} entails, for the same $ t,s $ as above, 
$$
\begin{aligned}
& \left\| u_k(\cdot,t) - u_k(\cdot,s) \right\|_{L^1(\HH^n)} \\
\le & \left\| u_k(\cdot,t) - u_k^\varepsilon(\cdot,t) \right\|_{L^1(\HH^n)} + \left\| u_k(\cdot,s) - u_k^\varepsilon(\cdot,s) \right\|_{L^1(\HH^n)} + \left\| u_k^\varepsilon(\cdot,t) - u_k^\varepsilon(\cdot,s) \right\|_{L^1(\HH^n)} \\
\le & \, 3 \varepsilon + 2 \left\| \left( k \wedge u_0 \right) - u_0 \right\|_{L^1(\HH^n)} ,
\end{aligned} 
$$
whence the claimed equicontinuity of the sequence $ \{ u_k \} $ follows. As a result, we can assert that $ u \in C\!\left([0,+\infty); L^1(\HH^n)\right) $ and
\begin{equation}\label{k5}
\left\| u_k - u \right\|_{C\left([0,T]; L^1(\HH^n)\right)} = 0 \qquad \forall T>0 \,  .
\end{equation}
Therefore, property \eqref{cont-L1} is established, the $L^1 $ contraction principle \eqref{contr-L1-p} following upon taking limits in \eqref{contr-L1-p-loc}. Property \eqref{energy-est} also follows upon letting $ R \equiv k \to \infty $ in \eqref{EIR}, exploiting lower semicontinuity and Fatou's lemma on the left-hand side. 

For the rest of the proof, it is convenient to point out that the semigroup property \eqref{SP} still holds with $ u_R \equiv u $, as a consequence of 
Definition \ref{weak-energy-sol}, the just established $ L^1(\HH^n) $ continuity property, and the uniqueness of weak energy solutions. Let us then turn to \eqref{contr-Lp}. Due to Proposition \ref{exuni-weak-loc}\eqref{contr-Lp-loc}, for all $ p \in[1,\infty ]$ we have 
\begin{equation*}\label{p-lim}
  \left\| u_k(\cdot,t) \right\|_{L^p(\HH^n)}  \le  \left\| k \wedge u_0 \right\|_{L^p(\HH^n)} \qquad \forall t>0 \, .
\end{equation*}
Hence, passing to the limit as $ k \to \infty $, and using the weak lower semicontinuity of $L^p(\HH^n) $ norms (or Fatou's lemma for $ p< \infty $), we obtain
\begin{equation}\label{p-lim-1}
  \left\| u(\cdot,t) \right\|_{L^p(\HH^n)}  \le  \left\| u_0 \right\|_{L^p(\HH^n)} \qquad \forall t>0 \, .
\end{equation}
On the other hand, from the semigroup property, we can shift \eqref{p-lim-1} by any arbitrary $ \tau>0 $, which yields
\begin{equation}\label{p-lim-2}
  \left\| u(\cdot,t+\tau) \right\|_{L^p(\HH^n)}  \le  \left\| u(\cdot,\tau) \right\|_{L^p(\HH^n)} \qquad \forall t>0 \, ,
\end{equation}
showing that $ t \mapsto \| u(\cdot,t) \|_{L^p(\HH^n)} $ is nonincreasing. Moreover, because continuity in $ L^1(\HH^n) $ holds, by exploiting again the weak lower semicontinuity of $ L^p(\HH^n) $ norms, it is plain that $ t \mapsto \| u(\cdot,t) \|_{L^p(\HH^n)} $ is a lower semicontinuous function, and a nonincreasing lower semicontinuous function is necessarily right continuous. 

The proof of \eqref{energy-decreas} is similar, up to some further technical issues. Assume first that $ \phi(u_0) \in H^1_c(\HH^n) $, which clearly implies that $  \phi(k \wedge u_0) \in H^1_0(B_k) $ for all $ k $ large enough. By virtue of Proposition \ref{exuni-weak-loc}\eqref{energy-decreas-loc} we have $ \phi(u_k(\cdot,t)) \in \hdot $ and 
\begin{equation}\label{EIR-k}
 \left\| \nabla \phi(u_k(\cdot,t)) \right\|_{L^2(\HH^n)} \le  \left\| \nabla \phi\!\left( k \wedge u_0 \right) \right\|_{L^2(\HH^n)} \le  \left\| \nabla \phi\!\left( u_0 \right) \right\|_{L^2(\HH^n)} \qquad \forall t>0 \, .
\end{equation}
Recalling the local estimate \eqref{ee-gg}, we can infer that $ \left\{ \phi(u_k(\cdot,t)) \right\} $ is bounded in $ L^2(B_1) $ and thus in $ \hdot $; since we already know that $ \left\{ u_k(\cdot,t) \right\} $ converges to $ u(\cdot,t) $ in $ L^1(\HH^n) $, in particular we can deduce that $ \phi(u(\cdot,t)) \in \hdot $ and $ \left\{ \nabla \phi(u_k(\cdot,t) ) \right\} $ converges weakly in $ L^2(\HH^n) $ to $ \nabla \phi(u(\cdot,t) ) $, so that upon taking limits as $ k \to \infty $ in \eqref{EIR-k} and exploiting the weak lower semicontinuity of $ L^2(\HH^n) $ norms we end up with   
\begin{equation}\label{EIR-k-bis}
 \left\| \nabla \phi(u(\cdot,t)) \right\|_{L^2(\HH^n)} \le  \left\| \nabla \phi\!\left( u_0 \right) \right\|_{L^2(\HH^n)} \qquad \forall t>0 \, .
\end{equation}
If $ \phi(u_0) \in \hdot $, without being compactly supported, then from Lemma \ref{approx-above} we know that there exists a sequence $ \{ \rho_i \} \subset H^1_c(\HH^n) $ such that $ \rho_i \to \phi(u_0) $ as $i \to \infty$ in $\hdot$ and, for all $ i \in \N $,
$$
0 \le \rho_i \le \phi(u_0) \, , \ \qquad \rho_i = \phi(u_0) \quad \text{in } B_{R_i} \, ,
$$
for some increasing sequence $ R_i \to +\infty $. If we set (recall definition \eqref{pseudo}) 
$$
u_{0,i} := 
\begin{cases}
u_0 & \text{in } B_i \, , \\
\phi_l^{-1}(\rho_i) & \text{in } B_i^c \, ,
\end{cases}
$$
it is readily seen that $ 0 \le u_{0,i} \le u_0 $, $ u_{0,i} \to u_0 $ as $i \to \infty$ in $L^1(\HH^n)$ (\emph{e.g.}\ by dominated convergence), and $ \phi(u_{0,i}) = \rho_i $.  
Hence, thanks to the previously established $ L^1(\HH^n) $ contraction property, we have $ \hat{u}_i(\cdot,t) \to u(\cdot,t) $ as $ i \to \infty $ in $ L^1(\HH^n) $ for every $ t>0 $, where $ \hat{u}_i $ stands for the solution to \eqref{filt-eqintr} with $ u_0 \equiv u_{0,i} $. Moreover, estimate \eqref{EIR-k-bis} applied to $ u \equiv \hat{u}_i $ and the repetition of the above weak-convergence argument ensure that $ \phi(\cdot,t) \in \hdot $ and \eqref{EIR-k-bis} is preserved. Hence, by means of the same time-shift trick as in \eqref{p-lim-2}, we can infer that $ t \mapsto \| \nabla \phi(\cdot,t) \|_{L^2(\HH^n)} $ is nonincreasing on $ [0,+\infty) $. Finally, in the case when $ \phi(u_0) \not \in \hdot $, it is enough to observe that the energy inequality \eqref{EIR} combined with the nonincreasing monotonicity of \eqref{EIR-mon} entail
$$
t \left\| \nabla \phi(u_k(\cdot,t)) \right\|_{L^2(\HH^n)}^2 \le \int_{\HH^n} \Phi(u_0) \, dV \qquad \forall t>0 \, ,
$$  
which, upon letting $ k \to \infty $ and passing to the limit as in the previous steps, shows that $ \phi(u(\cdot,t)) \in \hdot $. Therefore, still up to a time shift, we can reduce to the case $ \phi(u_0) \in \hdot $ to infer that  $ t \mapsto \| \nabla \phi(\cdot,t) \|_{L^2(\HH^n)} $ is again nonincreasing for every $t>0$.
    
To conclude, as concerns \eqref{rad-inc} we simply observe that if $ u_0 $ is radially nonincreasing so is $ k \wedge u_0 $, hence, by Proposition \ref{approx-phireg}\eqref{RA}, we have that each $ u_k(\cdot,t) $ is radially nonincreasing for all $ t>0 $, and this property is clearly stable as $k\to\infty$ thanks to \eqref{k5}.
\end{proof}

\begin{corollary}\label{conv-L1}
Let $ u_0 \ge 0 $ with $ u_0 , \Phi(u_0) \in L^1(\HH^n) $. Let $ u_R $ and $ u $ be the weak solution to \eqref{filt-eq-local} with $ u_0 \equiv R \wedge u_0 $ and the weak energy solution to \eqref{filt-eqintr}, respectively. Then, upon extending $u_R$ to zero in $B_R^c \times [ 0,+\infty)$, we have 
\begin{equation*}\label{gen-conv-R}
 \lim_{R \to +\infty} \left\| u_R - u \right\|_{C\left([0,T] ; L^1(\HH^n) \right)}   = 0 \qquad \forall T>0 \, .
\end{equation*}
Moreover, the convergence of $ \{u_R\}$ to $ u $ is monotone increasing.
\end{corollary}
\begin{proof}
The result is a mere consequence of the first part of the existence proof of Proposition \ref{exuni-weak}, up to noticing that, in the corresponding argument, one can replace $ k $ with an arbitrary increasing sequence $ R_k \to + \infty $, the limit $u$ being uniquely identified (independently of the sequence) as the weak energy solution to \eqref{filt-eqintr}.
\end{proof}

\section{Proofs of the main results}\label{proof-main}
To begin with, we state and prove a Hopf-type result for radially nonincreasing solutions to \eqref{ellip-1}, under the additional constraint that $ \beta $ is Lipschitz, in the spirit of \cite[Chapter 1, Section 1]{PW}. Before, we need a technical one-dimensional lemma, whose proof is postponed until Appendix \ref{aux}.

\begin{lemma} \label{1-d-lemma}
	Let $ v $ be a $  C^1\!\left(\left( a,b \right]\right)  $ nonnegative function satisfying the integro-differential inequality
	\begin{equation}\label{dec-est}
	v'(r) + \alpha(r) \int_r^b \gamma(s) \, v(s) \, ds \ge 0 \qquad  \forall r \in (a,b]  \, ,
	\end{equation}
where $ \alpha \in L^1_{\mathrm{loc}}\!\left( \left( a,b\right] \right) $ and $ \gamma \in L^\infty_{\mathrm{loc}}\!\left( \left( a,b\right] \right) $. Assume that $ v(b)=0 $. Then $ v = 0 $ on $ (a,b]  $.
\end{lemma}
\begin{proof}
See Appendix \ref{aux}, Subsection \ref{aux-4}.
\end{proof}

\begin{proposition} \label{decreas}
Let $ R>0 $, $ \beta: [0,+\infty) \to [0,+\infty) $ be  a locally Lipschitz and nondecreasing function with $ \beta(0)=0 $, and $ f \in L^\infty\!\left( B_R \right) $ with $ f \ge 0 $. Assume in addition that $ \mathbb{M}^n $ supports the P\'olya-Szeg\H{o} inequality, in the sense of Definition \ref{def-polya}, and that $f$ is radially nonincreasing and nontrivial. Then, the radially nonincreasing weak solution to \eqref{ellip-1}, according to the statement of Corollary \ref{decr-min}, satisfies
\begin{equation*}\label{dec-est-2}
v'(r) <0 \qquad  \forall r \in (0,R] \, .
\end{equation*}
\end{proposition}
\begin{proof}
Assume, by contradiction, that there exists $ r_0 \in (0,R ]$ such that $ v'(r_0) = 0 $. Then, the integral form over $(r_0,r)$ of the radial differential equation deriving from \eqref{ellip-1} reads (recall \eqref{laplacian-rad})
\begin{equation}\label{psi-1}
- \psi(r)^{n-1} \, v'(r) +  \int_{r_0}^r \beta(v(s)) \, \psi(s)^{n-1} \, ds = \int_{r_0}^r f(s) \, \psi(s)^{n-1} \, ds \qquad \forall r \in (0,R] \, ,
\end{equation}
that is 
\begin{equation}\label{psi-2}
v'(r) = \frac{1}{\psi(r)^{n-1}}  \int_{r_0}^r \left[ \beta(v(s)) - f(s) \right] \psi(s)^{n-1} \, ds \qquad \forall r \in (0,R] \, .
\end{equation}
Since $f$ is nonincreasing, its left and right limits are everywhere well defined and finite; let us denote them by $ f_+ $ and $ f_ -$, respectively, so that $ f_+(r) \le f_-(r) $ for all $ r \in (0,R) $. On the other hand, the function $ r \mapsto \beta(v(r)) $ is continuous. Therefore, because $ v' \le 0 $ everywhere, from \eqref{psi-2} we can infer that 
\begin{equation*}\label{psi-3}
 f_-(r_0) \le \beta(v(r_0)) \le f_+(r_0) \, ,
\end{equation*}
which implies that $f$ is in fact continuous at $r=r_0$ and
\begin{equation}\label{psi-4}
f(r_0) = \beta(v(r_0)) \, .
\end{equation}
Note that the latter identity is also satisfied when $ r_0=R $, where both terms are forced to vanish. Upon setting $ V(r) := v(r) - v(r_0) \ge 0 $ for all $ r \in (0,r_0] $, from \eqref{psi-1}, \eqref{psi-4}, and again the fact that $ f $ decreases, we obtain 
\begin{equation*}\label{psi-1-hopf}
 V'(r) + \frac{1}{\psi(r)^{n-1}} \int_{r}^{r_0} c(s) \, V(s) \, \psi(s)^{n-1} \, ds = \frac{1}{\psi(r)^{n-1}} \int_{r}^{r_0} \left( f(s) - f(r_0) \right) \psi(s)^{n-1} \, ds  \ge 0 
\end{equation*}
for all $ r \in (0,r_0] $, where
$$
c(r):=
\begin{cases}
\frac{\beta(v(r))-\beta(v(r_0))}{v(r)-v(r_0)} & \text{if } v(r) \neq  v(r_0)  \, , \\
0 & \text{if } v(r) = v(r_0) \, , 
\end{cases} 
$$
is a bounded and nonnegative coefficient in view of all of the assumptions. Hence, as a straightforward application of Lemma \ref{1-d-lemma}, we deduce that $ V = 0 $ and thus $ v(r)=v(r_0) $ for all $ r \in [0,r_0] $.  If $ v(r_0)=0 $ then $ v = 0 $ in the whole $ [0,R] $, but this is in contradiction with the fact that $ \beta(0)=0 $ and $ f $ is nontrivial. If $ v(r_0)>0 $, which is possible only for $ r_0<R $, then upon setting $ W(r) := v(r_0) - v(r) \ge 0 $ for all $ r \in [r_0,R]$ and reasoning similarly to above, we end up with  
\begin{equation*} \label{psi-1-hopf-bis}
- W'(r) + \frac{1}{\psi(r)^{n-1}} \int_{r_0}^{r} c(s) \, W(s) \, \psi(s)^{n-1} \, ds = \frac{1}{\psi(r)^{n-1}}  \int_{r_0}^{r} \left(  f(r_0) - f(s) \right) \psi(s)^{n-1} \, ds  \ge 0 
\end{equation*}
for all $ r \in [r_0,R] $. It is straightforward to check that the function $r \mapsto W(-r)$ still fulfills the assumptions of Lemma \ref{1-d-lemma} on the interval $ (-R,-r_0] $, whence $ v(r)=v(r_0) $ also for all $ r \in [r_0,R) $, which is however inconsistent with the homogeneous Dirichlet condition $ v(R) = 0 $. 
\end{proof}

The following result is a core tool in the proof of the parabolic concentration comparison theorem. Indeed, it provides an elliptic symmetrization estimate for homogeneous Dirichlet problems settled in centered geodesic balls, which will be crucially exploited in discretization scheme for the parabolic problems according to Proposition \ref{exuni-weak-loc}\eqref{u-disc}.

For convenience, in the next statement, we will consider \emph{weak solutions} to problems of the type \eqref{ellip-1} written in the form $ -\Delta \phi(u) + u = f $, which simply means that $v=\phi(u)$ is a weak solution as in Definition \ref{weak-sol}, corresponding to $\beta = \phi^{-1}$. 
 
\begin{proposition}\label{conc-ellip}
Let $R>0$, $ \phi: [0,+\infty) \to [0,+\infty) $ be a $C^1$ bijection with $ \phi(0)=0 $, $ \phi'>0 $, and $ f \in L^\infty(B_R) $ with $f \ge 0$. Let $u$ and $\overline{u}$ be the weak solutions of the Dirichlet problems
\begin{equation*}
\label{problem1-diri}
\begin{cases}
-\Delta \phi(u) + u =f  & \text{in } B_R \, , \\
\phi(u)=0 & \text{on } \partial B_R \, ,
\end{cases}
\end{equation*}
and
\begin{equation*}
\label{problem1}
\begin{cases}
-\Delta  \phi(\overline{u}) + \overline{u} = \overline{f} & \text{in } B_R \, , \\
\phi(\overline{u})= 0 & \text{on } \partial B_R \, ,
\end{cases}
\end{equation*}
respectively, where $ \overline{f} \in L^\infty(B_R)$, $ \overline{f} \ge 0 $, is radially nonincreasing and satisfies 
\begin{equation*}\label{prec-f}
f^{\star} \prec \overline{f} \, .
\end{equation*}
Then, if $ \HH^n$ supports the P\'olya-Szeg\H{o} inequality, in the sense of Definition \ref{def-polya}, we have
\begin{equation}\label{prec-discr}
u^{\star} \prec \overline{u} \, .
\end{equation}
\end{proposition}
\begin{proof}
First of all, upon setting $\beta(\rho)=\phi^{-1}(\rho)$, $v=\phi(u)$, and $\overline{v}=\phi(\overline{u})$, we observe that $ v $ and $ \overline{v} $ are the weak solutions of the Dirichlet problems 
\begin{equation}\label{changeofvar}
\begin{cases}
-\Delta v+ \beta(v) =f  & \text{in } B_R  \\
v=0 & \text{on } \partial B_R 
\end{cases}
\qquad \text{and} \qquad
\begin{cases}
-\Delta \overline{v} + \overline{v} = \overline{f}  & \text{in } B_R  \\
\overline{v}=0 & \text{on } \partial B_R 
\end{cases}
\, ,
\end{equation}
respectively, in the sense of Definition \ref{weak-sol}.  

Now, for the first part of the proof, we follow the strategy of \cite{Ngo}. To this end, let us fix $t,h>0$ and consider the following test function:
$$
w(x) := \frac{1}{h} \, \mathcal{G}_{t,h}(v(x)) \, ,
$$ 
where $\mathcal{G}_{t,h}(\theta)$ is defined by
$$
\mathcal{G}_{t,h}(\theta) :=
\begin{cases}
h \qquad & \text{if } \theta>t+h \, , \\
\theta-t & \text{if } t < \theta \leq t+h \, , \\
0  & \text{if } \theta\leq t \, .
\end{cases}
$$ 
It is easily checked that $ \mathcal{G}_{t,h}(v) \in H^1_0(B_R) $, $ \left( \mathcal{G}_{t,h}(v) \right)^\star =  \mathcal{G}_{t,h}(v^\star) \in H^1_0(B_R) $, and
$$ 
\nabla  \mathcal{G}_{t,h}(v) = \chi_{\left\{  t<v <t+h \right\} } \, \nabla v \, , \qquad \nabla  \mathcal{G}_{t,h}(v^\star) = \chi_{\left\{  t< v^\star <t+h \right\} } \, \nabla v^\star \, .
$$ 
Hence, inserting $ w $ in the weak formulation of the problem satisfied by $v$ (recall \eqref{weak-sol-ell}), we obtain the identity
\[
\frac{1}{h} \int_{\{t<v\leq t+h\}}|\nabla v|^{2}\,dV=\int_{\{v>t+h\}}\left(f-\beta(v) \right) dV+\int_{\{t<v\leq t+h\}} \frac{v-t}{h} \left(f -\beta(v) \right) dV \, .
\]
Next we use the P\'olya-Szeg\H{o} inequality \eqref{P-Z} (via the usual extension of $ H^1_0(B_R) $ functions to $ H^1_c(\HH^n) $), which entails
\begin{equation}\label{h-ratio}
\begin{aligned}
& \, \frac{\int_{\{v^\star>t\}}|\nabla v^{\star}|^{2}\,dV - \int_{\{v^{\star}>t+h\}}|\nabla v^{\star}|^{2}\,dV}{h} \\
= & \, \frac{1}{h}\int_{\{t<v^{\star}\leq t+h\}}|\nabla v^{\star}|^{2}\,dV \le \int_{\{v>t+h\}}\left(f-\beta(v) \right) dV+\int_{\{t<v\leq t+h\}} \frac{v-t}{h}  \left(f -\beta(v) \right) dV \, .
\end{aligned}
\end{equation}
From \eqref{h-ratio} and the boundedness of $ f $ and $v$, it is easy to see that the function $ t \mapsto \int_{\{t<v^\star\}}|\nabla v^{\star}|^{2}\,dV $ is Lipschitz. In particular, letting $h \to 0^+$, we find
\begin{equation} \label{firstestimrearrv}
-\frac{d}{dt}\int_{\{v^{\star}>t\}}|\nabla v^{\star}|^{2} \, dV \leq \int_{\{v>t\}}\left(f-\beta(v)\right)dV  \qquad \text{for a.e. } t \in (0,\| v \|_\infty) \, .
\end{equation}
Since $ v \in H^1_0(B_R) $, thanks to Lemma \ref{lem-cpt-h1} we have  $ v^\star \in H^1_0(B_R) $ as well. Due to the relation \eqref{from-ast-to-star}, and from the fact that $v^\star$ is bounded and nonincreasing, we then deduce that $ v^* $ is absolutely continuous in the whole interval $ [0,V(B_R)] $. 
Also, from the definition of $ \mu_v $, it follows that the identity
\begin{equation}\label{inv-mu}
   v^*(\mu_v(t)) = t 
\end{equation}
holds for all $ t \in [0,\| v \|_\infty] $. Because $ \mu_v(t) $ is a monotone function, it is differentiable almost everywhere in $t$. Recalling that absolutely continuous functions map negligible sets into negligible sets, we can claim that if $ B \subset [0,V(B_R)] $ is a negligible set then
\begin{equation*}\label{neglig}
\mu_v^{-1}(B) =  \left\{ t \in [0,\| v \|_\infty] : \ \mu_v(t) \in B  \right\} \subset v^*(B)
\end{equation*}
is also negligible. By choosing $B$ as the set where $ v^* $ is not differentiable, we immediately deduce that for almost every $t$ the value $ \mu_v(t) $ is such that $ v^* $ is differentiable at $s=\mu_v(t)$. Therefore, we are allowed to take the almost-everywhere derivative in \eqref{inv-mu} to obtain the key identity
\begin{equation}\label{key-der}
 \mu_v'(t) \, (v^*)'(\mu_v(t))   = 1 \qquad \text{for a.e. } t \in (0 , \| v \|_\infty)  \, .
\end{equation}

Let us go back to \eqref{firstestimrearrv} and explicitly compute the derivative on the left-hand side. In order to do so, we recall that the set $\left\{v^{\star}>t\right\}$ is precisely the centered geodesic ball $B_{\rho(t)}$ of radius $\rho(t)$ provided by formula \eqref{exprrho}, namely
\begin{equation}\label{rho-t}
\rho(t)=G^{-1}\!\left(\frac{\mu_{v}(t)}{\omega_{n}}\right) ,
\end{equation}
where $G^{-1}$ is the inverse function of \eqref{eq-G}. This yields
\begin{equation}\label{rho-t-der}
\rho^{\prime}(t)=\frac{\mu_{v}^{\prime}(t)}{\omega_{n} \, G^{\prime}(\rho(t))}
=\frac{\mu_{v}^{\prime}(t)}{\omega_{n} \, \psi(\rho(t))^{n-1}} \, .
\end{equation}
Since $v^\star$ is radial, we have 
\[
\int_{\{v^{\star}>t\}}|\nabla v^{\star}|^{2}\,dV = \int_{B_{\rho(t)}}|\nabla v^{\star}|^{2}\,dV 
= \omega_{n}\int_{0}^{\rho(t)}\left|
\left(v^{\ast}\right)^{\prime}\!\left(V(B_r)\right)
\omega_{n} \, \psi(r)^{n-1} 
\right|^{2} \psi(r)^{n-1} \, dr \, .
\]
Due to the above discussed properties of $ \mu_v(t) $, and thus of $ \rho(t) $, it is feasible to take the almost-everywhere derivative on the right-hand side, since for almost every $t$ we can make sure that $\rho(t)$ is a differentiability point of the function
$$
\rho \mapsto \int_{0}^{\rho}\left|
\left(v^{\ast}\right)^{\prime}\!\left(V(B_r)\right)
\omega_{n} \, \psi(r)^{n-1} 
\right|^{2} \psi(r)^{n-1} \, dr \, .
$$
Hence we obtain, for almost every $t$, 
\begin{equation*}\label{computderivenerg}
    -\frac{d}{dt}\int_{\{v^{\star}>t\}}|\nabla v^{\star}|^2 \,dV
=-\omega_{n}\,\rho^{\prime}(t)\, \psi(\rho(t))^{n-1}
\left|
\left(v^{\ast}\right)^{\prime}\!\left(V\!\left(B_{\rho(t)}\right)\right)
\omega_{n} \, \psi(\rho(t))^{n-1}
\right|^{2} ,
\end{equation*}
which, in view of \eqref{key-der}, \eqref{rho-t}, and \eqref{rho-t-der},  reads
\begin{equation}\label{computderivenerg2}
     -\frac{d}{dt}\int_{\{v^{\star}>t\}}|\nabla v^{\star}|\,dV
= -\omega_{n}^2 \, \psi(\rho(t))^{2(n-1)} \, \frac{1}{\mu_{v}^{\prime}(t)} \, .
\end{equation}
Then, a simple use of the Hardy-Littlewood inequality \eqref{h-litt} along with the fact that 
\[
\int_{\{v>t\}}\beta(v) \, dV = \int_{0}^{\mu_{v}(t)}\beta(v^{\ast}(s)) \, ds
\]
yield
\[
\int_{\{v>t\}}\left(f-\beta(v)\right) dV \leq \int_{0}^{\mu_{v}(t)} \left(f^{\ast}(s)-\beta(v^{\ast}(s)) \right) ds \, .
\]
By combining \eqref{firstestimrearrv} with \eqref{computderivenerg2} we arrive at the crucial estimate
\begin{equation}\label{mossino}
-\omega_{n}^2 \, \psi(\rho(t))^{2(n-1)} \, \frac{1}{\mu_{v}^{\prime}(t)} \leq \int_{0}^{\mu_{v}(t)} \left(f^{\ast}(s)-\beta(v^{\ast}(s)) \right) ds =\int_{B_{\rho(t)}} \left(f^{\star}-\beta(v^{\star}) \right) dV \, ,
\end{equation}
still for almost every $t$. 

We now argue as in \cite{Mossino}. First of all, upon setting  
\[
F(s):=\int_{0}^{s} \left[f^{\ast}(\sigma)-\beta(v^{\ast}(\sigma))\right] d\sigma \qquad \forall s \in [0,V(B_R)] \, ,
\]
and using \eqref{rho-t}, inequality \eqref{mossino} can be written in the form
\begin{equation}\label{aei}
1 \leq - \mu_{v}^{\prime}(t) \, \frac{F(\mu_{v}(t))}{\omega_n^2 \left[\psi\!\left(G^{-1}\!\left(\frac{\mu_{v}(t)}{\omega_{n}}\right)\right)\right]^{2(n-1)}} \qquad \text{for a.e. } t \in (0,\| v \|_\infty) \, .
\end{equation}
In particular, since $ \mu_v $ is right continuous, the above inequality ensures that $ F $ is nonnegative, at least on the image of $ \mu_v $ and therefore on its closure by continuity. On the other hand, a certain value $ s \in (0,V(B_R)) $ \emph{does not} belong to the image of $ \mu_v $ if and only if it lies in a flat zone of $v^*$, which means that $ s \in (s_-,s_+) $, $v^*$ is constant on such an interval, and the values $ s_-<s_+ $ both belong to the closure of the image of $ \mu_v $. If, by contradiction, $ F(s)<0 $, then $ f^*(\sigma) - \beta(v^*(\sigma)) $ is necessarily negative at some $ \sigma \in (s_-,s)$, since we know that $ F(s_-) \ge 0 $. However, because $ f^* $ is nonincreasing and $ v^* $ is constant on $ (s_-,s_+) $, this would mean that $ f^*(\sigma) - \beta(v^*(\sigma)) $ stays negative in the whole $ (s,s_+) $ and thus $ F(s_+)<0 $, a contradiction. In conclusion, we have shown that $F \ge 0$ everywhere.
Let us then consider the function
$$
\mathcal{F}(t) := \int_0^{\mu_v(t)} \frac{F(\sigma)}{\omega_n^2 \left[\psi\!\left(G^{-1}\!\left(\frac{\sigma}{\omega_{n}}\right)\right)\right]^{2(n-1)}} \, d\sigma \qquad \forall t \in [0,\| v \|_\infty] \, .
$$
From the above established properties of $\mu_v$ and $F$, it is plain that $ \mathcal{F} $ is bounded and nonincreasing, and its almost-everywhere derivative reads
\begin{equation}\label{m1}
 \mathcal{F}'(t) = \mu_{v}^{\prime}(t) \, \frac{F(\mu_{v}(t))}{\omega_n^2 \left[\psi\!\left(G^{-1}\!\left(\frac{\mu_{v}(t)}{\omega_{n}}\right)\right)\right]^{2(n-1)}} \le -1 \, ,
\end{equation}
the rightmost inequality following from \eqref{aei}. Moreover, because it is a $BV$ function, there exists a finite, negative Radon measure $ \nu $ such that
$$
\mathcal{F}(t) = \mathcal{F}(0)+ \nu([0,t]) \qquad \forall t \in [0,\| v \|_\infty] \, ,
$$
which satisfies $  \nu \le \mathcal{F}' $ in the sense of measures. Hence, taking any two values $t^{\prime}>t$ and integrating \eqref{m1} over $[t,t^{\prime}]$, we obtain
\[
t^{\prime}-t \leq - \int_t^{t'} \mathcal{F}'(\theta) \, d\theta \leq \mathcal{F}(t)-\mathcal{F}(t') = \dint_{\mu_{v}(t^{\prime})}^{\mu_{v}(t)} \frac{F(\sigma)}{\omega_n^2\left[\psi\!\left(G^{-1}\!\left(\frac{\sigma}{\omega_{n}}\right)\right)\right]^{2(n-1)}} \, d\sigma \, .
\]
Next, let us pick any two $0<s'<s$ such that $v^{\ast}(s^{\prime})>v^{\ast}(s)$ and set $t^{\prime}=v^{\ast}(s^{\prime})-\varepsilon$, $t=v^{\ast}(s)$, for some small enough $\varepsilon>0$. Since
\[
\mu_{v}(t^{\prime})=\mu_{v}(v^{\ast}(s^{\prime})-\varepsilon)\geq s^{\prime} \, , \qquad \mu_{v}(t)=\mu_{v}(v^{\ast}(s))\leq s \, ,
\] 
we have
\[
v^{\ast}(s^{\prime})-\varepsilon- v^{\ast}(s)\leq \int_{s^{\prime}}^{s} \frac{F(\sigma)}{\omega_{n}^2 \left[ \psi\!\left(G^{-1}\!\left(\frac{\sigma}{\omega_{n}}  \right) \right) \right]^{2(n-1)}} \, d\sigma \, ,
\]
so by letting $\varepsilon\rightarrow 0^+$ we end up with  
\[
v^{\ast}(s^{\prime})-v^{\ast}(s)\leq \int_{s^{\prime}}^{s} \frac{F(\sigma)}{\omega_{n}^2 \left[ \psi\!\left(G^{-1}\!\left(\frac{\sigma}{\omega_{n}}  \right) \right) \right]^{2(n-1)}} \, d\sigma \, .
\]
At this point we can fix any $s>0$, set $s^{\prime}=s-h$, $h>0$, and divide by $h$ both sides of the previous inequality (assuming without loss of generality that $ v^*(s-h) > v(s) $): because $F$ is continuous, upon letting $ h \to 0^+  $ we obtain
\[
- \omega_{n}^2 \left[ \psi\!\left(G^{-1}\!\left(\frac{s}{\omega_{n}}  \right) \right) \right]^{2(n-1)} (v^{\ast})^{\prime}(s)\leq F(s) \qquad \text{for a.e. } s \in  (0,V(B_R)) \, .
\]
Observing that
\[
\rho(v^{\ast}(s))=G^{-1}\!\left(\frac{\mu_{v}(u^{\ast}(s))}{\omega_{n}}\right)\leq G^{-1}\!\left(\frac{s}{\omega_{n}}\right),
\]
we may infer 
\begin{equation}\label{estimats}
-\omega_{n}^2 \left[ \psi\!\left(\rho(v^{\ast}(s)\right)) \right]^{2(n-1)} (v^{\ast})^{\prime}(s)\leq F( s ) = \int_{B_{\rho(v^{\ast}(s))}}\left(f^{\star}(x)-\beta(v^{\star}(x))\right)dV \, ,
\end{equation}
at least at those points $ s>0 $ such that $ (v^*)'(s) < 0 $ (so that $ F(s) = F( \mu_v(v^*(s)))   $).  Setting $s=V(B_{r})$ and observing that, still at such points, we have
\[
\rho(v^{\ast}(V(B_{r})))=\rho(v^{\star}(r))=r
\]
and 
\begin{equation}\label{rel-v-ast-star}
(v^{\star})^{\prime}(r)=\omega_{n} \, \psi(r)^{n-1} \, (v^{\ast})^{\prime}(V(B_{r})) \, ,
\end{equation}
from \eqref{estimats} we deduce that
\begin{equation}\label{estimatr}
-\omega_{n} \, \psi(r)^{n-1} \, (v^{\star})^{\prime}(r)\leq F(V(B_r)) =  \int_{B_r}\left(f^{\star}-\beta(v^{\star})\right) dV \qquad \text{for a.e. } r \in (0,R) \, ,
\end{equation}
the extension of such an inequality to those radii such that $ (v^{\star})^{\prime}(r) = 0 $ simply following from the non-negativity of $F$.

 Now an important remark. Owing to Corollary \ref{decr-min} and Proposition \ref{decreas}, the solution $\overline{v}$ is $C^1$ and radially decreasing, with strictly negative radial derivative away from the origin; in particular, it \emph{cannot have} any flat zone (without loss of generality we can assume that $f$ is nontrivial). Moreover, it is straightforward to check that all of the above inequalities used to derive  \eqref{estimatr}, because of the $C^1$ radially decreasing monotonicity, actually become \emph{identities}. As a result, we have 
\begin{equation}\label{estimatvr}
-\omega_{n} \, \psi(r)^{n-1} \, (\overline{v})^{\prime}(r) = \int_{B_r}\left(\overline{f}-\beta(\overline{v}) \right)dV \qquad \forall r \in (0,R] \, .
\end{equation}
Then, subtracting \eqref{estimatvr} from \eqref{estimatr}, we arrive at the inequality
\begin{equation}\label{boundeddomains}
-\omega_{n} \, \psi(r)^{n-1} \left(v^\star-\overline{v}\right)^{\prime}\!(r) + A(r) \leq \int_{B_r}\left(f^{\star}-\overline{f} \right) dV \qquad \text{for a.e. } r \in (0,R) \, ,
\end{equation}
where we define
\[
A(r) := \int_{B_r} \left(\beta(v^{\star})-\beta(\overline{v}) \right) dV  \, .
\]
Since $f^{\star}\prec \overline{f}$ by assumption,
from \eqref{boundeddomains} we infer
\begin{equation}\label{Afrac}
 \left(v^\star-\overline{v}\right)^{\prime}\!(r)
\geq \frac{A(r)}{\omega_{n} \, \psi(r)^{n-1}} \qquad \text{for a.e. } r \in (0,R) \, .
\end{equation}
We wish to prove that $A\leq 0$ everywhere. To this end, assume by contradiction that $A$ achieves a positive maximum at some $ r_0 \in (0,R] $. Due to the Lipschitz regularity of both $ v^\star $ and $ \overline{v} $, it is plain that $ A $ is a $ C^1 $ function. Hence, if $ r_0 \in (0,R) $, it necessarily holds
\[
A^{\prime}(r_{0})=0 \, ,
\]
which implies
\begin{equation}\label{der-v}
v^{\star}(r_{0}) = \overline{v}(r_{0}) \, .
\end{equation}
On the other hand, since $A(r_{0})>0$, we have
\[
\frac{A(r)}{\omega_{n} \, \psi(r)^{n-1}} >
\varepsilon  \qquad \forall  r \in [r_{0}-\delta,r_{0}+\delta] \, ,
\]
for some small enough $\varepsilon>0$ and $\delta>0$. Therefore, from \eqref{Afrac}, it follows that
\[
 \left(v^\star-\overline{v}\right)^{\prime}\!(r) > \varepsilon \qquad \text{for a.e. } r \in (r_{0}-\delta,r_{0}+\delta) \, ;
\]
hence, by integrating over $ (r_0,r) $, for an arbitrarily fixed $ r \in (r_0,r_0+\delta) $, we find
\[
v^\star(r)-\overline{v}(r) > v^\star(r_0)-\overline{v}(r_0) + \varepsilon \, (r-r_{0}) \, ,
\]
which, recalling \eqref{der-v}, entails
\[
v^\star(r)-\overline{v}(r) > \varepsilon(r-r_{0})>0 \, ,
\]
whence, in view of the strict monotonicity of $\beta$,
\[
\beta(v^\star(r)) >  \beta(\overline{v}(r)) \, .
\]
Because such a strict inequality holds for all $ r \in (r_0,r_0+\delta) $, we deduce that
\[
A(r_{0}+\delta)-A(r_{0})=\int_{B_{r_0+\delta} \setminus B_{r_{0}}} \left( \beta(v^{\star})-\beta(\overline{v}) \right) dV>0 \, ,
\]
which clearly contradicts the fact that $r_{0}$ is a maximum point for $A(r)$. We are thus left with the case $r_0=R$. Arguing exactly as before, we obtain
\[
 \left(v^\star-\overline{v}\right)^{\prime}\!(r) > \varepsilon \qquad \text{for a.e. } r \in (R-\delta,R) \, ,
\]
still for some small enough $ \varepsilon>0 $ and $ \delta>0 $. By virtue of the Dirichlet boundary conditions in \eqref{changeofvar}, we have $ v^\star(R) = \overline{v}(R) = 0 $, so that by integrating the above inequality over $ (r,R) $, for an arbitrarily fixed $ r \in (R-\delta,R) $, we find
\[
-v^\star(r)+\overline{v}(r) > \varepsilon \, (R-r) \, ,
\]
which implies
\[
v^\star(r) < \overline{v}(r) - \varepsilon \, (R-r) < \overline{v}(r) \, ,
\]
that is, again exploiting the strict monotonicity of $\beta$, 
\[
\beta(v^\star(r)) <  \beta(\overline{v}(r)) \, .
\]
The inequality being true for all $ r \in (R-\delta,R) $, it follows that
\[
A(R)-A(R-\delta)=
\int_{B_R \setminus B_{R-\delta}}\left( \beta(v^{\star})-\beta(\overline{v}) \right) dV<0 \, ,
\]
hence a contradiction is reached in this case too. We have thus shown that $A(r)\leq0$ for all $r \in [0,R]$, namely the thesis \eqref{prec-discr}.
\end{proof}

Now we are in position to finally prove Theorem \ref{th-conc}.
\begin{proof}[Proof of Theorem \ref{th-conc}: \eqref{TA} $ \Rightarrow $ \eqref{TB}]
\ \smallskip \\
We proceed by means of some approximation steps. To begin with, in agreement with the existence proof of Proposition \ref{exuni-weak}, for each $R>0$ let us denote by $ u_R $ the weak solution of the Cauchy-Dirichlet problem \eqref{filt-eq-local} that takes $ R \wedge u_0 $ as its initial datum. Also, let us denote by $ \overline{u}_R $ the weak solution of the same problem taking the Schwarz rearrangement of $ R \wedge u_0 $ (restricted to $B_R$) as its initial datum, which clearly coincides with $ \left( (R\wedge u_0) \cdot \chi_{B_R} \right)^\star $. Our first goal is to prove, at least when $ \phi $ complies with the assumptions of Proposition \ref{conc-ellip}, that 
\begin{equation}\label{conc-approx-1}
u_R^\star(\cdot,t) \prec \overline{u}_R(\cdot,t) \qquad \forall t>0 \, ,
\end{equation}
where the inequality is tacitly understood in the whole of $\HH^n$ up to the usual extension of the solutions to zero outside $ B_R $. To this end, for all $h>0$, we call  $ u^{(h)} $ and $ \overline{u}^{(h)} $ the piecewise-constant curves defined as in \eqref{pcint} by solving recursively the Dirichlet problems \eqref{eq-ui-rec} starting from $ w_0 \equiv  R \wedge u_0 $ and $ w_0 \equiv \left( (R\wedge u_0) \cdot \chi_{B_R} \right)^\star $, respectively. Owing to Proposition \ref{exuni-weak-loc}\eqref{u-disc}, we know in particular that for all $ t>0 $
\begin{equation}\label{discr-cont}
\lim_{h \to 0^+} \left\| u^{(h)}(\cdot,t) - u_R(\cdot,t) \right\|_{L^1(B_R)} = 0 \qquad \text{and} \qquad \lim_{h \to 0^+} \left\| \overline{u}^{(h)}(\cdot,t) - \overline{u}_R(\cdot,t) \right\|_{L^1(B_R)} = 0 \, .
\end{equation}
Moreover, by recursively applying Proposition \ref{conc-ellip}, we can deduce that for all $h>0$
\begin{equation}\label{conc-approx-h}
\left[u^{(h)}\right]^\star \! (\cdot,t) \prec \overline{u}^{(h)}(\cdot,t) \qquad \forall t>0 \, .
\end{equation}
Hence, thanks to \eqref{discr-cont} and the fact that Schwarz rearrangements are stable with respect to $ L^1 $ convergence (recall \eqref{nonexp-schw}), we easily obtain \eqref{conc-approx-1} by taking the limit in \eqref{conc-approx-h} as $ h \to 0^+ $. Let us remove the extra assumptions on $\phi$ of Proposition \ref{conc-ellip}. Thanks to Proposition \ref{approx-phireg}, and using the same notations, we know that $ \phi $ can be approximated by a sequence $\{ \phi_k \} $ of nonlinearities complying (in particular) with the hypotheses of Proposition \ref{conc-ellip}, so that the corresponding solutions satisfy for all $k \in \mathbb{N}$
\begin{equation}\label{conc-approx-1-k}
u_{R,k}^\star(\cdot,t) \prec \overline{u}_{R,k}(\cdot,t) \qquad \forall t>0 \, .
\end{equation}
Owing to Proposition \ref{Propconves} and the fact that $ \overline{u}_{R,k}(\cdot,t) $ is radially nonincreasing, we notice that \eqref{conc-approx-1-k} can equivalently be rewritten as 
\begin{equation*}\label{conc-approx-1-L}
\int_{\HH^{n}} {u}_{R,k}(x,t)  \, h(x) \, dV(x) \le \int_{\HH^{n}} \overline{u}_{R,k}(x,t) \, h^\star(x) \, dV(x) \, ,
\end{equation*}
for every nonnegative $ h \in L^\infty(\HH^n) $. Taking the limit in $k$, exploiting \eqref{unif-L1}, we end up with 
\begin{equation*}\label{conc-approx-1-L-lim}
\int_{\HH^{n}} {u}_{R}(x,t)  \, h(x) \, dV(x) \le \int_{\HH^{n}} \overline{u}_{R}(x,t) \, h^\star(x) \, dV(x) \, ,
\end{equation*}
which is precisely \eqref{conc-approx-1} for all $\phi$ as in the statement.

Now, we aim at passing to the limit in \eqref{conc-approx-1} as $R \to +\infty$; from Corollary \ref{conv-L1} we have 
$$
\lim_{R \to +\infty} \left\|  u_R(\cdot,t) - u(\cdot,t) \right\|_{L^1(\HH^n)} =0 \qquad \forall t>0 \, ,
$$
which, still in view of \eqref{nonexp-schw}, entails 
\begin{equation}\label{conc-approx-2}
\lim_{R \to +\infty} \left\|  u_R^\star(\cdot,t) - u^\star(\cdot,t) \right\|_{L^1(\HH^n)} =0 \qquad \forall t>0 \, .
\end{equation}
On the other hand, still Corollary \ref{conv-L1} ensures that $\overline{u}$ is the monotone limit in $R$ of the weak solutions $\left\{ \overline{v}_{R} \right\} $ of the Cauchy-Dirichlet problems with initial data data $ \left(R\wedge u_{0}^{\star}\right) \cdot \chi_{B_R}$, thus for all $t>0$ we have the pointwise inequality $ \overline{v}_{R}  \le \overline{u}  $. Moreover, it is plain that
\[
\left((R\wedge u_{0}) \cdot \chi_{B_R}\right)^{\star} \le \left(R\wedge u_{0}^{\star}\right) \cdot \chi_{B_R} \, ,
\]
hence the contraction principle in Proposition \ref{approx-phireg} yields $\overline{u}_{R} \le \overline{v}_{R}$, therefore it follows that $ \overline{u}_R \le \overline{u}$. As a consequence, \eqref{conc-approx-1} implies
\begin{equation}\label{conc-approx-3}
u_R^\star(\cdot,t) \prec \overline{u}(\cdot,t) \qquad \forall t>0 \, ,
\end{equation}
and upon taking limits in \eqref{conc-approx-3} as $R \to +\infty$, using \eqref{conc-approx-2}, we obtain the claimed concentration inequality \eqref{conc-prop}.
\end{proof}

\begin{proof}[Proof of Theorem \ref{th-conc}: \eqref{TB} $ \Rightarrow $ \eqref{TA}]
\ \smallskip  \\ 
First of all, we observe that, given an arbitrary nonnegative $ v \in H^1_c(\HH^n) \cap L^\infty(\HH^n)  $ and any $ h>0$, upon choosing $ \xi(x,t) = \chi_{[0,h]}(t) \, v(x) $ as a test function in \eqref{weak-form} (after a standard approximation of $ \chi_{[0,h]} $ and $ v $ with smooth functions), we deduce the following identity: 
\begin{equation}\label{weak-form-1}
-  \int_{\HH^n} u(x,h) \, v(x) \, dV(x) =  \int_{\HH^n}  \left\langle \nabla v , \int_0^{h} \nabla \phi(u) \, dt \right\rangle dV - \int_{\HH^n} u_0 \, v \, dV \, .
\end{equation}
Similarly, if with repeat the same computations with $u $ replaced by $\overline{u} $ and $ v $ replaced by $ v^\star$, we obtain 
\begin{equation}\label{weak-form-symm}
-  \int_{\HH^n} \overline{u}(x,h) \, v^\star(x) \, dV(x) =  \int_{\HH^n}  \left\langle \nabla v^\star , \int_0^{h} \nabla \phi(\overline{u}) \, dt \right\rangle dV - \int_{\HH^n} u_0^\star \, v^\star \, dV \, .
\end{equation}
Note that $ v^\star $ also belongs to $ H^1_c(\HH^n) $ by virtue of Lemma \ref{lem-cpt-h1}, whereas its uniform boundedness is just a consequence of the definition of Schwarz rearrangement.
Since $ u(\cdot,h) \prec \overline{u}(\cdot,h) $ by assumption and we know that $ \overline{u}(\cdot,h) $ is radially nonincreasing (Proposition \ref{exuni-weak}\eqref{rad-inc}), thanks to Proposition \ref{Propconves} and the Hardy-Littlewood inequality \eqref{h-litt} we can infer that
\begin{equation}\label{pol-1}
  \int_{\HH^n} u(x,h) \, v(x) \, dV(x) \le \int_{\HH^n} u^\star(x,h) \, v^\star(x) \, dV(x)  \le \int_{\HH^n} \overline{u}(x,h) \, v^\star(x) \, dV(x) \, . 
\end{equation}
Now, we temporarily make the additional assumption $\| v  \|_{L^\infty(\HH^n)} < \ell$, where we recall that $\ell$ has been defined in \eqref{ell-lim}. It is then feasible to pick $ u_0 = \phi^{-1}_l(v) $, which belongs to $ L^1(\HH^n) $ since $ v $ is uniformly bounded with compact support and $ \phi^{-1}_l(0) = 0 $. For the same reason, it is plain that also $ \Phi(u_0)  \in L^1(\HH^n) $. Still from the definition of $ \phi^{-1}_l $, and in particular the fact that it is (strictly) increasing, we can deduce that (recall \eqref{identity-int} and the subsequent discussion)
$$ 
u_0^\star = \left( \phi^{-1}_l(v) \right)^\star = \phi^{-1}_l(v^\star) \, , $$
whence $   \phi(u_0^\star) = v^\star \in H^1_c(\HH^n) $. Therefore, by virtue of Proposition \ref{exuni-weak}\eqref{energy-decreas}, we have  
\begin{equation}\label{cont-H1}
\lim_{t \to 0^+} \left\| \nabla \phi(u(\cdot,t)) - \nabla v \right\|_{L^2(\HH^n)} = 0 \qquad \text{and} \qquad  \lim_{t \to 0^+} \left\| \nabla \phi(\overline{u}(\cdot,t)) - \nabla v^\star \right\|_{L^2(\HH^n)} = 0 \, .
\end{equation}
Furthermore, identity \eqref{identity-int} with the function $ F(v) = \phi^{-1}_l(v) \, v  $ reads
\begin{equation}\label{pol-2}
\int_{\HH^n} u_0 \, v \, dV = \int_{\HH^n} \phi^{-1}_l(v) \, v \, dV = \int_{\HH^n} \phi^{-1}_l(v^\star) \, v^\star \, dV =  \int_{\HH^n} u_0^\star \, v^\star \, dV \, .
\end{equation}
As a result, by combining \eqref{weak-form-1}, \eqref{weak-form-symm}, \eqref{pol-1}, and \eqref{pol-2}, we end up with 
$$
\int_{\HH^n}  \left\langle \nabla v^\star , \frac 1 h  \int_0^{h} \nabla \phi(\overline{u}) \, dt \right\rangle dV \le \int_{\HH^n}  \left\langle \nabla v , \frac 1 h \int_0^{h} \nabla \phi(u) \, dt \right\rangle dV \qquad \forall h>0 \, ,
$$
so that by letting $ h \to 0^+ $ and using \eqref{cont-H1} we obtain
\begin{equation}\label{quasi-pol}
      \int_{\HH^n} \left| \nabla v^\star \right|^2 dV  \le \int_{\HH^n} \left| \nabla v \right|^2 dV \, . 
\end{equation}
Finally, the constraint $ \| v  \|_{L^\infty(\HH^n)} < \ell$ can be easily removed upon observing that \eqref{quasi-pol} is invariant under multiplication by constants, since $ (cv)^\star = c v^\star $ for any $c>0$. The thesis then follows from Proposition \ref{polya-compact-to-h1}.
\end{proof}

\begin{proof}[Proof of Theorem \ref{th-conc-rad}]
One can reproduce exactly the same proof as that of Theorem \ref{th-conc}, upon noticing that if $ u_0 $ is radial then the corresponding weak energy solution $ u(\cdot,t) $ to \eqref{filt-eqintr} is also radial for all $t>0$. This easily follows from the proof of Lemma \ref{ulteriori-prop-approx}\eqref{noninc}, where it is shown, in particular, that if $f$ is radial the weak solution to \eqref{ellip-1} is also radial, at least when $\beta$ is smooth. Then, in all of the subsequent approximation steps leading to the construction of $ u $ (Propositions \ref{weak-sol-comparison}, \ref{exuni-weak-loc}, and \ref{exuni-weak}), radiality is always preserved. Hence, when it comes to using or proving the P\'olya-Szeg\H{o} inequality as in the aforementioned Theorem \ref{th-conc}, it is enough to observe that only radial functions are involved at all stages.
\end{proof}

\begin{proof}[Proof of Proposition \ref{nazarov}]
To begin with, we need to recall the main result of \cite{BN}, in a simplified setting that is enough for our purposes, which is the key tool we rely on. Specifically, in \cite[Theorem 4]{BN} it is proved that if $\mathfrak{a} : [-1,1] \to [0,+\infty) $ is a continuous function satisfying
\begin{equation}\label{hyp-nazarov}
  \mathfrak{a}(1-t+s)  \le \mathfrak{a}(s) + \mathfrak{a}(t) \qquad \forall s,t \in [-1,1] : \ s \le t \, ,
\end{equation}
then the following one-dimensional P\'olya-Szeg\H{o}-type inequality holds:
\begin{equation}\label{one-dim-PZ}
\int_{-1}^{1} \left| \big(\mathsf{f}_{\#}\big)'(s) \right|^2 \mathfrak{a}(s)^2 \, ds \le \int_{-1}^{1} \left| \mathsf{f}'(s) \right|^2 \mathfrak{a}(s)^2 \, ds \qquad \forall \mathsf{f} \in W^{1,1}_\ell([-1,1]) \, ,    
\end{equation}
where we let $ W^{1,1}_\ell([-1,1]) $ denote the set of nonnegative absolutely continuous functions on $ [-1,1]  $ that vanish at the left endpoint $ s=-1 $, and $ \mathsf{f}_{\#} $ stands for the \emph{nondecreasing} rearrangement of $ \mathsf{f} $ on $ [-1,1] $ with respect to the standard Lebesgue measure. Now, by applying the change of variables
$$
y = \frac{R}{2}\,(1-s) \, ,
$$
where $ R>0 $ is an arbitrarily fixed parameter, it is not difficult to check that \eqref{one-dim-PZ} becomes 
\begin{equation}\label{one-dim-PZ-decr}
\int_{0}^{R} \left| \big(\mathsf{f}^{*}\big)'(y) \right|^2 \mathfrak{a}\!\left(\tfrac{R-2y}{R}\right)^2 dy \le \int_{0}^{R} \left| \mathsf{f}'(y) \right|^2 \mathfrak{a}\!\left(\tfrac{R-2y}{R}\right)^2 dy \qquad \forall \mathsf{f} \in W^{1,1}_{\mathsf{r}}([0,R]) \, ,    
\end{equation}
where in this case $ W^{1,1}_{\mathsf{r}}([0,R]) $ stands for the set of nonnegative absolutely continuous functions on $ [0,R] $ vanishing at the right endpoint $ y=R $ and $ \mathsf{f}^{*} $ for the usual \emph{nonincreasing} rearrangement of $\mathsf{f}$, that is, \eqref{f-sharp-int} when $ V $ is replaced by the one-dimensional Lebesgue measure on $[0,R]$. Note that \eqref{hyp-nazarov} is transformed into 
\begin{equation}\label{hyp-nazarov-change}
  \mathfrak{a}\!\left(1-\tfrac{R-2z}{R}+\tfrac{R-2y}{R}\right) \le \mathfrak{a}\!\left(\tfrac{R-2y}{R}\right) + \mathfrak{a}\!\left(\tfrac{R-2z}{R}\right) \qquad \forall y,z \in [0,R] : \ z \le y  \, ,
\end{equation}
As a result, upon redefining $ \mathfrak{a}(y) \equiv \mathfrak{a}\!\left(\tfrac{R-2y}{R}\right) $, we observe that \eqref{hyp-nazarov-change} reads
\begin{equation}\label{hyp-nazarov-change-bis}
  \mathfrak{a}\!\left(y-z\right) \le \mathfrak{a}\!\left(y\right) + \mathfrak{a}\!\left(z\right) \qquad \forall y,z \in [0,R] : \ z \le y  \, ,
\end{equation}
whereas \eqref{one-dim-PZ-decr} in turn becomes 
\begin{equation}\label{one-dim-PZ-decr-bis}
\int_{0}^{R} \left| \big(\mathsf{f}^{*}\big)'(y) \right|^2 \mathfrak{a}(y)^2 \, dy \le \int_{0}^{R} \left| \mathsf{f}'(y) \right|^2 \mathfrak{a}(y)^2 \, dy \qquad \forall \mathsf{f} \in W^{1,1}_{\mathsf{r}}([0,R]) \, .    
\end{equation}
Also, since \eqref{hyp-nazarov-change-bis}--\eqref{one-dim-PZ-decr-bis} are in fact independent of $R$, we can readily extend the validity of \eqref{one-dim-PZ-decr-bis} to arbitrary functions $ \mathsf{f} \in H^1_c([0,+\infty)) $, for any continuous function $ \mathfrak{a}:[0,+\infty) \to [0,+\infty)  $ complying with \eqref{hyp-nazarov-change-bis}. Our next goal is to reduce the radial P\'olya-Szeg\H{o} inequality to \eqref{one-dim-PZ-decr-bis}; recalling Proposition \ref{polya-compact-to-h1} (it holds in the radial setting as well with the same proof), it is enough to consider radial functions $ v \in H^1_c(\HH^n) $. Given any such a function, the validity of \eqref{P-Z} is equivalent to 
\begin{equation}\label{eq:radial-PZ}
     \int_0^{+\infty} \left| \left(v^\star\right)'\!\left(r\right) \right|^2 \psi(r)^{n-1} \,dr   \le \int_0^{+\infty} \left| v'(r) \right|^2 \psi(r)^{n-1} \,dr \, ,
\end{equation}
where, in agreement with \eqref{eq-G-vol}, \eqref{f-sharp-int}, and \eqref{from-ast-to-star}, we let $ v^\star $ denote the radially nonincreasing Schwarz rearrangement of $v$ with respect to the one-dimensional measure $ \psi(r)^{n-1} \,dr $. Similarly to the proof of Proposition \ref{iso-polya}, we may now introduce the change of variables 
$$
y=\int_0^r \psi(\tau)^{n-1} \, d\tau = G(r) \, ;
$$
upon defining
$$
\mathsf{f}(y) := v\!\left( G^{-1}(y) \right) ,
$$
it is plain that 
$$
\mathsf{f}^*(y) = v^\star\!\left( G^{-1}(y) \right) , 
$$
hence \eqref{eq:radial-PZ} can be written in terms of $\mathsf{f}$ as 
\begin{equation}\label{eq:radial-PZ-1}
     \int_0^{+\infty} \left| \left(\mathsf{f}^\ast\right)'\!\left(y\right) \right|^2 \left[ \psi\!\left( G^{-1}(y) \right)^{n-1}  \right]^2 dy \le \int_0^{+\infty} \left| \mathsf{f}'(y) \right|^2  \left[ \psi\!\left( G^{-1}(y) \right)^{n-1}  \right]^2 dy \, .
\end{equation}
Note that, since $ C^2_c(\HH^n) $ functions are dense in $ H^1_c(\HH^n) $, we may assume without loss of generality that $ v $ is also $ C^2_c(\HH^n) $, which easily implies that $\mathsf{f}$ is at least $C^1([0,+\infty))$. Because \eqref{prop-nazarov-G} is precisely \eqref{hyp-nazarov-change-bis} with
$$
\mathfrak a(y) = \psi\!\left( G^{-1}(y) \right)^{n-1} ,
$$
it follows that \eqref{eq:radial-PZ-1}, and therefore \eqref{eq:radial-PZ}, holds.

Rigorously, the equivalence between \eqref{eq:radial-PZ} and \eqref{eq:radial-PZ-1} has been established in the case where $\HH^n$ has infinite volume, however, we simply observe that if $ \mathsf{v}:=\omega_n^{-1}\,V(\HH^n)<+\infty $ then the only substantial difference in the above computations is that $ \mathsf{f} $ and the integrals in \eqref{eq:radial-PZ-1} are defined over $[0,\mathsf{R})$ as opposed to $[0,+\infty)$, where $ \mathsf{R}:=\lim_{y \to \mathsf{v}^-} G^{-1}(y) $.
\end{proof}

In order to proceed with the proof of Theorem  \ref{polya-failure}, we first need a preliminary lemma, which is a consequence of the asymptotic volume formula \eqref{chavel-vol}.

\begin{lemma}\label{lem-scal}
    Let $ B_{\rho}(\hat{o}) $ be the geodesic ball of radius $ \rho>0 $ centered at some $ \hat{o} \in \HH^n $. If $ B_R $ is the geodesic ball of radius $ R>0 $ having the same volume as $ B_{\rho}(\hat{o}) $, it holds
    \begin{equation}\label{radii}
      R  = \rho + \tfrac{S(o) - S(\hat o)}{6n(n+2)} \, \rho^3 + \mathcal{O}\!\left(\rho^4\right) .
    \end{equation}
 In particular,
        \begin{equation}\label{radii-surf}
\mathcal{V}_{n-1}\!\left(\partial B_R \right) = \omega_n \, \rho^{n-1}\left[ 1 + \left( - \tfrac {S(\hat{o})} {6n} + \tfrac{(n-1)(S(\hat{o}) - S(o))}{6n(n+2)} \right) \rho^2 + \mathcal{O}\!\left(\rho^3\right)\right] .
\end{equation}
\end{lemma}
\begin{proof}
From \eqref{chavel-vol}, we know that $R$ satisfies the approximate identity
\begin{equation}\label{exp-Rr}
R^{n} \left( 1 - \tfrac{S(o)}{6(n+2)} \, R^2 + \mathcal{O}\!\left( R^3 \right) \right)  =  \rho^{n} \left( 1 - \tfrac{S(\hat{o})}{6(n+2)} \, \rho^2 + \mathcal{O}\!\left( \rho^3 \right) \right) ,
\end{equation}
which, in particular, ensures at least that $ R = \rho + \mathcal{O}\!\left(\rho^2\right) $. This allows us to rewrite \eqref{exp-Rr} as 
$$
R^n = \rho^n \, \frac{ 1 - \tfrac{S(\hat{o})}{6(n+2)} \, \rho^2 + \mathcal{O}\!\left( \rho^3 \right)}{1 - \tfrac{S(o)}{6(n+2)} \, \rho^2 + \mathcal{O}\!\left( \rho^3 \right)} = \rho^{n} \left( 1 + \tfrac{S(o)-S(\hat{o})}{6(n+2)} \, \rho^2 + \mathcal{O}\!\left( \rho^3 \right) \right) ,
$$
whence
$$
R=\rho \left( 1 + \tfrac{S(o)-S(\hat{o})}{6(n+2)} \, \rho^2 + \mathcal{O}\!\left( \rho^3 \right) \right)^{\frac 1 n} = \rho \left( 1 +\tfrac{S(o)-S(\hat{o})}{6n(n+2)} \, \rho^2 + \mathcal{O}\!\left( \rho^3 \right) \right) ,
$$
that is \eqref{radii}. Upon substituting the latter formula into \eqref{chavel-area} (with $ \rho \equiv R $), and carrying out some standard expansions, we obtain:
$$
\begin{aligned}
\mathcal{V}_{n-1}\!\left(\partial B_R \right) 
= & \, \omega_n \left( \rho + \tfrac{S(o) - S(\hat o)}{6n(n+2)} \, \rho^3 + \mathcal{O}\!\left(\rho^4\right)\right)^{n-1} \left[ 1  - \tfrac{S(\hat{o})}{6n} \left( \rho + \tfrac{S(o) - S(\hat o)}{6n(n+2)} \, \rho^3 + \mathcal{O}\!\left(\rho^4\right) \right)^2 + \mathcal{O}\!\left( \rho^3 \right) \right] \\
= & \, \omega_n \, \rho^{n-1} \left[ 1 + (n-1) \, \tfrac{S(o) - S(\hat o)}{6n(n+2)} \, \rho^2 + \mathcal{O}\!\left( \rho^3 \right) \right] \left[ 1 - \tfrac{S(\hat{o})}{6n} \, \rho^2  + \mathcal{O}\!\left( \rho^3 \right) \right] \\
= & \, \omega_n \, \rho^{n-1} \left[ 1 + \left( - \tfrac {S(\hat{o})} {6n} + \tfrac{(n-1)(S(\hat{o}) - S(o))}{6n(n+2)} \right) \rho^2 + \mathcal{O}\!\left( \rho^3 \right)  \right] ,
\end{aligned}
$$
that is \eqref{radii-surf}.
\end{proof}

\begin{proof}[Proof of Theorem \ref{polya-failure}]
 Let us consider the following function, which belongs to $ H^1_c(\HH^n) $:
 $$
f(x) := \left( 1 - \tfrac{\mathrm{d}(x,\hat{o})}{\rho} \right)^+ \qquad \forall x \in \HH^n \, ,
 $$
 where $ \rho>0 $ is a fixed radius to be chosen small enough. Since $ f $ is supported in $ B_\rho(\hat{o}) $ and $ \left| \nabla f(x) \right| = \frac 1 \rho $ for all $ x \in B_\rho(\hat{o}) $, we have
 \begin{equation}\label{L2-grad-ex}
\int_{\HH^n} \left| \nabla f \right|^2 dV = \frac{1}{\rho^2} \, V\!\left( B_\rho(\hat{o}) \right) = \frac{\omega_n}{n} \, \rho^{n-2} \left( 1 - \tfrac{S(\hat{o})}{6(n+2)} \, \rho^2 +  \mathcal{O}\!\left( \rho^3 \right) \right) ,
 \end{equation}
 recalling the asymptotic formula for the volume \eqref{chavel-vol}. From the definition of $f$, it is easily seen that
 $$
 \left\{ x \in \HH^n : \ f(x)>t \right\} = B_{(1-t)\rho}(\hat{o}) \qquad \forall t \in [0,1) \, . 
 $$
In view of the definition of Schwarz rearrangement and \eqref{radii}, we obtain:
$$
\left\{ x \in \HH^n : \ f^\star(x)>t \right\} = B_{R_t} \qquad \forall t \in [0,1)
 $$
 with
 \begin{equation}\label{graph-rt}
R_t = (1-t)\rho + \tfrac{S(o) - S(\hat o)}{6n(n+2)} \, (1-t)^3 \rho^3 + \mathcal{O}\!\left((1-t)^4 \rho^4\right) .
 \end{equation}
Hence, formula \eqref{radii-surf} yields
 $$
\mathcal{V}_{n-1}\!\left(\partial B_{R_t} \right) 
=  \omega_n \left[ (1-t)^{n-1}\rho^{n-1} + \left( - \tfrac {S(\hat{o})} {6n} + \tfrac{(n-1)(S(\hat{o}) - S(o))}{6n(n+2)} \right) (1-t)^{n+1} \rho^{n+1} + \mathcal{O}\!\left((1-t)^{n+2}\rho^{n+2}\right)\right] .
 $$
 As a result, from the coarea formula (see \cite[Proposition 2.3]{MSo} and references therein) we find
 $$
 \begin{aligned}
 \int_{\HH^n} \left| \nabla f^\star \right| dV =  \int_0^1 \mathcal{V}_{n-1}\!\left(\partial B_{R_t} \right)  dt= \frac{\omega_n}{n} \, \rho^{n-1} \left[  1 + \left( - \tfrac {S(\hat{o})} {6(n+2)} + \tfrac{(n-1)(S(\hat{o}) - S(o))}{6(n+2)^2} \right) \rho^{2} + \mathcal{O}\!\left(  \rho^{3}  \right) \right] .
\end{aligned}
 $$
On the other hand, H\"older's inequality entails
\begin{equation}\label{mm-1}
\begin{aligned}
\int_{\HH^n} \left| \nabla f^\star \right|^2 dV  \ge  \frac{\left( \int_{\HH^n} \left| \nabla f^\star \right| dV   \right)^2}{V\!\left( B_{R_0} \right)} = \frac{\omega_n}{n} \, \rho^{n-2} \, \frac{\left[ 1 + \left( - \tfrac {S(\hat{o})} {6(n+2)} + \tfrac{(n-1)(S(\hat{o}) - S(o))}{6(n+2)^2} \right) \rho^{2} + \mathcal{O}\!\left(  \rho^{3}  \right) \right]^2}{ 1 - \tfrac{S(\hat{o})}{6(n+2)} \, \rho^2 +  \mathcal{O}\!\left( \rho^3 \right) } \, ,
\end{aligned}
\end{equation}
where we also used the property $ V\!\left(B_{R_0}\right) = V\!\left(B_\rho(\hat{o})\right) $. After some standard manipulations on the rightmost term of \eqref{mm-1},  it is not difficult to infer that
$$
\int_{\HH^n} \left| \nabla f^\star \right|^2 dV  \ge \frac{\omega_n}{n} \, \rho^{n-2} \left[ 1  +\left( - \tfrac {S(\hat{o})} {6(n+2)} + \tfrac{(n-1)S(\hat{o}) - S(o)}{3(n+2)^2} \right) \rho^{2} + \mathcal{O}\!\left( \rho^3 \right) \right] . 
$$
Since $ S(\hat{o}) > S(o) $ by assumption, comparing with \eqref{L2-grad-ex} we see that it is possible to pick $\rho$ so small that
$$
\int_{\HH^n} \left| \nabla f^\star \right|^2 dV  > \int_{\HH^n} \left| \nabla f \right|^2 dV \, , 
$$
which disproves the P\'olya-Szeg\H{o} inequality.
\end{proof}

\textbf{Acknowledgments.} The authors are members of the Gruppo Nazionale per l'Ana\-lisi Matematica, la Probabilit\`a e le loro Applicazioni (GNAMPA, Italy) of the Istituto Nazionale di Alta Matematica (INdAM, Italy). Their research was in part carried out within the project ``Geometric-Analytic Methods for PDEs and Applications (GAMPA)'', ref.\ 2022SLTHCE -- funded by European Union -- Next Generation EU within the PRIN 2022 program (D.D.\ 104 -- 02/02/2022 Ministero dell'Universit\`a e della Ricerca).


\appendix

\section{Proofs of auxiliary technical results}\label{aux}

\subsection{Proofs of Section \ref{prem-tool}}\label{aux-1} 

\begin{proof}[Proof of Lemma \ref{equiv}]
   Let $R>1$. In view of Proposition \ref{local poin}, we have 
    $$
    \begin{aligned}
\left\| \overline{v}_R \right\|_{L^2(B_1)} \le \left\| v \right\|_{L^2(B_1)} + \left\| v - \overline{v}_R \right\|_{L^2(B_1)} \le \left\| v \right\|_{L^2(B_1)} + \left\| v - \overline{v}_R \right\|_{L^2(B_R)} \le  \left\| v \right\|_{L^2(B_1)} + C_R \left\| \nabla v \right\|_{L^2(\HH^n)} ,
\end{aligned}
    $$
    whence 
    $$
    \begin{aligned}
\left\| v \right\|_{L^2(B_R)}  \le \left\| \overline{v}_R \right\|_{L^2(B_R)}  + \left\| v - \overline{v}_R \right\|_{L^2(B_R)} \le \frac{V(B_R)}{V(B_1)} \left( \left\| v \right\|_{L^2(B_1)} + C_R \left\| \nabla v \right\|_{L^2(\HH^n)} \right) +  C_R \left\| \nabla v \right\|_{L^2(\HH^n)} , 
\end{aligned}
    $$
   thus \eqref{loc-norm} holds \emph{e.g.}~with 
    $$
K_R = \frac{V(B_R)}{V(B_1)} \left( 1+C_R \right) + C_R \, .
    $$    
    On the other hand, if $ R \le 1 $ the inequality is trivially true from the definition of $\hnorm$ norm.
\end{proof}

\begin{proof}[Proof of Proposition \ref{lemma-norm}]
First of all, we observe that $ \| \cdot \|_\hnorm $ is indeed a norm on $ \hdot $ that induces a scalar product. The only nontrivial property to check is vanishing: however, if $ \| v \|_\hnorm = 0 $, then from Lemma \ref{equiv} we infer that $ v=0 $ in every ball $B_R$, therefore $v=0$ almost everywhere. As for separability, it is just a consequence of the separability of $ L^2(B_1) \times L^2(\HH^n ; T_x {\HH^n})  $ with respect to the graph norm, whereas the density of $ C_c^\infty(\HH^n) $ readily follows from \eqref{seq-def-h1}. Finally, in order to prove that such a norm makes $ \hdot $ complete, let $ \{ v_k \} \subset  \hdot $ be a Cauchy sequence. Thanks to Lemma \ref{equiv}, the same sequence is also Cauchy in $ L^2(B_R) $ for all $ R>0$. As a result, from the completeness of $ L^2 $ spaces, we can easily deduce that there exists $ v \in L^2_{\mathrm{loc}}(\HH^n) $ with $ \nabla v \in L^2(\HH^n) $ such that 
\begin{equation}\label{all-radii}
 \lim_{k \to \infty} \left\| v_k - v \right\|_{L^2(B_R)} = 0 \quad \forall R>0 \qquad \text{and} \qquad  \lim_{k \to \infty} \left\| \nabla v_k - \nabla v \right\|_{L^2(\HH^n)} = 0 \, .
\end{equation}
It remains to show that $ v \in \hdot $. To this end, we simply observe that, in view of the definition of $ \hdot $, we can find a sequence $ \{ \varphi_k \} \subset C^\infty_c(\HH^m) $ such that 
$$
\left\| \varphi_k - v_k \right\|_{L^2(B_1)} + \left\| \nabla \varphi_k - \nabla v_k \right\|_{L^2(\HH^n)} \le \frac{1}{k+1} \qquad \forall k \in \N \, ,
$$
so that \eqref{all-radii} also holds with $ \{ \varphi_k \} $ replacing $ \{ v_k \} $, whence $ v \in \hdot $.    
\end{proof}

\begin{proof}[Proof of Lemma \ref{approx-above}]
From the definition of $ \hdot $, we know that it is possible to pick a sequence $ \{ \varphi_k \} \subset C^\infty_c(\HH^n) $ such that 
    \begin{equation}\label{e-phi}
\lim_{k \to \infty} \left\| \nabla \varphi_k - \nabla  v \right\|_{L^2(\HH^n)} = \lim_{k \to \infty} \left\| \varphi_k - v \right\|_{L^2(B_1)} = 0 \, .
    \end{equation}
    For each $ k \in \N$, let us set
    $$ v_k = \varphi_k^+ \wedge v^+ - \varphi_k^- \wedge v^-  , $$ 
    which clearly belongs to $ H^1_c(\HH^n) $. Since
$$
\left| v_k - v \right| \le \left| \varphi_k - v \right| ,
$$
we immediately deduce that $ v_k \to v $ in $ L^2(B_1) $.  
    In order to show the convergence of the gradients, we simply observe that
    $$
    \begin{aligned}
 \int_{\HH^n} \left| \nabla v_k - \nabla v \right|^2 dV 
 = & \, \int_{\{ v>0 \}} \left| \nabla v_k - \nabla v \right|^2 dV + \int_{\{ v<0 \}} \left| \nabla v_k - \nabla v \right|^2 dV \\ 
 = & \, \int_{\left\{ \varphi_k^+ < v  \right\} \cap \{ v>0 \} } \left| \nabla \varphi^+_k - \nabla v \right|^2 dV + \int_{\left\{ \varphi_k^- < v^-  \right\} \cap \{ v<0 \} } \left| - \nabla \varphi^-_k - \nabla v \right|^2 dV \\ 
  \le & \, \int_{\left\{ \varphi_k < 0  \right\} \cap \{ v>0 \} } \left|  \nabla v \right|^2 dV + \int_{\left\{ \varphi_k > 0  \right\} \cap \{ v<0 \} } \left|  \nabla v \right|^2 dV + \int_{\HH^n} \left| \nabla \varphi_k - \nabla v \right|^2 dV \, .
  \end{aligned}
     $$
     The rightmost integral in the last line clearly vanishes as $ k \to \infty $ due to \eqref{e-phi}; moreover, since $ \{ \varphi_k \} $ converges pointwise almost everywhere to $ v $ (up to subsequences -- recall that $ L^2 $ convergence actually takes place in every $B_R$ due to Lemma \ref{equiv}) and $ |\nabla v|^2 \in L^1(\HH^n)  $, we have  
     $$
\lim_{k \to \infty} \left( \int_{\left\{ \varphi_k < 0  \right\} \cap \{ v>0 \} } \left|  \nabla v \right|^2 dV + \int_{\left\{ \varphi_k > 0  \right\} \cap \{ v<0 \} } \left|  \nabla v \right|^2 dV \right)  =0 \, ,
     $$
which completes the proof.    

In order to prove the last statement, for each $R>0$ let us write
\begin{equation}\label{eq-split}
v = \varphi_R \, v + \underbrace{\left(  1 - \varphi_R \right) v}_{=:v_R} \, ,
\end{equation}
where $ \varphi_R $ is a usual smooth cutoff function such that $ \varphi_R = 1 $ in $ B_R $, $ \varphi_R = 0 $ in $ B_{2R} $, and $ 0 \le \varphi_R \le 1 $ everywhere. It is plain that also $ v_R \in \hdot  $, so that by applying the first statement to this function we can find a sequence $ \left\{ v_{R,k} \right\}_k \in H^1_c(\HH^n) $ such that 
 \begin{equation*}\label{approx-conv-R}
  - v^-_R \le v_{R,k} \le v^+_R \quad \forall k \in \N \qquad \text{and} \qquad 
     \lim_{k \to \infty} \left\| v_{R,k} - v_R \right\|_{\hnorm} = 0 \, .
 \end{equation*}
In particular, from the definition of $v_R$, it follows that
$$
v_{R,k} = 0 \qquad \text{in } B_R \, , \text{ for all } k \in \N \, .
$$
Hence, recalling \eqref{eq-split}, for each $R>0$ we have that the sequence
$$
\hat{v}_{R,k} := \varphi_R \, v + v_{R,k}
$$
still satisfies 
 \begin{equation*}\label{approx-conv-R-bis}
  - v^- \le \hat v_{R,k} \le v^+ \quad \forall k \in \N \qquad \text{and} \qquad 
     \lim_{k \to \infty} \left\| \hat v_{R,k} - v \right\|_{\hnorm} = 0 \, ,
 \end{equation*}
 and, in addition, 
$$
\hat v_{R,k} = v \qquad \text{in } B_R \, , \text{ for all } k \in \N \, .
$$
The thesis then follows by means of a standard diagonal procedure over $(R,k)$.
\end{proof}

\begin{proof}[Proof of Proposition \ref{density-bochner}]  
First of all, from the definition of $ L^2\big((0,T) ; \hdot \big) $ we know that $u(t) \equiv u(\cdot,t) $ belongs to $ \hdot $ for almost every $ t \in (0,T) $, which is a separable space. Let us show that $u$ is weakly measurable; to this aim, because $ \hdot $ is Hilbert, it is enough to check that for every $ v \in \hdot $ the real function
$$
(0,T) \ni t \mapsto \int_{B_1} u(x,t) \, v(x) \, dV(x) + \int_{\HH^n} \left\langle  \nabla u(x,t) , \nabla v(x) \right\rangle dV(x) 
$$
is measurable. This is a direct consequence of the fact that both $ u(x,t) $ and $  \nabla u (x,t) $ are jointly measurable functions along with Fubini's theorem, since 
$$
\int_0^T \int_{B_1} \left| u \, v \right| dV dt + \int_0^T \int_{\HH^n} \left| \left\langle  \nabla u , \nabla v \right\rangle \right| dV dt \le \sqrt T \left\| u \right\|_{L^2\left( (0,T) ; \hnorm \right)}  \left\| v \right\|_{\hnorm}  < + \infty \, .
$$
As a result, we can infer that \eqref{v-norm} holds and thus there exists a sequence of simple functions
$$
s_k(x,t) = \sum_{i=0}^{N_k} \chi_{A_i}(t) \, v_i(x) \, , \qquad N_k \in \N \, , \qquad \{ A_i \}_{i=0}^{N_k} \subset \mathcal{L}((0,T)) \, , \qquad \{v_i\}_{i=0}^{N_k} \subset \hdot \, ,
$$
where $ \mathcal{L} $ stands for the Lebesgue $\sigma$-algebra, such that 
\begin{equation}\label{approx-Bochner}
\lim_{k \to \infty} \int_0^T \left\| u(\cdot,t) - s_k(\cdot,t) \right\|_{\hnorm}^2 dt    = 0 \, .
\end{equation}
Because $ C^\infty_c(\HH^n) $ is dense in $ \hdot $, it is apparent that \eqref{approx-Bochner} is still true if, in the definition of each $s_k$, we replace $ \{ v_i \} $ with $ \{ \varphi_i \} \subset C^\infty_c(\HH^n) $ suitably chosen. Similarly, since $ C_c^\infty((0,T))$ is dense in $ L^1((0,T)) $, there is no loss of generality if we replace $ \{ \chi_{A_i} \} $ with $ \{ \eta_i \} \subset C_c^\infty((0,T)) $ suitably chosen. Therefore, property \eqref{density-bochner-test} is established via a sequence of the type
$$
\xi_k(x,t) =  \sum_{i=0}^{N_k} \eta_i(t) \, \varphi_i(x) \, .
$$
Finally, a truncation argument analogous to that of Lemma \ref{approx-above} ensures that the sequence
$$
\hat \xi_k(x,t) =  \sum_{i=0}^{N_k} \left[ \eta_i^+(t) \wedge 1 \right] \left[ (-M) \vee \varphi_i(x) \wedge M \right] 
$$
also converges to $u$ in $ L^2\big((0,T);\hdot \big) $ and is such that $ \big|\hat \xi_k\big| \le M $. The drawback is that these functions are only Lipschitz; however, by means of a further standard regularization argument they can in turn be approximated in $ L^2\big((0,T);\hdot \big) $ by smooth and compactly supported functions whose moduli do not exceed the value $ M $ (in $\R^n$ it is a matter of pure convolution).
\end{proof}

\begin{proof}[Proof of Lemma \ref{approx-rearrange-below}]
From the definition of Schwarz rearrangement introduced in \eqref{from-ast-to-star}, the thesis is equivalent to proving that
$$
\lim_{k \to \infty} f_k^\ast(s) = f^\ast(s) \qquad \forall s \ge 0 \, .
$$
On the one hand, since $ \left| f_k \right| \le \left| f \right| $, we clearly have $ f_k^\ast \le f^\ast $ for all $ k \in \N $, thus the result follows provided we can show that
\begin{equation}\label{f-star-below}
\liminf_{k \to \infty} f_k^\ast(s) \ge f^\ast(s) \qquad \forall s \ge 0 \, .
\end{equation}
To this end, first of all, we claim that
\begin{equation}\label{claim-muk}
\lim_{k \to \infty} \mu_{f_k}(t) = \mu_{f}(t) \qquad \forall t>0 \, .
\end{equation}
Again, since $ \mu_{f_k}(t) \le \mu_{f}(t)  $, we only need to establish the opposite inequality
\begin{equation}\label{claim-muk-inf}
\liminf_{k \to \infty} \mu_{f_k}(t) \ge \mu_{f}(t) \qquad \forall t>0 \, .
\end{equation}
Because each (positive) upper level set of $f$ has finite measure, for every $ t,\epsilon>0 $ there exists $ R_{t,\epsilon}>0 $ large enough such that
$$
V\!\left( \left\{ x \in \HH^n : \ |f(x)|>t+\epsilon \right\} \cap B_{R_{t,\epsilon}} \right)  \ge \mu_f(t+\epsilon) - \epsilon \, .
$$
On the other hand, in $ B_{R_{t,\epsilon}} $ pointwise convergence implies convergence in measure, therefore
$$
\liminf_{k \to \infty} V\!\left( \left\{ x \in \HH^n : \ |f_k(x)|>t \right\} \cap B_{R_{t,\epsilon}} \right) \ge V\!\left( \left\{ x \in \HH^n : \ |f(x)|>t+\epsilon \right\} \cap B_{R_{t,\epsilon}} \right) .
$$
As a result,
$$
\liminf_{k \to \infty} \mu_{f_k}(t)  \ge \liminf_{k \to \infty} V\!\left( \left\{ x \in \HH^n : \ |f_k(x)|>t \right\} \cap B_{R_{t,\epsilon}} \right) \ge \mu_f(t+\epsilon)-\epsilon \, ,
$$
so that \eqref{claim-muk-inf}, and thus \eqref{claim-muk}, follows thanks to the arbitrariness of $\epsilon$ and the fact that $ \mu_f $ is right continuous. In order to deduce \eqref{f-star-below}, given any $ s \ge 0 $ we can assume with no loss of generality that $ f^\ast(s) > 0 $. Since \eqref{claim-muk} holds and $f^*$ is right continuous, for any $ t_s>0 $ such that $ t_s<f^\ast(s) $ we have 
$$
\mu_{f_k}(t_s) > s
$$
eventually in $k$, which implies
\begin{equation}\label{final-lim-s}
f_{k}^\ast(s) > t_s \, .
\end{equation}
Letting first $ k \to \infty $ and then $ t_s \to f^\ast(s)^- $ in \eqref{final-lim-s}, we finally obtain \eqref{f-star-below}.
\end{proof}

\begin{proof}[Proof of Lemma \ref{lem-cpt-h1}]
The fact that $ v^\star $ is also compactly supported is a direct consequence of its definition: if $ v $ has compact support in $ \HH^n $ then there exists $ R>0 $ so large that
$$
\left\{ x \in \HH^n : \ |v(x)|>0  \right\} \subset B_R \, ,
$$
hence 
$$
V\!\left(  \left\{ x \in \HH^n : \ |v(x)|>0  \right\}  \right) < V(B_R) \, ,
$$
thus the support of $ v^\star $ is also contained in $B_R$. Now, let us show that $ v^\star \in H^1_c(\HH^n) $ (without assuming the validity of the P\'olya-Szeg\H{o} inequality -- otherwise the thesis would immediately follow from \eqref{P-Z}). To this aim, we construct the following function in $ \mathbb{R}^n $, written in the corresponding polar coordinates:
\begin{equation}\label{Lip}
w(x) \equiv w(\rho , \theta) := v\!\left(G^{-1}\!\left( \tfrac{\rho^n}{n} \right) , \theta \right) \qquad \forall (\rho,\theta) \equiv x \in \mathbb{R}^n \, ,
\end{equation}
where $G$ has been introduced in \eqref{eq-G}. It is straightforward to check that, by construction,
\begin{equation*}\label{lev-eq}
\mathcal{L}^n\!\left( \left\{  x \in \R^n : \ |w(x)|>t \right\}  \right)=  V\!\left( \left\{  x \in \HH^n : \ |v(x)|>t \right\}  \right)  \qquad \forall t > 0 \, ,
\end{equation*}
where $\mathcal{L}^{n} $ stands for the $n$-dimensional Lebesgue measure in $ \R^n $. In particular, we infer that $ w $ is supported in the Euclidean ball of radius $ \left[ n G(R) \right]^{1 / n} $. Thanks to \eqref{prop-psi}, the function 
$$
\rho \mapsto G^{-1}\!\left( \tfrac{\rho^n}{n} \right) 
$$
turns out to be $ C^1([0,+\infty)) $ with positive derivative. Therefore, we can assert that $ w \in H^1_c(\R^n) $; because the P\'olya-Szeg\H{o} inequality holds in $ \R^n $, its Euclidean Schwarz rearrangement $ w^\star $ also belongs to $ H^1_c(\R^n)  $. We then observe that the following identity holds:
\begin{equation}\label{Lip-2}
v^\star(r) = w^\star\!\left( \left[ n \, G(r) \right]^{\frac 1 n} \right) \qquad \forall r \ge 0 \, ,
\end{equation}
which is still a consequence of the definition of $ G $. Clearly, the function 
\begin{equation}\label{Lip-3}
r \mapsto \left[ n \, G(r) \right]^{\frac 1 n} 
\end{equation}
is also $ C^1([0,+\infty)) $, whence $ v^\star \in H^1_c(\HH^n) $ as desired. Finally, if $ v $ is Lipschitz regular so is $ w $ thanks to \eqref{Lip}, and by well-known results (see \emph{e.g.}~\cite[Theorem 2.3.2]{Kes}) we have that $ w^\star $ is also Lipschitz regular; hence, the same property is enjoyed by $ v^\star $ owing to \eqref{Lip-2} along with the just observed $ C^1 $ regularity of \eqref{Lip-3}.
\end{proof}

\begin{proof}[Proof of Proposition \ref{polya-compact-to-h1}]
    Given any $ v \in \hdot \cap \mathcal{L}_{0}(\HH^{n}) $, let $ \{ v_k \} $ be the approximating sequence provided by Lemma \ref{approx-above}, which is made up of bounded functions by construction. In particular, up to subsequences that we will not relabel, we have  $ -v^- \le v_k \le v^+ $ and $ v_k \to v $ pointwise almost everywhere as $k \to \infty$ (simple consequence of Lemma \ref{equiv}). Hence, we are in a position to apply Lemma \ref{approx-rearrange-below}, which guarantees that 
    \begin{equation}\label{P-Z-ptwse}
   \lim_{k \to \infty} v_k^\star(x)  = v^\star(x) \qquad \forall x \in \HH^n \, .
    \end{equation}
Now, let us show that $ \left\{ v_k^\star \right\} $ converges weakly in $ \hdot $. Since each $ v_k $ belongs to $ H^1_c(\HH^n) $, we also have  $ |v_k| \in H^1_c(\HH^n) $, thus from the assumed P\'olya-Szeg\H{o} inequality in $ H^1_c(\HH^n) $ it holds
\begin{equation}\label{P-Z-k}
      \int_{\HH^n} \left| \nabla \! \left|v_k\right|^\star \right|^2 dV  \le \int_{\HH^n} \left| \nabla \! \left| v_k \right| \right|^2 dV \, .
    \end{equation}
On the other hand, it is plain that $ \left| \nabla |v_k| \right| = \left| \nabla v_k \right| $ almost everywhere, and by the definition of Schwarz rearrangement $ |v_k|^\star = v_k^\star $, so that \eqref{P-Z-k} is actually equivalent to 
\begin{equation}\label{P-Z-k-unsigned}
      \int_{\HH^n} \left| \nabla v_k^\star \right|^2 dV  \le \int_{\HH^n} \left| \nabla v_k \right|^2 dV \, .
    \end{equation}
    In particular, the sequence $ \left\{ \nabla v^\star_k \right\} $ is bounded in $ L^2(\HH^n) $, and from Lemma \ref{lem-cpt-h1} we know that $ \left\{ v^\star_k \right\} \subset H^1_c(\HH^n) $; owing to Lemma \ref{equiv}, such a sequence is also bounded in $ L^2_{\mathrm{loc}}(\HH^n) $.
    Hence, we can assert that up to another subsequence $ \left\{ v_k^\star \right\} $ converges weakly in $ \hdot $ to some function $w$, which necessarily coincides with $ v^\star $ due to \eqref{P-Z-ptwse}. We have thus established that $ v^\star $ belongs to $ \hdot $, and the validity of \eqref{P-Z} just follows from the strong $ L^2(\HH^n) $ convergence of $ \left\{ \nabla v_k \right\} $ to $ \nabla v $ along with lower semicontinuity on the left-hand side of \eqref{P-Z-k-unsigned}.
\end{proof}

\begin{proof}[Proof of Proposition \ref{all-grad}]
    To begin with, for each $ t>0 $ we set
    $$
v_t := \left[ \left( \tfrac 1 t \right) \wedge v^+-t \right]^+ - \left[ \left( \tfrac 1 t \right) \wedge v^--t \right]^+ .
    $$
It is readily seen that $ -v^-\le v_t \le v^+ $ and
\begin{equation}\label{aeconv}
 \lim_{t \to 0^+} v_t(x)  = v(x) \qquad \forall x \in \HH^n \, .
\end{equation}
Moreover, since
$$
\left| v_t \right| \le  \frac{1}{t} \, \chi_{\left\{ |v| > t \right\}} \, ,
$$
we have $ v_t \in L^p(\HH^n) $ for all $ p \in [1,\infty] $, recalling that the property $ v \in \mathcal{L}_{0}(\HH^{n}) $ implies $ V\!\left( \left\{ |v|>t \right\} \right) < +\infty$. On the other hand, a straightforward computation entails
\begin{equation*}\label{aeconv-grad}
\nabla v_t = \chi_{ \left\{ t < |v| < \frac 1 t \right\} } \, \nabla v \, ,
\end{equation*}
whence, by dominated convergence,
\begin{equation}\label{aeconv-grad-L2}
\lim_{t \to 0^+} \left\| \nabla v_t - \nabla v \right\|_{L^2(\HH^n)} = 0 \, .
\end{equation}
Next, let us show that each $ v_t $ belongs to $ \hdot $. To this end, we pick a standard sequence of smooth cutoff functions $ \left\{ \varphi_R \right\}_{R>1} $ such that
$$
\varphi_R = 1 \quad \text{in } B_R \, , \qquad \varphi_R = 0 \quad \text{in } B_{2R}^c \, , \qquad 0 \le \varphi_R \le 1 \, , \qquad \left| \nabla \varphi_R \right| \le \frac{C}{R} \, ,
$$
for a suitable constant $C>0$ independent of $ R $. It is plain that $ \left\{ v_t \, \varphi_R \right\}_{R>1} \subset H^1_c(\HH^n) $ and 
$$
\lim_{R \to +\infty}  \left\| v_t \, \varphi_R - v_t\right\|_{L^2(\HH^n)} = 0 \, .
$$
Furthermore, since
$$
\left| \nabla\!\left( v_t \, \varphi_R \right)  - \nabla v_t \right|^2 \le 2 \left( 1-\varphi_R \right)^2 \left| \nabla v_t \right|^2 +\frac{2C^2}{R^2} \, v_t^2 \, ,
$$
upon integrating over $ \HH^n $ and exploiting again dominated convergence, 
$$
\lim_{R \to +\infty}  \left\|  \nabla\!\left( v_t \, \varphi_R \right)  - \nabla v_t\right\|_{L^2(\HH^n)} = 0 \, ,
$$
so that $ v_t \in \hdot $. Therefore, in view of \eqref{aeconv-grad-L2} and Proposition \ref{lemma-norm}, in order to conclude it is enough to prove that  
\begin{equation}\label{eee}
\lim_{t \to 0^+} \left\| v_t -v \right\|_{L^2(B_1)} = 0 \, .
\end{equation}
Owing to \eqref{aeconv}, $|v_t| \le |v| $, and $ v \in L^1_{\mathrm{loc}}(\HH^n) $, we easily infer that $ v_t \to v $ in $ L^1(B_1) $, in particular $ \overline{\left(v_t\right)}_1  \to \overline{v}_1 $ as $ t \to 0^+ $. Hence, in light of the local Poincar\'e inequality \eqref{lp-ineq}, it follows that
$$
\begin{aligned}
 \left\| v_t -v \right\|_{L^2(B_1)} \le & \left\| \overline{\left(v_t\right)}_1  - \overline{v}_1 \right\|_{L^2(B_1)} +  \left\| (v_t -v) - \overline{\left( v_t - v \right)}_1  \right\|_{L^2(B_1)} \\
 \le & \left| \overline{\left(v_t\right)}_1  - \overline{v}_1 \right| \sqrt{V(B_1)} + C_1 \left\| \nabla v_t - \nabla v \right\|_{L^2(\HH^n)} , 
 \end{aligned}
$$
whence \eqref{eee} is established upon letting $ t \to 0^+ $, again exploiting \eqref{aeconv-grad-L2}. Rigorously, to apply the above inequality we need to know that $ v \in L^2(B_1) $; however, a completely analogous argument shows that $ \{ v_t \}_{t > 0} $ is Cauchy in $  L^2(B_1) $ as $t \to 0^+$.
\end{proof}

\subsection{Proofs of Section \ref{mainconcenttheo}}\label{aux-3-bis}

\begin{proof}[Proof of Proposition \ref{iso-polya}]
The proof is by now rather classical and similar to the Euclidean case, but for the reader's convenience we provide some details. First of all, let us take an arbitrary nonnegative and Lipschitz function $v\in H^1_{c}(\HH^n)$, so that, by virtue of Lemma \ref{lem-cpt-h1}, we find that its Schwarz rearrangement $v^{\star}$ is also Lipschitz and belongs to $ H^1_{c}(\HH^n)$. 
Due to the Cauchy-Schwarz inequality and the definition of $ \mu_v $ in \eqref{distr-fun}, it is not difficult to deduce that 
\begin{equation}\label{ms-1}
-\frac{d}{dt}\int_{\{v>t\}} |\nabla v|\,d V\leq \left| \mu_{v}^{\prime}(t)\right|^{\frac 1 2}\left|\frac{d}{dt}\int_{\{v>t\}}|\nabla v|^{2}\,d V\right|^{\frac 1 2}  
\qquad \text{for a.e. } t>0 \, .
\end{equation}
On the other hand, by reasoning exactly as in \cite[Proof of Theorem 1.1]{MSo}, from the coarea formula \cite[Proposition 2.3]{MSo} we find that
\[
-\frac{d}{dt}\int_{\{v>t\}}|\nabla v| \, dV
=\mathcal{V}_{n-1}\!\left(\left\{x\in \HH^{n} :\ v(x)=t \right\}\right)
\qquad \text{for a.e. } t>0 
\]
and
\[
-\frac{d}{dt}\int_{\{v>t\}}|\nabla v|^{2}\,d V
=\int_{\{v=t\}}|\nabla v|\,d\mathcal{V}_{n-1} \qquad \text{for a.e. } t>0 \, ;
\]
hence, by combining these two identities and \eqref{ms-1}, we end up with 
\begin{equation}\label{isoimpPol1}
\mathcal{V}_{n-1}\!\left(\left\{x\in \HH^{n} :\ v(x)=t \right\}\right) \leq \left| \mu_{v}^{\prime}(t)\right|^{\frac 1 2}\left(\int_{\{v=t\}}|\nabla v|\,d\mathcal{V}_{n-1} \right)^{\frac 1 2}   \qquad \text{for a.e. } t>0 \, .
\end{equation}
Thanks to the assumed centered isoperimetric inequality \eqref{iso-1}, along with the fact that the De Giorgi perimeter is less than or equal to the Hausdorff measure of the topological boundary (see again \cite[Proof of Theorem 1.1]{MSo}), we have 
\[
\begin{gathered}
\mathcal{V}_{n-1}\left(\left\{x\in \HH^{n}: \ v(x)=t\right\}\right) \ge \mathrm{Per}\!\left( \left\{x\in \HH^{n}: \ v(x)>t\right\} \right)
\geq \omega_{n}\left[
\psi\!\left(
G^{-1}\!\left(\tfrac{\mu_{v}(t)}{\omega_{n}}
\right)
\right)
\right]^{n-1} \qquad
\text{for a.e. } t>0 \, ,
\end{gathered}
\]
therefore \eqref{isoimpPol1} entails
\[
\int_{\{v=t\}} |\nabla v|\, d\mathcal{V}_{n-1}
\geq \left| \mu_{v}^{\prime}(t) \right|^{-1} 
\omega_{n}^2 \left[
\psi\!\left(
G^{-1}\!\left(\tfrac{\mu_{v}(t)}{\omega_{n}}
\right)
\right)
\right]^{2(n-1)}  \qquad \text{for a.e. } t>0 \, .
\]
Note that, as observed in \cite[Proof of Theorem 1.1]{MSo}, the pointwise derivative $ \mu'_v(t) $ is negative and finite for almost every $t \in (0,\| v \|_\infty)$, whereas for values outside this range the gradient integral in the above inequality is clearly $0$. Then, still using the coarea formula, we obtain 
\begin{equation}\label{isoimpPol2}
\begin{aligned}
\int_{\HH^{n}}|\nabla v|^{2} \, d V
= \int_{0}^{+\infty} \left(
\int_{\{v=t\}}|\nabla v|\,d\mathcal{V}_{n-1} \right) dt \geq
\int_{0}^{+\infty}
\omega_{n}^2 \left[
\psi\!\left(
G^{-1}\!\left(\tfrac{\mu_{v}(t)}{\omega_{n}}
\right)
\right)
\right]^{2(n-1)} \left| \mu_{v}^{\prime}(t) \right|^{-1}  dt \, .
\end{aligned}
\end{equation}
Concerning the Schwarz rearrangement $v^{\star}$, from \eqref{rel-v-ast-star} we find 
\begin{equation}\label{grad-const}
|\nabla v^\star(x)| = \left|(v^{\star})^{\prime}(r)\right| = \omega_{n} \, \psi(r)^{n-1} \left| (v^{\ast})^{\prime}(V(B_{r})) \right| \qquad \text{for a.e. } r \equiv \mathrm{d}(x,o) \, .
\end{equation}
Furthermore, by the very definition of $v^\star$ we know that the upper level sets $\left\{v^{\star}>t\right\}$ are centered geodesic balls $B_r$ of volume $\mu_{v}(t)$, so that the radius $r$ is provided by
\begin{equation}\label{r-t}
r=G^{-1}\!\left(\tfrac{\mu_{v}(t)}{\omega_{n}}\right) ;
\end{equation}
hence, in this case the isoperimetric inequality holds as an \emph{identity}:  
\begin{equation}\label{isoimpPol3}
\mathcal{V}_{n-1}\!\left(\left\{x\in \HH^{n}: \ v^{\star}(x)=t\right\}\right)
= \omega_{n}\left[
\psi\!\left(
G^{-1}\!\left(\tfrac{\mu_{v}(t)}{\omega_{n}}
\right)
\right)
\right]^{n-1} \qquad \text{for a.e. } t>0 \, .
\end{equation} 
Therefore, from \eqref{grad-const} (in particular the fact that $|\nabla v|$ is constant on the boundary of geodesic balls), it follows that 
\[
\begin{gathered}
\int_{\{ v^{\star}=t \}}|\nabla v^{\star}|\,d\mathcal V_{n-1}
= \omega_{n} \, \psi(r)^{n-1} \left| (v^{\ast})^{\prime}(V(B_{r})) \right| 
\mathcal{V}_{n-1}\!\left(\left\{x\in \HH^{n}: \ v^{\star}(x)=t\right\}\right) \qquad \text{for a.e. } t>0 \, ,
\end{gathered}
\]
where $r$ is as in \eqref{r-t}. As a result, by virtue of  \eqref{isoimpPol3} and the relation \eqref{key-der} between $ (v^\ast)' $ and $ \mu_v' $, we end up with 
$$
\begin{gathered}
\int_{\{v^{\star}=t\}} |\nabla v^{\star}|\,d\mathcal V_{n-1}
 = \omega_{n}^2 \left[
\psi\!\left(
G^{-1}\!\left(\tfrac{\mu_{v}(t)}{\omega_{n}}
\right)
\right)
\right]^{2(n-1)}
\left| (v^*)'(\mu_v(t)) \right|
=
\omega_{n}^2 \left[
\psi\!\left(
G^{-1}\!\left(\tfrac{\mu_{v}(t)}{\omega_{n}}
\right)
\right)
\right]^{2(n-1)}
\left| \mu_{v}^{\prime}(t) \right|^{-1}  \\ \text{for a.e. } t>0 \, ,
\end{gathered}
$$
whence 
\begin{equation*}
\int_{\HH^{n}}|\nabla v^{\star}|^{2}\,dV 
 =\int_{0}^{+\infty} \left(
\int_{\{v=t\}}|\nabla v^\star|\,d\mathcal{V}_{n-1} \right) dt =
\int_{0}^{+\infty}
\omega_{n}^2 \left[
\psi\!\left(
G^{-1}\!\left(\tfrac{\mu_{v}(t)}{\omega_{n}}
\right)
\right)
\right]^{2(n-1)}
\left| \mu_{v}^{\prime}(t) \right|^{-1} 
dt \, ,
\end{equation*}
so that inequality \eqref{P-Z} can be derived from \eqref{isoimpPol2}. The general case $ v \in \hdot \cap \mathcal{L}_{0}(\HH^{n}) $ follows from a density argument owing to Proposition \ref{polya-compact-to-h1} and the standard fact that Lipschitz functions are dense in $ H^1_c(\HH^n) $. 
\end{proof}

\subsection{Proofs of Section \ref{loc-elliptic}}\label{aux-2} 

\begin{proof}[Proof of Proposition \ref{weak-sol-min}]
First of all, we notice that, in view of the assumptions, the functional 
$$
\mathcal{F}(v) := \frac 1 2 \int_{B_R}  \left| \nabla v \right|^2  dV +\int_{B_R}  \mathsf{B}(v) \,  dV - \int_{B_R}  f \, v \,  dV
$$
is well defined and strictly convex in $ H^1_0(B_R)  $, since the only unsigned term can simply be estimated by
\begin{equation*}\label{cnt-emb}
\left| \int_{B_R}  f \, v \,  dV \right| \le  \left\| f \right\|_{L^\infty\left( B_R \right)}  \left\| v \right\|_{L^1\left( B_R \right)}  \le C_{2,R} \, V(B_R)^{\frac 1 2} \left\| f \right\|_{L^\infty\left( B_R \right)} \left\| \nabla v \right\|_{L^2\left( B_R \right)} , 
\end{equation*}
where $ C_{2,R}>0 $ is the constant appearing in \eqref{loc-sob}. Then, a routine application of Young's inequality ensures that the sublevels of $ \mathcal{F} $ are bounded in $ H^1_0(B_R) $. Moreover, it is immediate to see that $  \mathcal{F}$ is weakly lower semicontinuous. Therefore, as a standard application of the direct method in the calculus of variations, we have that $ \mathcal{F} $ admits a minimizer $v_0$, which is unique due to strict convexity. Since 
$$
 \int_{B_R}  \left| \nabla \left| v_0 \right|  \right|^2  dV = \int_{B_R}  \left| \nabla v_0 \right|^2  dV \, ,  \qquad \int_{B_R}  \mathsf B(|v_0|) \,  dV = \int_{B_R}  \mathsf B(v_0) \,  dV
$$
and 
$$
- \int_{B_R}  f \left| v_0 \right|  dV  \le - \int_{B_R}  f \, v_0 \, dV \, ,
$$
we necessarily have $ v_0 \ge 0 $. Our next goal is to show that $ v_0 \in L^\infty \!\left( B_R \right) $. To this end, for each $  \kappa >0 $ let us set $ v_\kappa :=  v_0 \wedge \kappa \in H^1_0(B_R) $ and observe that, due to the minimality property of $ v_0 $, 
$$
\mathcal{F}(v_\kappa) \ge  \mathcal{F}(v_0) \, ,
$$
that is
$$
 \frac 1 2 \int_{B_R} \left( \left| \nabla v_\kappa \right|^2 - \left| \nabla v_0 \right|^2 \right) dV +\int_{B_R} \left[  \mathsf{B}(v_\kappa) - \mathsf{B}(v_0) \right]  dV + \int_{B_R}  f \left( v_0 - v_\kappa \right)  dV \ge 0 \, .
$$ 
Upon noticing that $ \mathsf{B}(v_\kappa) \le \mathsf{B}(v_0) $ and $ \nabla v_\kappa = \chi_{A_\kappa^c} \nabla v_0 $, where $ A_\kappa := \{ v_0 > \kappa \} $, the above inequality implies in turn 
\begin{equation}\label{trunc-1}
 \frac 1 2 \int_{A_\kappa}  \left| \nabla v_0  \right|^2 dV  = \frac 1 2 \int_{A_\kappa}  \left| \nabla \!\left( v_0 - v_\kappa \right) \right|^2 dV  \le \left\| f \right\|_{L^\infty\left( B_R \right)} \left\| v_0 - v_\kappa \right\|_{L^1\left( A_\kappa \right)}  .
\end{equation}
On the other hand, given any $ p>2 $ as in the statement, from the local $p$-Sobolev inequality \eqref{loc-sob} and the fact that $ (v_0-v_\kappa) $ is supported in $ A_\kappa $, we have
$$
\left\| v_0 - v_\kappa \right\|_{L^1\left( A_\kappa \right)} \le V(A_\kappa)^{\frac {p-1}{p}} \left\| v_0 - v_\kappa \right\|_{L^p\left( A_\kappa \right)} \le C_{p,R} \, V(A_\kappa)^{\frac {p-1}{p}} \left\| \nabla \! \left(v_0-v_\kappa\right) \right\|_{L^2\left( A_\kappa \right)} ,
$$
which, along with \eqref{trunc-1}, yields
$$
\left\| v_0 - v_\kappa \right\|_{L^1\left( A_\kappa \right)} \le \sqrt 2  \, C_{p,R} \, V( A_\kappa )^{\frac {p-1}{p} } \left\| f \right\|_{L^\infty\left( B_R \right)}^{\frac 1 2} \left\| v_0 - v_\kappa \right\|_{L^1\left( A_\kappa \right)}^{\frac 1 2}  ,
$$
that is 
\begin{equation}\label{trunc-2}
\left\| v_0 - v_\kappa \right\|_{L^1\left( A_\kappa \right)} \le 2 \, C_{p,R}^2  \left\| f \right\|_{L^\infty\left( B_R \right)} V( A_\kappa) ^{\frac{2(p-1)}{p} } \, .
\end{equation}
From the definition of $ A_\kappa $ and \eqref{trunc-2}, it is plain that for all $ \kappa'>\kappa $
\begin{equation*}\label{trunc-3}
\left( \kappa' - \kappa \right) V( A_{\kappa'} ) \le  2 \, C_{p,R}^2  \left\| f \right\|_{L^\infty\left( B_R \right)} V( A_\kappa) ^{\frac{2(p-1)}{p} } \, .
\end{equation*}
If we set $ \kappa_j := c \left( 2-2^{-j} \right) $, for a fixed $ c>0 $ to be chosen and each $ j \in \mathbb{N} $, and apply \eqref{trunc-2} with $ \kappa' = \kappa_{j+1} $ and $  \kappa = \kappa_{j}  $, we end up with the recurrence relation
$$
V\!\left( A_{\kappa_{j+1}} \right) \le \frac{2^{j+2}}{c} \, C_{p,R}^2  \left\| f \right\|_{L^\infty\left( B_R \right)} V\!\left( A_{\kappa_j} \right)^{\frac{2(p-1)}{p} } .
$$
Since $ 2(p-1)/p > 1 $, this is a classical De Giorgi sequence, thus if $ V\!\left( A_{\kappa_0} \right) = V( A_{c} ) $ is small enough it satisfies
$$
\lim_{j \to \infty} V\!\left(A_{\kappa_{j}}\right)  = 0 \, ,
$$
which entails $ V( A_{2c} ) = 0 $, namely $ v_0 \le 2c $. Note that the right choice of $c$ (large enough) can be made quantitative in terms of $ p  , C_{p,R} , \left\| f \right\|_{L^\infty\left( B_R \right)} $, and $ \left\| v_0 \right\|_{L^1\left( B_R \right)} $ as well, with monotone dependencies. However, from \eqref{trunc-2} with $ v_\kappa = 0 $ we obtain 
$$
\left\| v_0 \right\|_{L^1\left( B_R \right)} \le 2 \, C_{p,R}^2  \left\| f \right\|_{L^\infty\left( B_R \right)} V(B_R)^{\frac{2(p-1)}{p} } \,  , 
$$
thus the latter dependence can be removed.

From the boundedness of $ v_0 $ it readily follows that $ \mathcal{F} $ is differentiable at $ v_0 $, therefore the Euler-Lagrange equation
$$
	\int_{B_R}  \left\langle \nabla v_0 , \nabla w \right\rangle  dV +\int_{B_R}  \beta(v_0) \, w \,  dV - \int_{B_R}  f \, w \,  dV  = 0  \qquad \forall w \in H^1_0(B_R) 
$$
holds, which is precisely the weak formulation of \eqref{ellip-1}. Vice versa, let $ v_1 $ be any weak solution to \eqref{ellip-1}, in the sense of Definition \ref{weak-sol}. The fact that $ v_1  \in H^1_0\!\left( B_R \right) \cap L^\infty\!\left( B_R \right) $ and the validity of \eqref{weak-sol-ell} with $ v = v_1 $ imply that $ \mathcal{F}'(v_1) $ is well defined and coincides with the null functional, whence
$$
\mathcal{F}(w) \ge \mathcal{F}(v_1) \qquad \forall w \in H^1_0(B_R) 
$$
thanks to the convexity of $ \mathcal{F} $. This means that $ v_1 $ is also a minimizer of $ \mathcal{F} $, but by strict convexity  $ v_1=v_0 $, and the thesis is proved.
\end{proof}

\begin{proof}[Proof of Corollary \ref{decr-min}] 
The P\'olya-Szeg\H{o} inequality and the definition of Schwarz rearrangement imply, in particular, that for every function $ v \in H^1_0(B_R) $ we have $ v^\star \in H^1_0(B_R) $ and  
\begin{equation}\label{polya-1}
\int_{B_R} \left| \nabla v^\star \right|^2 dV  \le  \int_{B_R} \left| \nabla v \right|^2 dV \, .
\end{equation}
Moreover, since $ 0 \le f = f^\star $ by assumption, Proposition \ref{Propconves}\eqref{HL} entails
\begin{equation}\label{polya-2}
\int_{B_R}  f  \, v \,  dV  \le  \int_{B_R}  f^\star  \, v^\star \,  dV  = \int_{B_R}  f  \, v^\star \,  dV \, .
\end{equation}
On the other hand, the function $ \mathsf{B} $ being continuous, nondecreasing on $ [0,+\infty) $, and even with $ \mathsf{B}(0)=0 $, the identity $ \mathsf{B}(v^\star) = \left[ \mathsf{B}(v) \right]^\star $ holds, therefore (recall \eqref{identity-int} and comments below)
\begin{equation}\label{polya-3}
\int_{B_R}  \mathsf{B}(v) \,  dV  = \int_{B_R}   \left[ \mathsf{B}(v) \right]^\star  dV  =  \int_{B_R}   \mathsf{B}(v^\star) \,  dV  \, . 
\end{equation}
As a result, by virtue of \eqref{polya-1}, \eqref{polya-2}, and \eqref{polya-3}, it follows that 
\begin{equation}\label{polya-4}
\mathcal{F}(v) \ge \mathcal{F}(v^\star) \, .
\end{equation}
Hence, if $v$ is the weak solution to \eqref{ellip-1}, from \eqref{polya-4} and Proposition \ref{weak-sol-min} we necessarily deduce that $ v=v^\star $. Finally, from the radial formulation of \eqref{weak-sol-ell}, it is readily seen that 
$$
-\psi(r)^{n-1} \left[v^\star\right]'(r) + \psi(s)^{n-1} \left[v^\star\right]'(s) + \int_s^r  \beta(v^\star(t)) \, \psi(t)^{n-1} \, dt-\int_s^r f(t) \, v^\star(t) \, \psi(t)^{n-1} \, dt = 0 
$$
for almost every $ s,r \in (0,R) $, which shows that $ \left[v^\star\right]' $ has a continuous version on $ [0,R] $ that vanishes at $r=0$, therefore $ v^\star \in  C^1\!\left( \overline{B}_R \right)  $. 
\end{proof}

\subsection{Proofs of Section \ref{loc-parabolic}}\label{aux-3} 

\begin{proof}[Proof of Proposition \ref{exuni-weak-loc}]
Due to the generality of the nonlinearities we aim at considering, in order to construct the claimed weak solution we resort to the theory of gradient flows in Hilbert spaces as in \cite{BreMon,BreOp}, along with the theory of m-accretive operators in Banach spaces \cite{CL} to obtain additional $L^p(B_R)$ properties.

First of all, by virtue of \cite[Theorem 23 and Corollary 31]{BreMon}, we can always construct a solution $ u_R $ to \eqref{filt-eq-local} as the gradient flow in $ H^{-1}(B_R) $ of the energy functional
$$
J(v) :=
\begin{cases}
\int_{B_R} \Phi(v) \, dV & \text{if } v \in  H^{-1}(B_R) \cap L^1(B_R) \, , \\
+ \infty & \text{if }  v \in H^{-1}(B_R) \setminus L^1(B_R) \, .
\end{cases}
$$
Since $ L^\infty(B_R) \subset H^{-1}(B_R) $, the initial data $ u_0 $ we are interested in are included in this abstract theory; moreover, the assumption in \cite{BreMon} that $ \phi $ be surjective can easily be circumvented as observed in \cite[Proof of Proposition 2.6]{GIMP}. In particular, by the statement of such results, we see that $ \phi(u_R(\cdot,t)) \in H^1_0(B_R) $ for all $ t>0 $, $ \phi(u_R) \in L^\infty_{\mathrm{loc}}\!\left((0,+\infty);H^1_0(B_R)\right) $, $ u_R \in C\!\left([0,+\infty);H^{-1}(B_R)\right) \cap AC_{\mathrm{loc}}\!\left((0,+\infty);H^{-1}(B_R)\right) \cap L^\infty_{\mathrm{loc}}\!\left((0,+\infty);L^1(B_R)\right)$, and by duality the identity
\begin{equation}\label{almost-weak}
\int_s^{T} \int_{B_R} u_R \, \partial_t \xi \, dV dt = \int_s^{T} \int_{B_R} \left\langle \nabla \phi(u_R), \nabla \xi  \right\rangle dV dt - \int_{B_R} u_R(x,s) \, \xi(x,s) \, dV(x) 
\end{equation}
holds for every  $ \xi \in C^1\!\left( \overline{B}_R \times [0,T] \right) $ such that $ \xi = 0 $ on $ \left(\partial B_R \times (0,T \right)) \cup \left( B_R \times \{ T \} \right) $, for all $ T>s>0 $. Moreover, because $ J(u_0) < +\infty $, we have that the time function $ t \mapsto J(u_R(\cdot,t)) $ is locally absolutely continuous on $ [0,+\infty) $ and satisfies (see \emph{e.g.}~\cite[Theorem 3.6]{BreOp})
\begin{equation}\label{energy-ptwse}
\frac{d}{dt} J(u_R(\cdot,t)) = - \int_{B_R} \left| \nabla \phi (u_R(x,t)) \right|^2 dV(x) \qquad \text{for a.e. } t >0 \, . 
\end{equation}
Upon integrating \eqref{energy-ptwse} from $t=0$ to $ T $, we readily obtain the energy inequality \eqref{EIR} (it is in fact an identity at this stage). Also, we are in a position to let $ s \to 0^+ $ in \eqref{almost-weak}, so that \eqref{weak-form-loc} is established (up to interpreting the left-hand side in the $ H^{-1}(B_R) -  H^1_0(B_R) $ duality sense). In order to prove that $ u_R $ is indeed a weak solution according to Definition \ref{weak-energy-sol-loc}, we are only left with showing that $ u_R $ is nonnegative and belongs to $ L^\infty(B_R \times (0,T) ) $, but both properties will be a direct consequence of the second part of the proof; for the moment we are tacitly considering the odd extension of $ \phi $ on $ \R^- $, see similar comments after \eqref{defB}. The nonincreasing monotonicity of \eqref{EIR-mon} and, in particular, its right continuity are standard consequences of the fact that $ \nabla \phi (v) $ represents the $H^{-1}(B_R)$-subgradient $ \partial J $ of $ J(v) $ when $ \phi(v) \in H^1_0(B_R) $ \cite[Theorem 17]{BreMon}, so that  \eqref{EIR-mon} is lower semicontinuous and the identity
\begin{equation}\label{weak-der-id}
 \left\| \partial_t u_R(\cdot,t) \right\|_{H^{-1}(B_R)}  = \left\| \nabla \phi(u_R(\cdot,t)) \right\|_{L^2(B_R)} \qquad \text{for a.e. } t>0 
\end{equation}
holds. For more details we refer to the proof of \cite[Proposition 3.3]{BreOp} and again \cite[Theorem 23]{BreMon}. Hence, by combining such a monotonicity with \eqref{weak-der-id} and the energy inequality \eqref{EIR}, it is straightforward to deduce \eqref{est-bre}, which also entails $ u_R \in AC_{\mathrm{loc}}\!\left([0,+\infty);H^{-1}(B_R)\right) $. 

Next, let us focus on the aforementioned $ L^p(B_R) $ properties. The key observation is that  solving the discretized equations \eqref{eq-ui-rec} is equivalent to writing
$$
w_{i+1} = \left( I + h \, \partial J \right)^{-1} w_i \qquad \forall i \in \N \, ,
$$
that is, applying recursively the resolvent operator associated with $ \partial J $ (see \cite[Chapter 1]{BreOp}). From well-known results in the theory of linear and nonlinear semigroups \cite[Corollary 4.4]{BreOp}, we have that the corresponding piecewise-constant interpolant $ u^{(h)} $ defined in \eqref{pcint} converges locally uniformly in $ H^{-1}(B_R) $ to $ u_R $ as $ h \to 0^+ $. On the other hand, recalling that $ u_0 \in L^\infty(B_R) $, we can assert that actually $ u_R \in C\!\left([0,+\infty);L^1(B_R)\right) $ and local uniform convergence also occurs in $ L^1(B_R) $. This is a consequence of the Crandall-Liggett theorem applied to \eqref{eq-ui-rec} in $ L^1(B_R) $, because for this kind of data the $H^{-1}(B_R) $ (gradient flow) and $ L^1(B_R) $ (m-accretive operator) semigroup theories give rise exactly to the same discretized elliptic equations: we refer the reader to \cite[Sections 10.2 and 10.3]{Vaz}, see in particular Corollary 10.21 there. As an immediate byproduct, we infer that $ u_R \ge 0 $ because so is every $ u^{(h)} $. Hence, at least in the case when $  \phi $ is bijective, by applying Lemma \ref{ulteriori-prop-approx}\eqref{pA} and Proposition \ref{weak-sol-comparison} with the choices
$$
\beta = \tfrac 1 h \, \phi^{-1} \, , \qquad v = \phi(w_{i+1}) \, , \qquad f= \tfrac 1 h \, w_i \, ,
$$
we obtain, for all $ p \in [1,\infty] $, the estimates 
$$
\left\| w_{i+1} \right\|_{L^p(B_R)} \le \left\| w_{i} \right\|_{L^p(B_R)} \le \ldots \le \left\| u_0 \right\|_{L^p(B_R)} \qquad \forall i \in \N \, ,
$$
whence, upon letting $ h \to 0^+ $ and exploiting the above recalled uniform $ L^1(B_R) $ convergence along with dominated convergence, we end up with
\begin{equation}\label{deecreas}
\left\| u_R(\cdot,t) \right\|_{L^p(B_R)} \le \left\| u_R(\cdot,s) \right\|_{L^p(B_R)}  \qquad \forall t>s\ge 0 \, ,
\end{equation}
at least for $ p < \infty $, the case $ p=\infty $ following as usual by taking the limit $ p \to \infty $. We have thus shown the nonincreasing monotonicity of \eqref{nn-mon}, so that continuity for $ p < \infty $ again follows from $ L^1(B_R) $ continuity and dominated convergence, whereas the right continuity of the $ L^\infty(B_R) $ norm is a straightforward consequence of \eqref{deecreas} and its weak$^\ast$ lower semicontinuity. Properties \eqref{contr-L1-p-loc} and \eqref{comp-R-eq} can be obtained with a completely analogous argument, using \eqref{L1-contr-A} and \eqref{ee-2}  on the $h$-discretized versions of $ u_R $, $v_R$, and $ u_{R+1} $, respectively, and taking limits as $ h \to 0^+ $. If $ \phi $ is not bijective, then the same properties can be 
Finally, we may apply a standard density argument to extend the validity of \eqref{almost-weak} from the above restricted class of test functions to every $\xi$ as in Definition \ref{weak-energy-sol-loc}.

With regards to the claimed uniqueness of weak solutions, we limit ourselves to pointing out that it can be established by means of an argument commonly known as Ole\u{\i}nik's trick, which is carried out in detail in the proof of Proposition \ref{exuni-weak}, where it is actually more subtle to be justified rigorously and holds under the mere assumption \eqref{cond-phi}.
\end{proof}

\subsection{Proofs of Section \ref{proof-main}}\label{aux-4} 

\begin{proof}[Proof of Lemma \ref{1-d-lemma}]
Upon integrating \eqref{dec-est} over $ (r,b] $, we obtain: 
	\begin{equation*}\label{dec-est-bis}
-v(r) +  \int_r^b \alpha(t) \int_t^b \gamma(s) \, v(s) \, ds \, dt \ge 0  \qquad  \forall r \in (a,b]  \, ,
\end{equation*}
which implies
	\begin{equation}\label{dec-est-ter}
-v(r) +  \int_r^b \left| \alpha(t) \right| dt \,  \int_r^b \left| \gamma(s) \right|  v(s) \, ds  \ge 0  \qquad  \forall r \in (a,b]  \, .
\end{equation}
Let us set
$$
A(r) :=  \int_r^b \left| \alpha(t) \right| dt \, , \qquad \Gamma(r) := \left\| \gamma \right\|_{L^\infty\left(\left(r,b\right)\right)} , \qquad W(r) := \int_r^b v(s) \, ds \,  ;
$$
then, inequality \eqref{dec-est-ter} yields in turn 
\begin{equation*}\label{dec-est-quater}
W' (r) + A(r) \, \Gamma(r)\,  W(r) \ge 0 \qquad  \forall r \in (a,b]  \, ,
\end{equation*}
which can be rewritten as  
\begin{equation}\label{dec-est-quater-q}
\left( e^{-\int_r^b  A(s) \, \Gamma(s) \, ds}  \, W(r) \right)' \ge  0 \qquad  \forall r \in (a,b] \, .
\end{equation}
Since $ W(b)=0 $ and $ W \ge 0 $, a final integration of \eqref{dec-est-quater-q} entails $ W(r) = 0 $ for all $ r \in (a,b] $, whence the thesis immediately follows.
\end{proof}

\end{document}